\documentclass[11pt,a4paper]{amsart}
\usepackage[all,2cell]{xy} \usepackage{mathrsfs}
\usepackage{MnSymbol} \usepackage{stmaryrd}
\UseAllTwocells

\def\twosquare{{\square\mkern-3.4mu\square}}

\def\squarearrow{{\square\mkern-3.6mu\lefthalfcap}}

\newcommand*{\myref}[1]{(\ref{#1})}
\newtheorem{theorem}{Theorem}[section]

\newtheorem{lemma}[theorem]{Lemma}
\newtheorem{proposition}[theorem]{Proposition}

\theoremstyle{definition}
\newtheorem{definition}[theorem]{Definition}
\newtheorem{example}[theorem]{Example}
\newtheorem{remark}[theorem]{Remark}
\newcommand{\ie}{{\em i.e.}}
\newcommand{\cf}{{\em cf.}\ }
\newcommand{\eg}{{\em e.g.}}
\newcommand{\ko}{\: , \;}
\newcommand{\vir}{\:, \;}
\newcommand{\resp}{{\em resp.}}
\newcommand*{\up}[1]{$^{\text{#1}}$}
\newcommand*{\ul}[1]{\underline{#1}}
\newcommand*{\ol}[1]{\overline{#1}}

\renewcommand*{\tilde}[1]{\widetilde{#1}}
\newdir{ >}{{}*!/-8pt/\dir{>}}
\newdir{ (}{{}*!/-6pt/\dir^{(}}

\newcommand{\A}{{\mathbb A}}
\newcommand{\D}{{\mathbb D}}
\newcommand{\E}{{\mathbb E}}
\newcommand{\F}{{\mathbb F}}
\renewcommand{\S}{\mathbb{S}}
\newcommand{\T}{\mathbb{T}}
\newcommand{\U}{\mathbb{U}}
\newcommand{\V}{\mathbb{V}}
\newcommand{\Z}{\mathbb{Z}}
\newcommand{\N}{\mathbb{N}}

\newcommand{\eps}{\varepsilon}

\newcommand{\sfa}{{\mathsf A}}
\newcommand{\sfd}{{\mathsf D}}
\newcommand{\sfe}{{\mathsf E}}
\newcommand{\sff}{{\mathsf F}}
\newcommand{\ca}{{\mathcal A}}
\newcommand{\cb}{{\mathcal B}}
\newcommand{\cc}{{\mathcal C}}
\newcommand{\cd}{{\mathcal D}}
\newcommand{\ce}{{\mathcal E}}

\newcommand{\ch}{{\mathcal H}}

\newcommand{\cm}{{\mathcal M}}
\newcommand{\cn}{{\mathcal N}}

\newcommand{\cs}{{\mathcal S}}
\newcommand{\ct}{{\mathcal T}}

\newcommand{\cw}{{\mathcal W}}

\newcommand*{\opname}[1]{\operatorname{#1}}

\newcommand{\la}{\leftarrow}
\newcommand{\lra}{\longrightarrow}

\newcommand{\eqiso}{\stackrel{_\sim}{=}}
\newcommand{\iso}{\stackrel{_\sim}{\rightarrow}}

\newcommand{\Iso}{\stackrel{_\sim}{\longrightarrow}}

\newcommand{\hodirlim}{\underrightarrow{\mathsf{holim}}}
\newcommand{\hoinvlim}{\underleftarrow{\mathsf{holim}}}
\newcommand{\id}{\mathbf{1}}
\newcommand{\il}{{i_\lefthalfcap}}
\newcommand{\ir}{{i_\righthalfcup}}

\newcommand{\ten}{\otimes}
\newcommand{\HOM}{{\mathcal H}om}
\newcommand{\dia}{{\mathsf{dia}}}
\newcommand{\f}{{\mathsf{f}}}
\newcommand{\Hom}{\mathsf{Hom}}
\newcommand{\Mor}{\mathsf{Mor}}
\newcommand{\mor}{\mathsf{mor}}
\newcommand{\ev}{\mathsf{ev}}
\newcommand{\R}{{\mathbf R}}
\renewcommand{\L}{{\mathbf L}}

\newcommand{\hocolim}{\mathsf{hocolim}}
\newcommand{\holim}{\mathsf{holim}}

\newcommand{\kernel}{\mathsf{ker}}
\newcommand{\cok}{\mathsf{cok}}
\newcommand{\Cat}{\mathscr{C}at}
\newcommand{\CAT}{\mathscr{CAT}}

\newcommand{\ADD}{\mathscr{ADD}}

\newcommand{\EXA}{\mathscr{EXA}}

\newcommand{\TRIA}{\mathscr{TRIA}}
\newcommand{\Dia}{\mathscr{D}ia}

\newcommand{\Cubes}{\mathscr{C}ubes}
\newcommand{\CA}{{\cc^{b}\ca}}
\newcommand{\HA}{{\ch^{b}\ca}}
\newcommand{\DA}{{\cd^{b}\ca}}

\title[Universal property of triangulated derivators via Keller's towers]
{Universal property of triangulated derivators via Keller's towers}

\author{Marco Porta}

\address{Marco Porta}
\email{marcoporta1@libero.it}

\begin{document}

\begin{abstract}
In \cite{Keller91} B. Keller solved the universal problem of the
extension of an exact category to its (bounded) derived category by
introducing the notions of tower of exact and triangulated categories
and proving the universal property in this setting. In this note we
show that his result extends to the corresponding universal problem
for Grothendieck's derivators.
\end{abstract}


 \date{April 19, 2017}


\maketitle

\tableofcontents

\section{Introduction}
\label{s:intro}
Triangulated categories were invented in the early sixties by
Grothendieck--Verdier \cite{Verdier96} and, simultaneously and
independently, by Dold--Puppe \cite{DoldPuppe61}.
Grothendieck--Verdier sought to axiomatize the properties of derived
categories of sheaves, while Dold--Puppe were motivated by examples
from topology, and notably the stable homotopy category of finite
$CW$-complexes.

In spite of the success of derived and triangulated categories in
recent years, during the thirty years that followed, triangulated
categories were largely considered to be too poor in structure to
allow the development of a significant general theory, analogous
to that of abelian categories established by Freyd, Mitchell,
Gabriel, \ldots .
As an example, let us consider a triangulated category $\ct$ and let
$I$ be a finite directed diagram (\ie, a small finite category with
no nontrivial loops). We can form the category $\Hom(I^\circ,\ct)$
of contravariant functors from the diagram $I$ into $\ct$. In general,
there is no canonical triangulated structure on this category
\cite{Keller91}, even in the simple case where $I$ is the category
containing only two objects, their identities, and an arrow connecting
them, the main problem being that the mapping cone is a non-functorial
construction.
Another important consequence of this fact is the following:
Let $\ce$ be an exact category and $\cd^b(\ce)$ its bounded
derived category. Then, the universal property does not hold
for the triangulated category $\cd^b(\ce)$, \ie, the natural
functor $\HOM_{ex}(\cd^b(\ce),\ct) \to \HOM_{ex}(\ce,\ct)$,
relating categories of exact functors, is not quasi-invertible
for all triangulated categories $\ct$.

It follows that, in order to have a satisfactory theory, the notion
of triangulated category must be modified or enhanced.
A long list of attempts appeared in the literature in recent
years in order to define new kinds of structures assuring the
universal property to the construction $\cd^b$, \eg,
$\opname{DG}$-categories and $\ca_\infty$-categories.
Derivators were introduced in a systematic fashion by Alexander
Grothendieck \cite{Grothendieck90} in the nineties in a long letter
addressed to Quillen and by Alex Heller in \cite{Heller88}. However,
the idea of a `derivatory notion' was already around for a few years,
and possibly others also contributed. For instance, there is the paper
\cite{Anderson79} by D. W. Anderson from the late seventies and other
papers of his in which the philosophy was already very `derivatory'.

Their advantage, compared to the other constructions, is that they
have the least amount of added structure.
I will not try to give a complete definition of what a derivator is
in this introduction. It suffices to say that it is a $2$-functor
$\D$ from the $2$-category of diagrams into the $2$-category of
categories $\cc\ca\ct$ satisfying a list of axioms inspired by
the properties of the formalism of (homotopy) Kan extensions in
the context of Quillen's definition of a model category
\cite{Quillen67} together with its homotopy category.

Similar constructions were independently introduced by Jens Franke
and Bernhard Keller in the nineties.
Keller's construction \cite{Keller91}, called `epivalent tower',
only considers categories `indexed' over hyper-cubical diagrams,
but nevertheless he was able to prove the universal property for
the epivalent towers enhancing the derived category.
Some years later, G. Maltsiniotis \cite{Maltsiniotis07}, principally
inspired by works of Grothendieck \cite{Grothendieck90} and Franke
\cite{Franke96}, gave the axioms for derivators that we follow in
this paper. He introduced a notion of triangulated derivator, \ie,
a derivator taking values in triangulated categories and satisfying
two more axioms \cite{Maltsiniotis07}. At the same time \cite{Keller07},
Keller associates a derivator $\D_\ce$ with an exact category $\ce$
by means of the following definition:
$\D_\ce(I) := \cd^b(\ul\Hom(I^\circ,\ce))$, and proves that this
derivator is triangulated.

The principal aim of this paper is to show that, similarly to the
case of `towers', the triangulated derivator associated to an exact
category $\ce$ has a universal property. This is Theorem
\ref{thm:main} in the text. As a corollary, in the particular case
of the diagram $I=e$ (the terminal category), we obtain a result
(Theorem \ref{thm:ext}) that we can consider as the correct universal
property for the (bounded) derived category $\D_\ce(e) = \cd^b(\ce)$
among the `basic' triangulated categories, \ie, categories of the
form $\T(e)$ for some triangulated derivator $\T$.
This result is similar in spirit to a string of other results about
universal properties of derivator-like structures obtained by other
authors. A. Heller proved his main results on `homotopy theories',
a notion very near to Grothendieck's notion of `d{\'e}rivateur', in
\cite{Heller88} already in the '80s.
J. Franke introduced in 1996 his notion of `systems of triangulated
diagram categories' in \cite{Franke96}. Both these authors proved
universal properties for the structures they have introduced, the
main difference being that in the former the base is given
by the homotopy category of simplicial sets of arbitrary size rather
than of finite spectra as in the latter. We summarize and report these
historical results in Theorem \ref{thm:spectra} in the text.
Finally, in 2008 D.-C. Cisinski \cite{Cisinski08} proved a universal
property for the derivator ${\mathbb{H}\mathrm{ot}}_I$ associated to
the homotopy theory of presheaves over a diagram $I$ with values in
small categories which models classical homotopy theory.

Let us briefly describe the contents of the sections of the present
article. In section \ref{s:der} we fix the notations and recall the
definitions of derivator, pointed derivator, triangulated derivator,
morphism of derivators. Then, after recalling the notion of exact
category, we introduce the associated triangulated derivator defined
by Keller in \cite{Keller07}, which is the central object of interest
in this paper.
Section \ref{s:tow} is dedicated to the description of the notion of
Keller's tower and its related morphisms. We prove our main result
about the universal property of this derivator in section \ref{s:main}.
Let us remark that one could write down a proof completely expressed
in the language of derivators, following the lines of Keller's proof.
However, in this paper we prefer to show how the theorem in the
setting of derivators naturally follows from the case of towers.

Section \ref{s:epirec} is devoted to a very useful property
of morphisms of derivators that we call, following Keller's
\cite{Keller91}, `redundancy of the connecting morphism'.
This means that a morphism of derivators which preserves
bicartesian squares (in the sense of derivators) automatically
preserves connecting morphisms, hence preserves distinguished
triangles. This fact entails that an additive morphism of
exact or triangulated derivators $F$ is automatically
triangulated provided that it locally preserves distinguished
triangles, even when it does not preserves connecting morphisms
functorially. The proof involves the formalism said `recollement
of triangulated categories' or `$6$-gluing functors' for low
dimensional diagrams.

The proof of the universal property for (bounded) derived categories
is left to section \ref{s:cor}. The techniques used in the proof
are new with respect to Keller's approach in \cite{Keller91}.
Indeed, Keller could prove his result using induction over
the integer dimensions of his diagrams since he was dealing
with hypercubes. This is not possible if we consider all
finite directed diagrams (\cf definition in Section \ref{s:der})
and we have to change our strategy.
A key ingredient in our proof is a careful analysis of the property
of epivalence (\ie, essential surjectivity and fullness) of the
diagram functor $\dia_I$, for any $I$, which sends an object $X$ in
the category $\T(I)$ to its `diagram' in the base $\T(e)$. This is
done in subsections \ref{ss:ff} and \ref{ss:epidia} under a
special hypothesis which we shall commonly refer to as `Toda
condition' (\cf Remark \ref{rmk:Toda} for an explanation of
the name).

\subsection{Acknowledgements} I am very happy to thank Bernhard
Keller, for helpful discussions and inspiration about this project.
I thanks the anonymous referees for their useful suggestions to
improve the readability of the paper.


\section{Grothendieck's derivators} \label{s:der}
In this section we briefly recall the definition of derivator in the
sense of Grothendieck, following the exposition in \cite{Maltsiniotis07}.
The reader is invited to look at the original manuscript
\cite{Grothendieck90} for a complete exposition. We begin by reminding
some useful categorical notions and fixing the notations.

We denote by $\Cat$ the $2$-category of small categories. Among its
objects there are $\emptyset$, the empty category, and $e$, the
terminal category with one object $\ast$ and only the identity
morphism. If $\ca$ is a small category, we write $\ca^\circ$ to
indicate the opposite category. Given two small categories $\ca$
and $\cb$, then $\ul\Hom(\ca,\cb)$ denotes the category of functors
($1$-morphisms in the $2$-category $\Cat$) from $\ca$ to $\cb$
with natural transformations as morphisms ($2$-morphisms in $\Cat$).
For any functor $u : \ca \to \cb$ and a fixed object $b$ in the
category $\cb$, the objects of the category $\ca / b$ are the
pairs $(a,f)$, where $a$ is an object of $\ca$ and $f$ a
morphism from $u(a)$ to $b$ in the category $\cb$. A morphism
$\varphi : (a,f) \to (a',f')$ in the category $\ca / b$ is
determined by an arrow $g : a \to a'$ such that $f' \circ u(g) = f$.
Composition of morphisms in $\ca / b$ is clearly induced by
composition in $\ca$. Dually, we have the category $b \backslash \ca$
defined by $b \backslash \ca = (\ca^\circ / b)^\circ$. There are
canonical forgetful functors from $\ca / b$ and $b \backslash \ca$
to $\ca$, defined by $(a,f) \mapsto a$. By considering the identity
functor of $\cb$, we can canonically associate the categories $\cb / b$
and $b \backslash \cb$ with $\cb$, for any arbitrarily fixed object
$b \in \cb$. Let us call $u / b : \ca / b \to \cb / b$ and
$b \backslash u : b \backslash \ca \to b \backslash \cb$ the
induced functors which associate the pair $(a,f)$ with the pair
$(u(a),f)$. It is easy to check that the following commutative
squares are cartesian
\[
\xymatrix{
\ca / b \ar[r] \ar[d]_{u / b} & \ca \ar[d]^u &&
b \backslash \ca \ar[r] \ar[d]_{b \backslash u} & \ca \ar[d]^u \\
\cb / b \ar[r] & \cb & \mbox{,} & b \backslash \cb \ar[r] & \cb & \mbox{.}
}
\]

In the spirit of the theory of derivators we usually call {\em diagrams}
the objects of $\Cat$. Let us name $\Dia_\f$ the full $2$-subcategory of
$\Cat$ whose objects are {\em finite directed diagrams}, \ie, categories
whose nerves have only a finite number of non-degenerate simplices.
Equivalently, we can say that the objects in $\Dia_\f$ are the finite
categories whose underlying quiver (vertices: objects, arrows: non
identical morphisms) does not have oriented cycles (\eg, finite posets).
These objects, together with the functors of categories as $1$-morphisms
and the natural transformations as $2$-morphisms, endow $\Dia_\f$ with
the structure of a $2$-category.

Clearly, more general full sub-$2$-categories $\Dia$ of the $2$-category
$\Cat$ can be used as diagrams in the definition of a derivator (see,
\eg, \cite{Maltsiniotis07} for a discussion and references therein).
We simply write $\Dia$ whenever it is possible to work in those wider
situations.

\begin{definition} \label{def:prederivator}
A {\em prederivator of type $\Dia$} is a $2$-functor
\[
\D : \Dia^\circ \xymatrix{\ar[r] &} \CAT \ko
\]
\end{definition}
to the $2$-category of categories.

More explicitly, this means that there are a category $\D(I)$
associated with any diagram $I$, a functor
$u^* := \D(u)$, $u^* : \D(J) \to \D(I)$, associated with any functor
$u : I \to J$, and a natural transformation $\alpha^* := \D(\alpha)$,
\[
\xymatrix{\D(J) \rtwocell^{u^*}_{v^*}{^{\alpha^*}} & \D(I)} \ko
\]
associated with any natural transformation
\[
\xymatrix{I \rtwocell^u_v{_\alpha} & J} .
\]
These data have to verify the following coherence axioms: \\
\begin{itemize}
\item $\id_I^* = \id_{\D(I)}$, for any object $I$ in $\Dia$;
\item $\id_u^* = \id_{u^*}$, for any arrow
      $I \stackrel{^u}{\rightarrow} J$ in $\Dia$;
\item $(vu)^* = u^*v^*$, for any diagram $I \stackrel{^u}{\rightarrow} J
      \stackrel{^v}{\rightarrow} K$ in $\Dia$;
\item $(\beta\alpha)^* = \alpha^*\beta^*$, for any diagram
      $\xymatrix{I \ruppertwocell^u{\alpha} \rlowertwocell_w{\beta}
      \ar[r]_(.35)v & J}$ in $\Dia$;
\item $(\beta\star\alpha)^* = \alpha^*\star\beta^*$, for any diagram
      $\xymatrix{I \rtwocell^{u}_{u'}{_\alpha} & J
      \rtwocell^{v}_{v'}{_\beta} & K}$ in $\Dia$.
\end{itemize}

If $u : I \to J$ is a morphism in $\Dia$, we respectively denote
$u_*$ and $u_!$ the right and left adjoint functor to $u^*$, when
they exist.

\subsection{Triangulated derivators} \label{triader}
We have to introduce some more notations and terminology in order to
define the notion of derivator.

For any object $x$ of a small category $I$ lying in $\Dia$, we write
$i_{x,I} : e \to I$ to indicate the functor that is uniquely determined
by $x$ and $I$. Sometimes, we will simply write $i_x$, or even $x$,
when the context is clear.
Given an arbitrary prederivator $\D$ and an object $F$ of $\D(I)$, the
object $F_x := i_x^*(F)$ is called the {\em fiber of $F$ at the point $x$}.
If $\varphi : F \to F'$ is a morphism in $\D(I)$, the morphism
$\varphi_x := i_x^*(\varphi) : F_x \to F'_x$ is the
{\em morphism of fibers induced by $\varphi$}.

For the rest of this discussion we shall always assume that the
restriction functors admit the necessary adjoints. For any $2$-square
in $\Dia$,
\begin{equation*}
\begin{matrix}
\\\\\\
\cd \quad =
\end{matrix}
\quad \xymatrix{I' \ar[r]^v \ar[d]_{u'} & I \ar[d]^u
\ar@{}[dl]|*{\Swarrow}^\alpha \\
J' \ar[r]_w & J} \qquad
\begin{matrix}
\\\\\\
\qquad
\xymatrix{\alpha : uv \ar[r] & wu'} \ko
\end{matrix}
\end{equation*}
we denote by $\varepsilon : u^*u_* \to \id_{\D(I)}$,
$\eta : \id_{\D(J)} \to u_*u^*$, $\varepsilon' : u'^*u'_* \to \id_{\D(I')}$,
$\eta' : \id_{\D(J')} \to u'_*u'^*$ the adjunction morphisms. Let us define
the {\em base change morphism} $c_\cd : w^*u_* \to u'_*v^*$ to be the
composition $(u'_*v^* \star \varepsilon)(u'_* \star \alpha^* \star u_*)
(\eta' \star w^*u_*)$ of the following morphisms
\begin{tiny}
\[
\xymatrix{
w^*u_* \ar[rr]^(0.3){\eta' \star w^*u_*} &&
u'_*u'^*w^*u_* = u'_*(wu')^*u_*
\ar[rr]^(0.5){u'_* \star \alpha^* \star u_*} &&
u'_*(uv)^*u_* = u'_*v^*u^*u_*
\ar[rr]^(0.65){u'_*v^* \star \varepsilon}  && u'_*v^* .
}
\]
\end{tiny}
Clearly, we have a dual morphism $c'_\cd : v_!u'^* \to u^*w_!$.

Given a small category $I$ in $\Dia$, we indicate as $p_I : I \to e$
the unique functor from $I$ to the terminal category $e$. In order
to remain compatible with the standard notations of model category
theory, we can set, for any object $F$ in $\D(I)$,
\[
\hoinvlim_I F := (p_I)_*(F) \qquad \mbox{and} \qquad
\hodirlim_I F := (p_I)_!(F)
\]
in the category $\D(e)$ and talk about {\em homotopical
projective and inductive limit} of $F$. Clearly, our notation is
inspired by the theory of model categories and the corresponding
homotopy derivators. Let us emphasize that in those examples
the functors above correspond to the classical homotopy (co-)limits
but respectively over the corresponding opposite categories.

We remark that sometimes the notations $\Gamma_*(I,F) := (p_I)_*(F)$ and
$\Gamma_!(I,F) := (p_I)_!(F)$ are also used. In this way, these objects
can be thought of as the {\em global sections of $F$ over $I$}. The
topologically inclined reader might speak of the {\em (co-)homology of
$I$ with coefficients in $F$}. Other common notations that we also
use in this paper are $\holim_I F$ and $\hocolim_I F$, respectively
for $\hoinvlim_I F$ and $\hodirlim_I F$.

Note that the notions of homotopy limit and fiber of an object $F$ are
not completely unrelated. Indeed, we can construct a comparison
morphism in the following way. Let $u : I \to J$ be a morphism in $\Dia$,
$y$ an object of $J$ and $F$ an object in $\D(I)$. There are the
forgetful functors $j : I / y \to I$ and $j : y \backslash I \to J$.
We denote by $F|_{I/y}$ (\resp, $F|_{y \backslash I}$) the image $j^*(F)$
in $\D(I/y)$ (\resp, in $\D(y \backslash I)$). Let us consider the $2$-squares
\begin{equation*}
\begin{matrix}
\\\\\\
\mbox{$\cd_{u / y}$ = }
\end{matrix}
\xymatrix{
I/y \ar[r]^j \ar[d]_{p_{I/y}} & I \ar[d]^u
\ar@{}[dl]|*{\Swarrow}^\alpha \\
e \ar[r]_y & J
}
\qquad
\begin{matrix}
\\\\\\
\mbox{and}
\end{matrix}
\qquad
\begin{matrix}
\\\\\\
\mbox{$\cd_{y \backslash u}$ = }
\end{matrix}
\xymatrix{
y \backslash I \ar[r]^j \ar[d]_{p_{y \backslash I}} & I \ar[d]^u
\ar@{}[dl]|*{\Nearrow}_{\alpha'} \\
e \ar[r]_y & J
}
\begin{matrix}
\\\\\\
\mbox{\ko}
\end{matrix}
\end{equation*}
where the $2$-morphisms $\alpha$ and $\alpha'$ are defined, for any
object $x$ of $I$ and any morphism $f : u(x) \to y$ or, respectively,
$g : y \to u(x)$, by the formulae $\alpha_{(x, \; f : u(x) \to y)} := f$
and $\alpha'_{(x, \; g : y \to u(x))} := g$. For any $F$ in $\D(I)$,
there are the associated canonical base change morphisms
\[
\xymatrix{
\quad c_{\cd_{u / y}} : \; (u_*F)_y \ar[rr] && (p_{I/y})_*j^*(F) =
\hoinvlim_{I/y}(F|_{I/y}) & \mbox{and}
}
\]
\[
\xymatrix{
c'_{\cd_{y \backslash u}} : \; \hodirlim_{y \backslash I}
(F|_{y \backslash I}) = (p_{y \backslash I})_!j^*(F)
\ar[rr] && (u_!F)_y & \mbox{.}
}
\]

Suppose that we have fixed a prederivator $\D$. Let us consider two
arbitrary diagrams $I, J$ in $\Dia$. For any object $x$ of $I$, there
is the canonical functor $x_{I,J} : J \to I \times J$, which
associates any object $y$ in the category $J$ with the pair $(x,y)$.
By $2$-functoriality of $\D$, there exist the functor
${x_{I,J}}^* : \D(I \times J) \to \D(J)$ and the morphism of functors
${\alpha_{I,J}}^* : {x'_{I,J}}^* \to {x_{I,J}}^*$, for every morphism
$\alpha : x \to x'$ in $I$. In other words, we have a functor
\[
\D(I \times J) \times I^\circ \lra \D(J) \ko
\]
which associates the object ${x_{I,J}}^*(F)$ of $\D(J)$ with a pair
$(F,x)$. By adjunction, we get the functor
\[
\dia_{I,J} : \, \D(I \times J) \lra \ul\Hom(I^\circ, \D(J)) \ko
\]
where $\ul\Hom(I^\circ, \D(J))$ denotes the category of contravariant
functors from $I$ to $\D(J)$. When $J = e$, the functor $\dia_{I,e}$
induces a functor
\[
\dia_I : \, \D(I) \lra \ul\Hom(I^\circ, \D(e)) \ko
\]
whose image $\dia_I(F)$, for any $F$ in $\D(I)$, is called
{\em underlying diagram of $F$} or simply {\em diagram of $F$}. It is
the presheaf
\[
\dia(F) := \dia_I(F) : I^\circ \lra \D(e)
\]
defined by the equality $\dia_I(F)(x) = F_x$, for any object $x$ in $I$.

\begin{definition}
A {\em Grothendieck's derivator of type $\Dia$} is a prederivator
$\D$ of type $\Dia$ which satisfies the following axioms.
\begin{itemize}
\item[{\bf Der 1}]
   \begin{itemize}
   \item[a)] If $I$ and $J$ are in $\Dia$, the functor
   \[
   \xymatrix{\D(I \coprod J) \ar[rr]^{(i^*, j^*)} && \D(I) \times \D(J) \ko}
   \]
   induced by the canonical functors $i : I \to I \coprod J$ and
   $j : J \to I \coprod J$, is an equivalence of categories;
   \item[b)] the category $\D(\emptyset)$ is equivalent to the
   point category $e$.
   \end{itemize}
   \item[{\bf Der 2}] For any category $I$ in $\Dia$, the family of
   functors $i_x^* : \D(I) \to \D(e)$, indexed by the objects $x$
   lying in $I$, is conservative. Explicitly, this means that any
   morphism $\varphi : F \to F'$ in $\D(I)$ is an isomorphism iff
   the morphism of fibers $\varphi_x = i_x^*(f) : F_x \to F'_x$ is
   an isomorphism in $\D(e)$, for all $x$ in $I$.
   \item[{\bf Der 3}] For any morphism $u : I \to J$ in $\Dia$, the
   induced functor $u^* : \D(J) \to \D(I)$ admits a right adjoint
   $u_* : \D(I) \to \D(J)$ and a left adjoint $u_! : \D(I) \to \D(J)$.
   \item[{\bf Der 4}] For any morphism $u : I \to J$ in $\Dia$, any
   object $y$ in $J$ and any $F$ in $\D(I)$, the associated canonical
   base change morphisms
   \[
   \xymatrix{
   & \qquad c_{\cd_{u / y}} : \; (u_*F)_y \ar[rr] && (p_{I/y})_*j^*(F) =
   \hoinvlim_{I/y}(F|_{I/y}) & \mbox{and}
   }
   \]
   \[
   \xymatrix{
   c'_{\cd_{y \backslash u}} : \; \hodirlim_{y \backslash I}
   (F|_{y \backslash I}) = (p_{y \backslash I})_!j^*(F)
   \ar[rr] && (u_!F)_y
   }
   \]   
   are invertible.
\end{itemize}
\end{definition}

We usually refer to the notion of a Grothendieck's derivator simply as
`derivator'.

\begin{definition} \label{def:epivalence}
A prederivator (\resp, derivator) is an {\em epivalent prederivator}
(\resp, {\em epivalent derivator}) if it enjoys the property stated
in the following axiom.
\begin{itemize}
\item[{\bf Der 5}] (\em{Epivalence axiom}). For any $J$ in $\Dia$,
the functor
   \[
   \dia_{\Delta_1,J} : \D(\Delta_1 \times J)
   \lra \ul\Hom(\Delta_1^\circ,\D(J)) \ko 
   \]
where $\Delta_1$ denotes the category $\{0 \la 1\}$ of $\Dia$,
is full and essentially surjective.
\end{itemize}
\end{definition}

We recall that some authors (\eg, Moritz Groth in \cite{Groth11},
Andrei Radulescu-Banu in \cite{Radulescu-Banu09} for the notion
he calls `Heller derivator' and Kevin Carlson in \cite{Carlson16})
reserve the name `strong derivator' for the notion of derivator
which satisfies (some version of) axiom {\bf Der 5}.

\begin{remark} \label{rmk:epivalence}
We agree that the status of Axiom {\bf Der 5} is perhaps not in
its final form. Indeed, we use the axiom version of Franke
\cite{Franke96} but Heller in \cite{Heller97} uses a variant of
the axiom for finite free categories. Moreover, one can show that
a yet more general version of this axiom is satisfied by typical
examples such as homotopy derivators of model categories (\cf
\cite{Radulescu-Banu09} by Radulescu–Banu), and this more refined
version might be useful in future applications.
Recently, yet a different version of the fifth axiom was proposed
by Carlson, and this version is enjoyed by homotopy derivators of
quasi-categories (certain partial underlying diagram functors are
assumed to be smothering functor in the language of Riehl–-Verity).
This version was used to offer a different perspective on the
relation between derivators and homotopy theories of homotopy
theories. (This version, however, is not satisfied by homotopy
derivators of model categories.)
\end{remark}

Let us remark that axioms {\bf Der 1} and {\bf Der 2} above assure
us that, for any diagram $I$ in $\Dia$, the category $\D(I)$ has
finite products and coproducts and in particular it admits initial
and final objects.

A morphism $j : U \to I$ of $\Dia$ is an {\em open immersion}
if it is injective on objects, fully faithful, and if any morphism
$f : y \to j(x)$ in $I$ is in the image of $j$, \ie, it is of the
form $j(g) : j(x') \to j(x)$, for some morphism $g : x' \to x$ in $U$.
Dually, a morphism $i : Z \to I$ of $\Dia$ is a
{\em closed immersion} if $i^\circ : Z^\circ \to I^\circ$ is an
open immersion. Open immersions and closed immersions are stable
under composition and pullback.

\begin{definition} \label{def:pointed}
A derivator $\D$ is {\em pointed} if the following axiom holds.
\begin{itemize}
\item[{\bf Der 6}] For any closed immersion $i : Z \to I$ in $\Dia$,
the induced functor $i_*$ admits a right adjoint $i^!$. Dually, for
any open immersion $j : U \to I$ in $\Dia$, the induced functor $j_!$
admits a left adjoint $j^?$.
\end{itemize}
\end{definition}

Notice that in \cite{Groth11} Groth proves that axiom {Der 6} is
equivalently stated by requiring that the base $\D(e)$ admits a
zero object.

Coming back to notations and terminology, we write $\boxempty$ to
indicate the category $\Delta_1 \times \Delta_1$, $\lefthalfcap$,
$\righthalfcup$ to indicate the two subcategories
\begin{equation*}
\xymatrix{(0,0) & (0,1) \ar[l] \\ (1,0) \ar[u]} \qquad
\begin{matrix}
\\\\\\
\mbox{,}
\end{matrix}
\qquad \xymatrix{& (0,1) \\ (1,0) & (1,1) \ar[l] \ar[u]}
\end{equation*}
of $\boxempty$, and $i_\lefthalfcap : \lefthalfcap \to \boxempty$,
$i_\righthalfcup : \righthalfcup \to \boxempty$ to indicate the
inclusion functors. An object $F$ of $\D(\boxempty)$ is
{\em homotopically cartesian} (\resp, {\em homotopically cocartesian}),
or, more simply, {\em cartesian} (\resp, {\em cocartesian}), if the
adjunction morphism
\begin{equation*}
\eta^F : F \lra (i_\righthalfcup)_*(i_\righthalfcup)^*F
\quad\quad\quad (\mbox{\resp,} \;\;
\varepsilon^F : (i_\lefthalfcap)_!(i_\lefthalfcap)^*F \lra F)
\end{equation*}
is an isomorphism.

\begin{definition} \label{def:triangulated}
An epivalent derivator is {\em triangulated} if it is pointed and if
the following axiom holds.
\begin{itemize}
\item[{\bf Der 7}] Any object $F$ in $\D(\boxempty)$ is homotopically
cartesian if and only if it is homotopically cocartesian (\ie,
homotopically bicartesian).
\end{itemize}
\end{definition}

We prefer to follow Maltsiniotis--Keller \cite{Maltsiniotis07} and
Cisinski--Neeman \cite{CisinskiNeeman08} and include axiom
{\bf Der 5} in the very definition of a triangulated derivator.
The main argument in favour of this choice is that in this work
we are above all interested in triangulated derivators
and the epivalence axiom is central for the construction of the
triangulated structure of the local categories $\D(I)$, as in
the proof of Theorem \ref{thm:Malts}.

\begin{example}
The following example by Cisinski \cite{Cisinski03} is of
particular interest.
Let $\cm$ be a category such that there exists a stable Quillen
closed model structure on with $\cw$ as weak equivalences. Then,
the prederivator $\D_{(\cm,\cw)}$ defined, on any small category
$I$, by
\[
\D_{(\cm,\cw)}(I) := \ul\Hom(I^\circ, \cm)[\cw^{-1}]
\]
and, on any functor $u : I \to J$, by
\[
\xymatrix{\D_{(\cm,\cw)}(J) \ar[rrr]^{u^* := \ol{(u^\circ)^*}} &&&
\D_{(\cm,\cw)}(I)} \ko
\]
is a triangulated derivator. Here the overline means the induced
functor between the localized categories.
The fact that these derivators are triangulated is verified in all
detail in \cite{GrothPontoShulman14} by Groth–-Ponto-–Shulman.
\end{example}

For example, the stable model category of spectra
\cite{BousfieldFriedlander78,Dwyer04,Lydakys98,
MandellMaySchwedeShipley01,Quillen67} gives rise to a triangulated
derivator $\S p$.
The following theorem was announced in \cite{Maltsiniotis07}.
A precursor version, based on Franke's theory, can be found in
\cite{Franke96}.
The interested reader can find a proof in \cite[Theorem 4.15]{Groth11}.
The fact that the triangulations are `canonical' was made more precise
in the paper \cite{Groth16} by Groth (this includes a $2$-functoriality
statement as opposed to a mere functoriality, a `$2$-naturality statement',
and a uniqueness statement).

\begin{theorem}[Maltsiniotis] \label{thm:Malts}
If $\D$ is a triangulated derivator, for every $I$ in $\Dia$, there
is a canonical structure of triangulated category on $\D(I)$, such
that for every morphism $u : I \to J$ in $\Dia$, the functor
$u^* : \D(J) \to \D(I)$ is canonically endowed with the structure
of triangulated functor.
\end{theorem}
We refer to the articles of Maltsiniotis and Groth for all the
details concerning the canonical triangulated structure of the
values of a triangulated derivator. In what follows we are going
to use notation an terminology of \cite{Maltsiniotis07}.

\subsection{Morphisms of derivators} \label{ss:mor}
\begin{definition} \label{def:dermor}
Let $\D$ and $\E$ be two prederivators of type $\Dia$. A
{\em morphism} $F : \D \to \E$ is given by
\begin{itemize}
\item a functor $F_I : \D(I) \to \E(I)$, for each $I$ in $\Dia$;
\item an invertible transformation of functors
\[ \varphi_u : F_I u^* \iso u^* F_J \ko \]
for each morphism $u : I \to J$ in $\Dia$;
\end{itemize}
such that
\begin{itemize}
\item $\varphi_{\id_I} = \id_{F_I}$, for each $I$ in $\Dia$;
\item $\varphi_{uv} = (v^* \varphi_u)(\varphi_v u^*)$, for each pair
of morphisms $K \stackrel{v}{\to} I \stackrel{u}{\to} J$ in $\Dia$;
\item $\varphi_u \alpha^* = \alpha^* \varphi_v$, for each $2$-morphism
$\xymatrix{I \rtwocell^u_v{_\alpha} & J}$ in $\Dia$.
\end{itemize}
\end{definition}

Let us remark that this is a special case of the notion of a
`pseudo--natural' transformation.

The {\em composition} of two morphisms of prederivators $F : \D \to \E$
and $G : \E \to \F$ is defined by setting $(GF)_I = G_I F_I$, for all $I$
in $\Dia$, and $(\gamma\varphi)_u = (\gamma_u F_J)(G_I \varphi_u)$, for
all morphisms $u : I \to J$ in $\Dia$.

If $F$ and $G$ are two morphisms of prederivators from $\D$ to
$\E$, a {\em $2$-morphism} $\alpha : F \to G$ is given by a
natural transformation $\alpha_I : F_I \to G_I$, for each $I$ in
$\Dia$, such that $(u^*\alpha_J)\varphi_u = \gamma_u(\alpha_I u^*)$,
for all morphisms $u : I \to J$ of $\Dia$ (this is a special case of
the notion of a `modification'). The {\em composition} of two such
morphisms is clear.
Thus, the morphisms between any fixed pair of prederivators $\D$ and
$\E$, and their $2$-morphisms, are the objects and, respectively, the
morphisms of a category that we denote $\ul\HOM(\D,\E)$.

A {\em morphism of derivators} is just a morphism of the underlying
prederivators and similarly for $2$-morphisms. Thus, given a pair of
derivators $\D$ and $\E$, we use the same notation $\ul\HOM(\D,\E)$
to indicate the category with morphisms of derivators from $\D$ to $\E$
as objects and the related $2$-morphisms as morphisms.

Let us remark that, in the presence of another morphism $v : K \to I$,
for any object $X$ lying in $\D(J)$, the isomorphism $\varphi_{uv}^X$
is explicitly given by the composition
\[
\xymatrix{
F_K v^*u^*X \ar[rr]^{\varphi_v^{u^*X}}_\sim && v^*F_I u^*X
\ar[rr]^{v^*(\varphi_u^X)}_\sim && v^*u^*F_JX .
}
\]
The morphism $\varphi_u$ induces, via the adjunction morphisms
$\eta_u : \id \to u^*u_!$ and $\varepsilon_u : u_!u^* \to \id$,
another morphism
\[
\varphi^u : u_!F_I \to F_J u_! \ko
\]
whose action on any object $X$ of $\D(I)$ is given by the composition
\[
\xymatrix{
u_!F_IX \ar[rr]^<>(0.5){u_!F_I(\eta_u^{X})} && u_!F_I u^*u_!X
\ar[rr]^<>(0.5){u_!(\varphi_u^{u_!X})}_\sim && u_!u^*F_J u_!X
\ar[rr]^<>(0.5){\varepsilon_u^{F_J u_!X}} && F_J u_!X .
}
\]

Analogously to the contravariant case, we have the relation
\[
\varphi^{u,v_!X} \circ u_!(\varphi^{v,X}) = \varphi^{uv,X} \ko
\]
valid for all objects $X$ of $\D(K)$.
Indeed, functoriality of the morphisms involved and the relation
$\varphi_{uv}^X = v^*(\varphi_u^X) \circ \varphi_v^{u^*X}$ give us
commutative diagrams that we can use in the following sequence of
equalities
\begin{eqnarray*}
\varphi^{uv,X} 
& = &
\varepsilon_{uv}^{F_J(uv)_!X} \circ
(uv)_![\varphi_{uv}^{(uv)_!X}] \circ (uv)_!F_J(\eta_{uv}^X) \\
& = &
\varepsilon_u^{F_J(uv)_!X} \circ u_![\varepsilon_v^{u^*F_J(uv)_!X}] \circ
(uv)_![v^*(\varphi_u^{(uv)_!X}) \circ \varphi_v^{u^*(uv)_!X}] \circ \\
&&
\circ (uv)_!F_Jv^*(\eta_u^{v_!X}) \circ (uv)_!F_J(\eta_v^X) \\
& = &
\varepsilon_u^{F_J(uv)_!X} \circ u_![\varphi_u^{(uv)_!X}] \circ
u_![\varepsilon_v^{F_Iu^*(uv)_!X}] \circ (uv)_!v^*F_I(\eta_u^{v_!X}) \circ \\
&&
\circ (uv)_![\varphi_v^{v_!X}] \circ (uv)_!F_J(\eta_v^X) \\
& = &
\varepsilon_u^{F_J(uv)_!X} \circ u_![\varphi_u^{(uv)_!X}] \circ
u_!F_I(\eta_u^{v_!X}) \circ u_![\varepsilon_v^{F_I v_!X}] \circ
u_!v_![\varphi_v^{v_!X}] \circ u_!v_!F_J(\eta_v^X) \\
& = &
\varphi^{u,v_!X} \circ u_!(\varphi^{v,X}) .
\end{eqnarray*}

Let us remark that one might be able to find the last formula
(and quite a few others of this kind) by using the `calculus of
mates', \ie, a $2$-categorical machinery which provides a
framework for pasting diagrams arising in the presence of adjoint
pairs in $2$-categories. Nevertheless, since we intend this paper
useful for the applications and we are trying to get a high
readability in the exposition, we prefer to maintain all the higher
categorical details at the very least.

\subsection{Exact categories and their derived categories}
\label{ss:exdercat}
In this subsection we recall the notion of exact category and the
construction of the related derived category. This notion was
introduced by Quillen in \cite{Quillen73} and modified by Keller
in \cite{Keller90}.

\begin{definition} \label{def:excat}
An {\em exact category} is an additive category $\ca$ endowed
with a class $\Phi$ of {\em short exact sequences}, also said
{\em exact pairs},
\[
\xymatrix{X \ar@{ >->}[r]^i & Y \ar@{->>}[r]^p & Z}
\]
(\ie, $\kernel(p) = i$, $\cok(i) = p$), which is closed under isomorphisms.
Elements in this class are called {\em conflations} and the arrows of
type $\xymatrix{\ar@{ >->}[r]&}$ (\resp, $\xymatrix{\ar@{->>}[r]&}$)
are called {\em inflations} (\resp, {\em deflations}), such that the
following axioms hold.
\begin{itemize}
\item[{\bf Ex 0}] The identity morphism $\xymatrix{0 \ar[r] & 0}$
     is a deflation.
\item[{\bf Ex 1}] The class of deflations is stable under composition.
\item[{\bf Ex 2}] For any deflation $\xymatrix{Y \ar@{->>}[r]^d & Z}$
and any morphism $\xymatrix{Z' \ar[r]^f & Z}$ in $\ca$, there exists
a cartesian square (pull-back)
\[
\xymatrix{
   Y' \ar@{-->>}[rr]^{d'} \ar@{-->}[d]_{f'} && Z'\;\; \ar@<-0.7ex>[d]^f  \\
   Y \ar@{->>}[rr]_d && Z
   }
   \begin{matrix} \\\\\\ \mbox{\ko} \end{matrix}
\]
where $d'$ is a deflation.
\item[{\bf Ex2\up{o}}] For any inflation $\xymatrix{X \ar@{ >->}[r]^i & Y}$
and any morphism $\xymatrix{X \ar[r]^f & X'}$ in $\ca$, there exists
a cocartesian square (push-out)
\[
\xymatrix{
   X \ar@{ >->}[rr]^i \ar[d]_f && Y\;\;\: \ar@{-->}@<-1ex>[d]^{f'}  \\
   X' \ar@{ >-->}[rr]_{i'} && Y'
   }
   \begin{matrix} \\\\\\ \mbox{\ko} \end{matrix}
\]
where $i'$ is an inflation.
\end{itemize}
\end{definition}

\begin{example} \label{ex:exact cats}
The following are examples of exact categories.
\begin{itemize}
\item[a)] The opposite category $\ca^\circ$ of any exact category $\ca$
canonically inherits a structure of exact category.
\item[b)] Any additive category $\ca$, endowed with all of its split
short exact sequences.
\item[c)] Let $\ca$ be an exact category. Then, for any small
category $I$ the category defined by
\[
\ul\ca(I) := \ul\Hom(I^\circ,\ca)
\]
becomes exact when endowed with componentwise conflations.
\end{itemize}
\end{example}

Remark that the exact category $\ul\ca(I)$ usually have non split
conflations even if all conflations of $\ca$ split.

Let us call $\CA$ the {\em category of bounded complexes}
\[
\xymatrix{
\ldots \ar[r] & M^p \ar[r]^{d^p} & M^{p+1} \ar[r] & \ldots \ko \quad
d^{p+1}d^p=0 \ko \quad p \in \Z \ko
}
\]
over an exact category $\ca$, where $M^p=0$ for all $|p| >> 0$.
The category $\HA$ is the quotient of $\CA$ by the ideal of
nullhomotopic morphisms. Note that the category $\HA$ is
triangulated in Verdier sense \cite[Ch.~2.2]{Verdier96}, with
suspension functor $\Sigma$ given by
\[
(\Sigma X)^p := X^{p+1} \quad \ko \quad d_{\Sigma X} := -d_X \ko
\]
and triangles obtained from the componentwise split short exact sequences.

A complex $N$ is {\em strictly acyclic} if there exist conflations
\[
\xymatrix{Z^p \ar@{ >->}[r]^{i^p} & N^p \ar@{->>}[r]^{q^p} & Z^{p+1}}
\ko p \in \Z \ko
\]
such that $d^p = i^{p+1}q^p$, for all $p \in \Z$. A complex $N$
is {\em acyclic} if it is isomorphic in $\HA$ to a strictly acyclic
complex. One can show that, if the category $\ca$ is idempotent
complete, then any complex is acyclic iff it is strictly acyclic
(\cf \cite{Keller07} for details on the constructions of this paragraph).
Moreover, (without the idempotent completeness condition) it is true
that the images in $\HA$ of the acyclic complexes form a thick
subcategory $\cn$ of $\HA$, \ie, $\cn$ is a triangulated subcategory
of $\HA$ which is stable under retracts. The {\em (bounded) derived
category} of $\ca$, that we denote $\DA$, is the Verdier quotient
\cite[Ch.~2.2]{Verdier96} $\HA \, / \, \cn$. Let us remark that the
derived category might be a large category.

Note that there is a canonical embedding \cite{Keller07}
\[
can : \ca \lra \DA \ko \quad X \mapsto
(\ldots \to 0 \to X \to 0 \to \ldots) \ko
\quad \deg(X) = 0 \quad .
\]
For any conflation in $\ca$,
\[
\varepsilon : \quad
\xymatrix{X \ar@{ >->}[r]^i & Y \ar@{->>}[r]^p & Z} \ko
\]
we have a canonical distinguished triangle in $\DA$,
\[
can(\varepsilon) : \quad
\xymatrix{can(X) \ar[rr]^{can(i)} && can(Y) \ar[rr]^{can(p)} &&
can(Z) \ar[rr]^{\partial\varepsilon} && \Sigma can(X)} .
\]

In this way, we have constructed a $2$-functor
\[
\EXA \to \TRIA \quad \ko \quad \ca \mapsto \DA \ko
\]
from the $2$-category of exact categories to the $2$-category
of triangulated categories.

It may be that some explanations are in order, here. The objects
of $\EXA$ are pairs $(\ca,\Phi)$ formed by a (possibly large)
additive category $\ca$ with a class of conflations $\Phi$ as in
definition \ref{def:excat}, $1$-morphisms
$F : (\ca,\Phi) \to (\ca',\Phi')$ are pairs formed by an additive
morphism $F : \ca \to \ca'$ such that $F(\Phi) \subseteq \Phi'$
with an obvious morphism of conflation, $2$-morphisms
$\alpha : F \to F'$ are natural transformations of additive functors
which are compatible with the exact structure.
Analogously, objects of the $2$-category $\TRIA$ are pairs
$(\ct,\Psi)$ formed by a (possibly large) additive category
$\ct$ with a class of distinguished triangles $\Psi$,
$1$-morphisms $(F,\delta) : (\ct,\Psi) \to (\ct',\Psi')$ are pairs
formed by an additive morphism $F : \ct \to \ct'$ such that
$F(\Psi) \subseteq \Psi'$ in the sense of \cite{KellerVossieck87}
with an obvious morphism of distinguished triangles, $2$-morphisms
$\alpha : (F,\delta) \to (F',\delta')$ are natural transformations
of functors which are compatible with the triangulated structure
(\cf \cite{KellerVossieck87} for complete definitions). The action
of the $2$-functor $\cd^b$ on the functor $F$ gives a {\em derived}
functor $\R F$ and functoriality is expressed by the commutativity
of the following diagram
\[
\xymatrix{
\ca \ar[rr]^{F} \ar[d]_{can} &&
\ca' \ar[d]^{can'} \\
\cd^b\ca \ar[rr]_{\R F} && \cd^b\ca' & .
}
\]
As usual in the literature, we do not write all the heavy
$2$-categorical notation and we leave to the reader filling in
the details. The canonical functors $can$ and $can'$ are examples
of exact or $\partial$-functors (\cf definitions just after theorem \ref{thm:main}).

The reason we consider derivators or other structures like towers
is that the canonical embedding $can$ does not have the universal
extension property we are looking for.

\subsection{The triangulated derivator associated with an exact category}
\label{ss:triader}
Let $\ca$ be an exact category. Here is one central notion of this paper.
\begin{definition} \label{def:derex}
$\D_\ca : \Dia^\circ \to \CAT$ is the prederivator that associates the
bounded derived category $\cd^b(\ul\Hom(I^\circ,\ca))$ with any $I$
in $\Dia$.
\end{definition}

In the Appendix \cite{Keller07} to the article \cite{Maltsiniotis07},
B. Keller proves the following

\begin{theorem}
Let us consider the restriction of the derivator in the definition
\ref{def:derex} to the $2$-subcategory $\Dia_\f^\circ$. Then, the
prederivator $\D_\ca$ we obtain this way is a triangulated derivator.
\end{theorem}

Remark that triangularity is a property of a derivator as opposed
to the triangulated extra structure that one can put on an additive
category. In order to emphasize that all axioms of a triangulated
derivator ask for properties some authors (\eg, \cite{Groth11})
refer to it as a {\em stable derivator} (in analogy to the theory
of model categories and quasi--categories).

Let us call $\ADD$ the full $2$-subcategory of $\CAT$ which contains
the additive categories. We can think of the category $\EXA$ consisting
of the exact categories, \ie, additive categories endowed with {\em
some} exact structure. Morphisms in this category are known as `exact
functors', \ie, additive functors which preserve conflations. Notice
that, thanks to Proposition A.1 of \cite{Keller90}, this amounts to
ask that the additive functors preserve {\em bicartesian squares of
the structure}, \ie, bicartesian squares with two parallel maps
consisting of inflations whereas the other two parallel maps are
deflations.

An {\em additive derivator} is a derivator $\A$ whose base $\A(e)$
is an additive category. In \cite{Groth12} it is shown that this
amounts to ask that the image of the derivator $\A$ lies in the
full $2$-subcategory $\ADD$ of additive categories, \ie, it is
componentwise given by additive categories and by additive
functors among them.

Let us emphasize that in this article we consider a {\em particular}
case of the yet not defined notion of an {\em exact (pre)derivator},
\ie, whenever we mention this notion in this article we always mean
the represented exact prederivator associated with an exact category.
We leave the discussion on the general notion of an exact derivator
for future work.

Let us denote $\ul\ca : \Dia \to \CAT$ the prederivator that associates
the category $\ul\ca(I) := \ul\Hom(I^\circ,\ca)$ of contravariant
functors from $I$ to an exact category $\ca$, with any diagram $I$
in $\Dia$.
It is clear that, even if the image of $\ul\ca$ is in the $2$-category
$\CAT$, we may think, for any diagram $I$, of the {\em exact} category
$\ul\ca(I)$ endowed with short exact sequences which are argumentwise
conflations. Moreover, it is important to stress that the prederivator
just constructed is an epivalent prederivator \myref{def:epivalence}.

Clearly, the canonical embedding of an exact category into its bounded
derived category gives rise to a canonical morphism of prederivators,
which we still call
\[
can : \ul\ca \lra \D_\ca \ko
\]
given, for any diagram $I$, by the canonical embedding
\[
can_I : \ul\ca(I) \lra \D_\ca(I) = \cd^b(\ul\ca(I)) \; \mbox{.}
\]

For each bicartesian square in $\ul\ca(\boxempty)$ of type
\begin{eqnarray} \label{eqn:square}
\begin{matrix}
\\\\\\
X
\end{matrix}
&
\begin{matrix}
\\\\\\
\mbox{=}
\end{matrix}
&
\xymatrix{
X_{00} \ar@{ >->}[r] \ar@{->>}[d] & X_{01} \ar@{->>}[d] \\
X_{10} \ar@{ >->}[r] & X_{11}
}
\begin{matrix}
\\\\\\
\mbox{,}
\end{matrix}
\end{eqnarray}
its image $can_\square(X)$ is bicartesian in $\cd^b(\ul\ca(\boxempty))$,
\ie, isomorphisms
\[
X_{00} \Iso X_{10} \prod_{X_{11}}^\R X_{01} \:\:\:\:\:\: \mbox{ and }
\:\:\:\:\:\: X_{10} \coprod_{X_{00}}^\L X_{01} \Iso X_{11}
\]
hold. We say that the additive morphism $can : \ul\ca \to \D_\ca$
is {\em exact}. More generally we have the following

\begin{definition} \label{def:extria}
Let $\E$ and $\F$ be triangulated derivators. 

\begin{itemize}
\item[a)] An additive morphism $F : \ul\ca \to \E$ is said
{\em exact} or {\em $\partial$-morphism} if $F_\square(X)$ is
bicartesian in $\E(\boxempty)$ for each $X$ of type (\ref{eqn:square}).
\item[b)] An additive morphism of derivators  $F :\E \to \F$ is
{\em triangulated} if $F_\square(X) \in \F(\boxempty)$ is
bicartesian for each bicartesian $X \in \E(\boxempty)$.
\item[c)] An additive $2$-morphism $\alpha : F \to G$ of
$\partial$-morphisms of derivators $F$ and $G$ from
$\ul\ca$ to $\E$ is {\em exact} (or $\partial$) if, for
each bicartesian square $X$ in $\ul\ca(\boxempty)$ of type
\myref{eqn:square}, the morphism
\[
\alpha_{\square}^X : F_\square X \to G_\square X
\]
is a {\em bicartesian morphism} of bicartesian squares in
$\E(\boxempty)$.
More explicitly this means that the canonical morphisms
\[
\eta^{\alpha_\square^X} : \alpha_{\square}^X \lra
{i_\righthalfcup}_* {i_\righthalfcup}^* \alpha_{\square}^X \ko
\]
\[
\eps^{\alpha_\square^X} : {i_\lefthalfcap}_! {i_\lefthalfcap}^*
\alpha_{\square}^X \lra \alpha_{\square}^X \ko
\]
induced by the adjunction isomorphisms $\eta^{X} : X \iso \ir_*\ir^*X$
and $\eps^{X} : \il_!\il^* X \iso X$, are invertible.
\item[d)] An additive $2$-morphism $\alpha : F \to G$ of triangulated
morphisms of derivators $F$ and $G$ from $\E$ to $\F$ is {\em triangulated}
if, for each bicartesian square $X$ in $\E(\boxempty)$, the morphism
\[
\alpha_{\square}^X : F_\square X \to G_\square X
\]
is a {\em bicartesian morphism} of bicartesian squares in
$\F(\boxempty)$.
\end{itemize}
\end{definition}

\begin{remark} \label{rmk:all}
Since $\varepsilon : {i_\lefthalfcap}_! {i_\lefthalfcap}^* \to \id$
is a natural transformation the following square
\[
\xymatrix{
{i_\lefthalfcap}_! {i_\lefthalfcap}^* F_\square X
\ar[r]^<>(0.5){\eps^{F_\square X}}
\ar[d]_{{i_\lefthalfcap}_! {i_\lefthalfcap}^* \alpha_\square^ X} &
F_\square X \ar[d]^{\alpha_\square^X} \\
{i_\lefthalfcap}_! {i_\lefthalfcap}^* G_\square X
\ar[r]_<>(0.5){\eps^{G_\square X}} & G_\square X
}
\]
is commutative. It follows that if the horizontal arrows are invertible
then the canonical morphism $\eps^{\alpha_\square^X} : {i_\lefthalfcap}_!
{i_\lefthalfcap}^* \alpha_{\square}^X \lra \alpha_{\square}^X$ is
invertible, too. An analogous observation holds for the natural
transformation $\eta : \id \to {i_\righthalfcup}_* {i_\righthalfcup}^*$.
Therefore, the conditions in items c) and d) are always true for all
additive $2$-morphisms $\alpha : F \to G$, \ie, we get that {\em all}
additive $2$-morphisms of exact/triangulated (pre)derivators are
automatically exact or triangulated.
\end{remark}

We shall prove in proposition \ref{prop:rightdef} that our
definition \ref{def:extria} is correct, in the sense that if
an additive morphism of (pre)derivators $F$ preserves bicartesian
squares then it locally preserves conflations/distinguished
triangles of the canonical structure, thanks to the fact that
the involved derivators enjoy the `epivalence axiom' (\ie,
if axiom {\bf Der 5} holds).

The {\em composition} of two morphisms is clear. Thus, the
$\partial$-morphisms between any fixed pair of derivators
$\A$ and $\E$, and their $2$-$\partial$-morphisms, are the
objects and, respectively, the morphisms of a category that
we indifferently denote $\ul\HOM_{ex}(\A,\E)$ or
$\ul\HOM_{\partial}(\A,\E)$. Analogously, the category
$\ul\HOM_{tr}(\E, \F)$ is the category of triangulated morphisms
of derivators from $\E$ to $\F$ and their triangulated morphisms.

The following is the main result of this article.

\begin{theorem} \label{thm:main}
Let $\E$ be a triangulated derivator of type $\Dia_\f$. Then,
the canonical morphism $can : \ul\ca \to \D_\ca$ induces an
equivalence of categories
\[
\ul\HOM_{tr}(\D_\ca,\E) \Iso \ul\HOM_{ex}(\ul\ca,\E) .
\]
\end{theorem}

Let $\ct$ be a triangulated category and $\ca$ an exact category. A
{\em $\partial$-functor (also called exact functor)} is an additive
functor $F : \ca \to \ct$ endowed with functorial distinguished
triangles
\[
\xymatrix{
F(X) \ar[r] & F(Y) \ar[r] & F(Z) \ar[r]^{\partial\varepsilon} & \Sigma F(X)
}
\]
for each conflation $\varepsilon$ :
\[
\xymatrix{
X \ar@{ >->}[r] & Y \ar@{->>}[r] & Z \quad .
}
\]
A {\em $2$-morphism of $\partial$-functors} $\alpha : F \to F'$
is a natural transformation such that the square
\[
\xymatrix{
FZ \ar[r]^{\partial\eps} \ar[d]_{\alpha^Z} &
\Sigma FX \ar[d]^{\Sigma\alpha^X} \\
F'Z \ar[r]^{\partial'\eps} & \Sigma F'X
}
\]
commutes, for each conflation $\eps$.
This also gives a category that we indifferently denote
$\ul\Hom_{ex}(\ca,\ct)$ or $\ul\Hom_{\partial}(\ca,\ct)$.

Let us recall that a (triangulated) category (\resp, a (triangulated)
functor) is {\em basic} if it occurs as the evaluation over the terminal
diagram $e$ of a (triangulated) derivator (\resp, of a morphism of
(triangulated) derivators). Apparently, it is not known whether all
(triangulated) categories are basic, nor if the `extension' of a
(triangulated) category to a (triangulated) derivator having it as a
base is unique. However, we have some reasons to think of a negative
answer to both of these statements.

We want to stress that `basic' always implies a {\em choice} of a
derivator. As a consequence, when we claim that some basic morphism
extends `uniquely', we always mean `uniquely with respect to a
{\em chosen} derivator'.

The main aim of the second part of this article is to show that
from Theorem \ref{thm:main} we can prove the following theorem,
whose proof is completed by the end of the paper.

\begin{theorem} \label{thm:ext}
\begin{itemize}
\item[a)] Let $F : \ca \to \ct$ be a $\partial$-functor.
Suppose that $\ct = \T(e)$ for some triangulated derivator $\T$
of type $\Dia_\f$ and that
\[
\Hom_{\ct}(F(X),\Sigma^n F(Y)) = 0
\]
for each $n < 0$, for all $X$, $Y$ in $\ca$ (Toda condition).
Then, the functor $F$ uniquely extends (up to a unique basic
natural isomorphism) to a basic triangulated functor
\[
\tilde{F} : \cd^b\ca \to \ct .
\]
\item[b)] Let $F, F' : \cd^b\ca \to \ct$ be two basic triangulated
functors such that the (Toda) conditions
\[
\Hom_{\ct}(\Sigma^n F(X), F(Y)) = 0 \ko  \qquad n>0 \ko
\]
\[
\Hom_{\ct}(\Sigma^n F'(X), F'(Y)) = 0 \ko  \,\,\,\quad n>0 \ko
\]
\[
\Hom_{\ct}(\Sigma^n F(X), F'(Y)) = 0 \ko \,\,\,\,\,\quad n>0 \ko
\]
hold for all $X, Y$ in $\ca$. Suppose that $\mu$ is a $2$-morphism
between the restrictions of $F$ and $F'$ to $\ca$. Then, the
$2$-morphism $\mu$ uniquely extends to a basic $2$-morphism of
triangulated functors
\[
\tilde{\mu} : F \to F' .
\]
\end{itemize}
\end{theorem}

Let us notice that item b) in Theorem \ref{thm:ext} {\em differs}
from the analogous fact, true for the towers, which is contained in
Corollary 2.7 in \cite{Keller91}. In fact, the same conclusion there
follows from the last of our three Toda conditions alone. The
analogous situation holds for the item b) in our Theorem
\ref{thm:ext1} when compared with Keller's Theorem 2.7 in
\cite{Keller91}.

\begin{remark}
The extension of $F : \ca \to \ct$ to the derived category is unique
once the derivator $\T$ is fixed. The question whether the extension
is not unique since there might be more than one triangulated derivator
with the same base $\ct$ remains open.
\end{remark}

\begin{remark} \label{rmk:Toda}
The name {\em Toda condition} is inspired by the use made by
topologists who refer to the obstructions to building a Postnikov
system as {\em Toda brackets}. Let us recall that the same condition
already appears in the foundational work \cite{Kapranov88} by M.
Kapranov.
More importantly to us for the similarity to the results of the
present paper are the theorems about unicity of enhancement of
some triangulated categories and their functors to
$\opname{DG}$-categories proved by V. Lunts and D. Orlov in
\cite{LuntsOrlov10} which are based on the same condition as
ours. We refer to this article for a thorough explanation of
some important cases in which the Toda condition holds.
\end{remark}

Let us recall that Heller \cite{Heller97}, Franke \cite{Franke96},
Cisinski \cite{Cisinski08}, \ldots have shown universality properties
of many important derivators. The following result is proved by Jens
Franke in \cite{Franke96} in the case that $\Dia$ is the full
subcategory of $\Cat$ consisting of finite posets and by Alex
Heller in \cite{Heller97} when $\Dia=\Cat$.

\begin{theorem} \label{thm:spectra}
Let $\S p : \Dia^\circ \to \CAT$ be the {\em triangulated derivator
of spectra} associated with the homotopy category of finite spectra,
\ie, spectra with finitely many cells (more precisely with a localizer
defining this homotopy category). Let $\D : \Dia_\f^\circ \to \CAT$ be
a triangulated derivator. Then, we have the equivalence of categories
\[
\ul\HOM_!(\S p,\D) \Iso \D(e), \:\:\: F \mapsto (F_e)(S^0) .
\]
\end{theorem}
Here, the decoration $!$ indicates commutativity with all homotopy
Kan extensions. More precisely, we can refer to these as `cocontinuous
morphisms' (because this is the correct universal property if we drop
the finiteness assumptions). In the stable case and given our choice
of diagram categories, it follows by proposition 5.1 from the paper
\cite{PontoShulman16} by Ponto–-Shulman that such morphisms also
preserve all finite limits, and by the pointwise formula for right
Kan extensions also all right Kan extensions in this diagram category
(see \cite{Groth16}).

\section{Keller's towers} \label{s:tow}
We begin by defining a $2$-category that we call $\Cubes$.
For any positive integer $n$, let us denote $C_n$ the $n$-product
$\Delta_1 \times \Delta_1 \times \ldots \times \Delta_1$.
Clearly, $C_0$ is the terminal category $e = \{0\}$ and $C_2$ is the
diagram $\boxempty$. Sometimes, we will refer to these diagrams as
$n$-cubes. These diagrams are the objects of the $2$-category $\Cubes$.

Let us think of the objects $C_n$ as the partially ordered
$n$-dimensional cubes with side of length one, so that any vertex
$x$ is identified by an $n$-tuple $(x_1, \ldots, x_n)$ of numbers in
the set  $\{ 0, 1 \}$.
The morphisms in $\Cubes$ are all possible compositions of the
following order preserving maps. For any $n \in \N$,
\[
i^j_\varepsilon : C_n \lra C_{n+1}, \qquad (x_1, \dots, x_n)
\mapsto (x_1, \ldots, x_{j-1}, \varepsilon, x_j, \ldots, x_n) \ko
\]
\[
p^j : C_{n+1} \lra C_n, \qquad (x_1, \dots, x_n)
\mapsto (x_1, \ldots, x_{j-1}, x_{j+1}, \ldots, x_n) \ko
\]
where $\varepsilon \in \{ 0, 1 \}$. Clearly, there are relations
among these functors (see \cite{Keller91}).
The $2$-morphisms in $\Cubes$ are given by the following order relation :
for any pair of morphisms $u,v$ from $C_l$ to $C_m$, we write
$u \Rightarrow v$ if there is an arrow $u(x) \to v(x)$, for all $x \in C_l$.

Notice that $\Cubes$ is {\em not} a full $2$-subcategory of $\Dia_\f$.

\begin{definition}
A {\em Keller's tower of additive categories} or simply a {\em tower
of additive categories} $\sfe$ is a $2$-functor from the opposite
$2$-category $Cubes^\circ$ to the $2$-category of additive
categories $\ADD$.
\end{definition}

For any morphism $u : C_l \to C_m$ and any $2$-morphism
$\alpha : u \Rightarrow v$ in $Cubes$, we denote
$u^* : \sfe(C_m) \to \sfe(C_l)$ the induced morphism and
$\alpha^* : v^* \Rightarrow u^*$ the induced $2$-morphism,
as in the case of derivators.

Clearly, any additive (pre-)derivator $\E$ gives rise to a tower $\sfe$
when it is restricted to $\Cubes^\circ$. An important example of additive
tower is the restriction of the prederivator $\ul\ca$ defined in
\ref{ss:triader} to $\Cubes^\circ$.

If the category $\ca$ has an exact structure, then we can speak of a
{\em tower of exact categories} $\ul\ca$ by endowing each $\ul\ca(C_n)$
with the pairs of composable $2$-morphisms of functors (\ie, natural
transformations) whose evaluation at each $x \in C_n^\circ$ is a
conflation of $\ca$. Analogously, we can define the notion of {\em
tower of triangulated categories} as a contravariant $2$-functor from
$\Cubes$ to the $2$-category of triangulated categories $\TRIA$. As
an important example we can consider, for any arbitrary exact category
$\ca$, the triangulated tower defined by $\sfd_\ca : \Cubes^\circ
\to \CAT$, \ie, the tower that associates the bounded derived
category $\cd^b(\ul\Hom(C_n^\circ,\ca))$ with any $C_n$ in $\Cubes$.
With analogy to the derivator case, we can define the {\em morphisms
of towers} (called `towers of morphisms' in \cite{Keller91}) and their
compositions.

\begin{definition}
Let $\sfd$ and $\sfe$ be two towers. A {\em morphism of towers}
$F : \sfd \to \sfe$ consists of
\begin{itemize}
\item an additive functor $F_n : \sfd_n \to \sfe_n$, for each $n \in \N$;
\item an isomorphism of functors
\[
\varphi_u : F_m u^* \iso u^* F_n \ko
\]
for each morphism $u : C_m \to C_n$ in $\Cubes$;
\end{itemize}
such that
\begin{itemize}
\item $\varphi_{\id_n} = \id_{F_n}$, for each $n \in \N$;
\item $\varphi_{uv} = (v^* \varphi_u)(\varphi_v u^*)$,
for each pair of morphisms
$C_l \stackrel{v}{\to} C_m \stackrel{u}{\to} C_n$ in $\Cubes$;
\item $\varphi_u \alpha^* = \alpha^* \varphi_v$, for each $2$-morphism
$\xymatrix{C_m \rtwocell^u_v{_\alpha} & C_n}$ in $\Cubes$.
\end{itemize}
\end{definition}

The {\em composition} of two morphisms of towers $F : \sfd \to \sfe$
and $G : \sfe \to \sff$ is defined by setting $(GF)_n = G_n F_n$, for
all $n \in \N$, and $(\gamma\varphi)_u = (G_n \varphi_u)(\gamma_u F_m)$,
for all morphisms $u : C_m \to C_n$ in $\Cubes$.

If $F$ and $G$ are two morphisms of towers from $\sfd$ to $\sfe$,
a {\em $2$-morphism} $\alpha : F \to G$ is given by a functor
$\alpha_n : F_n \to G_n$, for each $n$ in $\N$, such that
$(u^* \alpha_n)\varphi_u = \gamma_u(\alpha_m u^*)$, for all
morphisms $u : C_m \to C_n$ of $\Cubes$. The {\em composition}
of two such morphisms is clear. Thus, the morphisms between any
fixed pair of towers $\sfd$ and $\sfe$ and the $2$-morphisms among
them are the objects and, respectively, the morphisms of an additive
category that we denote $\ul\HOM_{add}(\sfd,\sfe)$, or simply
$\ul\HOM(\sfd,\sfe)$.

It is clear how to define morphisms of exact towers (\resp, morphisms
of triangulated towers) and $2$-morphisms among them. Thus, for any pair
of exact (\resp, triangulated) towers $\sfd$ and $\sfe$, we get the
category $\ul\HOM_{ex}(\sfd,\sfe)$ (\resp, $\ul\HOM_{tr}(\sfd,\sfe)$).
A morphism of towers $F: \sfe \to \sff$ from a tower of exact
categories $\sfe$ to a tower of triangulated categories $\sff$ is
a {\em $\partial$-morphism} or {\em exact morphism of towers} from
$\sfe$ to $\sff$ if $F_n : \sfe_n \to \sff_n$, $n \in \N$, is a
sequence of $\partial$-functors. We denote $\ul\HOM_{ex}(\sfe,\sff)$
the category of these functors.

\begin{theorem}[Keller, \cite{Keller91}] \label{thm:kellerthm}
Let $\ca$ be an exact category and $\sfe$ a tower of triangulated
categories. Then, the canonical morphism $can : \ul\ca \to \sfd_\ca$
induces an equivalence of categories
\[
\ul\HOM_{tr}(\sfd_\ca,\sfe) \Iso \ul\HOM_{ex}(\ul\ca, \sfe) .
\]
\end{theorem}

Clearly, this is an important theorem. The aim of the present article
is to show that it also holds in the context of derivators.


\section{Epivalence and Recollement} \label{s:epirec}
In this section we want to show a very important property of
triangulated derivators and their morphisms. We will see that the
property of recollement, which is enjoyed by derivators by the
very definition, is important in this setting.

\begin{definition} \label{def:recollement}
A {\em recollement} of triangulated categories $\ct'$, $\ct$ and $\ct''$
is a diagram of triangulated functors
\[
\xymatrix{
\ct' \ar[rr]^{j_!} && \ct \ar[rr]^{i^*} \ar@/^1.5pc/[ll]^{j^*}
\ar@/_1.5pc/[ll]_{j^?} && \ct'' \ar@/^1.5pc/[ll]^{i_*}
\ar@/_1.5pc/[ll]_{i_!}
}
\]
such that
\begin{itemize}
\item the pairs $j^? \dashv j_!$, $j_! \dashv j^*$, $i_! \dashv i^*$,
$i^* \dashv i_*$ are adjunctions;
\item $i^*j_! = 0$ \, ;
\item the functors $i_*$, $i_!$ and $j_!$ are fully faithful;
\item for every object $X$ lying in $\ct$, there are distinguished
triangles
\[
\xymatrix{
i_!i^*X \ar[r]^<>(0.5){\varepsilon_i^X} & X \ar[r]^<>(0.5){\eta_j^X} &
j_!j^?X \ar[r]^{} & \Sigma i_!i^*X \ko
}
\]
\[
\xymatrix{
j_!j^*X \ar[r]^<>(0.5){\varepsilon_j^X} & X \ar[r]^<>(0.5){\eta_i^X} &
i_*i^*X \ar[r]^{} & \Sigma j_!j^*X \ko
}
\]
\end{itemize}
where ($\varepsilon_i^X$, $\eta_i^X$) and $(\varepsilon_j^X$, $\eta_j^X)$
are pairs of adjunction morphisms related to $i$ and $j$, respectively.

The functors $j^?$, $j_!$, $j^*$, $i_!$, $i^*$, $i_*$ are also known
as the {\em $6$-gluing functors}.
\end{definition}

Let us recall that in the context of derivators open / closed
inclusions of diagrams naturally give rise to recollements of
triangulated categories. This allows us constructing the shift
and loop autoequivalences in terms of the $6$-gluing functors
as in the following lemma that we state even if we are not
going to use it in what follows.

\begin{lemma} \label{lem:shift}
Let $\S$ be a triangulated derivator. Fix an arbitrary diagram
$I$ in $\Dia$, let $j : I \to I \times \Delta_1$
and $i : I \to I \times \Delta_1$ be the obvious open and closed
inclusions, respectively. Then, there are {\em canonical}
isomorphisms of functors
\[
j^?i_* \iso \Sigma_I \quad \ko \quad
\Omega_I = \Sigma_I^{-1} \iso i^!j_! \;\; .
\]
\end{lemma}

\begin{proof}
It is well known (\eg, look at \cite[Section 9]{CisinskiNeeman08})
that an open inclusion $j$ with a closed inclusion $i$ such that
the diagram spanned by the union of their images is all of $I \times
\Delta_1$ give rise to a recollement of triangulated categories
\[
\xymatrix{
\S(I) \ar[rr]^{j_!} && \S(I \times \Delta_1) \ar[rr]^{i^*}
\ar@/^1.5pc/[ll]^{j^*} \ar@/_1.5pc/[ll]_{j^?} && \S(I)
\ar@/^1.5pc/[ll]^{i_*} \ar@/_1.5pc/[ll]_{i_!} .
}
\]
Thus, for any object $X$ in $\S(I \times \Delta_1)$ there is a
distinguished triangle
\[
\xymatrix{
i_!i^*X \ar[r]^<>(0.5){} & X \ar[r]^<>(0.5){} &
j_!j^?X \ar[r]^<>(0.5){} & \Sigma_{I \times \Delta_1} i_!i^*X .
}
\]
Here, by applying the functor $\Hom(-,j_!j^?X)$ and passing to long
exact sequence, we see that the connecting morphism
$j_!j^?X \to \Sigma i_!i^*X$ is {\em unique} because
$\Hom(\Sigma i_!i^*X , j_!j^?X) = 0$ by adjunction and by recollement.

Let us apply the triangulated functor $j^*$ to this triangle.
Because of the recollement axioms we get a distinguished
triangle in $\S(I)$
\[
\xymatrix{
i^*X \ar[r]^<>(0.5){} & j^*X \ar[r]^<>(0.5){} &
j^?X \ar[r]^<>(0.5){} & \Sigma_I i^*X .
}
\]
In particular, when the object $X$ is of the form $i_*Y$, for
some object $Y$ in $\S(I)$, the triangle becomes
\[
\xymatrix{
i^*i_*Y \ar[r]^<>(0.5){} & j^*i_*Y \ar[r]^<>(0.5){} &
j^?i_*Y \ar[r]^<>(0.5){} & \Sigma_I i^*i_*Y .
}
\]
Here $i^*i_*Y = Y$ because the functor $i_*$ is fully faithful
and $j^*i_*Y = 0$ by adjunction. It survives a functorial iso
$j^?i_*Y \iso \Sigma_I Y$, for all $Y$ in $\S(I)$. The
conclusion follows easily.

The analogous proof of the second isomorphism starts by applying
the triangulated functor $i^!$ to the canonical distinguished triangle
\[
\xymatrix{
\Sigma^{-1} i_*i^*X \ar[r]^{} & j_!j^*X
\ar[r]^<>(0.5){\varepsilon_j^X} &
X \ar[r]^<>(0.5){\eta_i^X} & i_*i^*X
}
\]
and taking $X = j_*Y$.
\end{proof}

If we are willing to study commutativity of derivator morphisms
with the $6$-gluing functors we need another concept.

\begin{definition} \label{def:weakly}
Let $\ca$ be an exact category and let $\cs$, $\ct$ be triangulated
categories. We say that
\begin{itemize}
\item[a)] an additive functor $F : \cs \to \ct$ is
{\em weakly triangulated} if, for each distinguished triangle
\[
\xymatrix{
X \ar[r]^u & Y \ar[r]^v & Z \ar[r]^w & \Sigma X
}
\]
of $\cs$, there is {\em some} distinguished triangle
\[
\xymatrix{
F(X) \ar[r]^{F(u)} & F(Y) \ar[r]^{F(v)} & F(Z) \ar[r]^{w'} & \Sigma F(X)
}
\]
of $\ct$;
\item[b)] an additive functor $F : \ca \to \cs$ is a {\em weak
$\partial$-functor} or {\em weakly exact} if, for each conflation
\[
\xymatrix{
X \ar@{ >->}[r]^u & Y \ar@{->>}[r]^v & Z
}
\]
of $\ca$, there is {\em some} distinguished triangle
\[
\xymatrix{
F(X) \ar[r]^{F(u)} & F(Y) \ar[r]^{F(v)} & F(Z) \ar[r]^{z} & \Sigma F(X)
}
\]
of $\cs$.
\end{itemize}
\end{definition}

Clearly, the notion of weakly triangulated functor is much weaker
than the related notion of triangulated functor, which is described
in \cite{KellerVossieck87} under the name of $S$-functor.
Nevertheless, we will see that (Proposition \ref{prop:redundancy}),
in order to extend an additive morphism of triangulated derivators
to a triangulated morphism, it is redundant to ask that all the
functors $F_I$ are triangulated. We shall see that this extension
is possible if all these functors are {\em weakly} triangulated.
Indeed, even less is required for this happens, it is enough to
ask that $F_{\Delta_1}$ and $F_\square$ are {\em weakly}
triangulated functors.

Let $\S$ and $\T$ be triangulated derivators and $\ul\ca$ a
represented exact derivator. Let $F : \S \to \T$ and
$F : \ul\ca \to \S$ be {\em additive} functors of derivators.
Let us fix an arbitrary diagram $I$ in $\Dia$. Let
$j : I \to I \times \Delta_1$ and $i : I \to I \times \Delta_1$
be the obvious open and closed inclusions, respectively. We
remark that in this situation the natural transformation
(\cf subsection \ref{ss:mor})
$\varphi^j : j_! F_I \to F_{I \times \Delta_1} j_!$ is invertible.
Indeed, since every point inclusion $e \to I \times \Delta_1$
either factors through $i$ or $j$, it follows that $i^*$ and $j^*$
detect isomorphisms and we can easily check that
$i^*(\varphi^{j,X}) = 0$ and $j^*(\varphi^{j,X})$ are isomorphisms,
for all $X$ in $\S(I \times \Delta_1)$. 

Thus, there is an induced natural transformation
\[
\psi_j : j^? F_{I \times \Delta_1} \to F_{I} j^? \ko
\]
whose action $\psi_j^X$ on an object $X$ is given by the composition
\[
\xymatrix{
j^? F_{I \times \Delta_1}X
\ar[rr]^<>(0.5){j^? F_{I \times \Delta_1}(\eta_j^{X})} &&
j^? F_{I \times \Delta_1}j_!j^?X
\ar[rr]^<>(0.5){j^?(\varphi^{j,j^?X})^{-1}}_\sim && j^?j_!F_I j^?X
\ar[rr]^<>(0.5){\varepsilon_j^{F_I j^?X}}_\sim && F_I j^?X .
}
\]
Also, there is an induced natural transformation
\[
\psi_{jj} : j_!j^? F_{I \times \Delta_1}
\to F_{I \times \Delta_1} j_!j^? \;\; \ko
\]
whose action on every $X$ in $\S(I \times \Delta_1)$ is defined
by the composition
\[
\xymatrix{
j_!j^? F_{I \times \Delta_1} X \ar[rr]^<>(0.5){j_!(\psi_j^X)}
&& j_! F_I j^? X \ar[rr]^<>(0.5){\varphi^{j,j^?X}}_\sim &&
F_{I \times \Delta_1} j_!j^? X .
}
\]
\begin{lemma} \label{lem:diagramtimesarrow}
In the situation just described for the {\em additive} morphism
$F : \S \to \T$ of triangulated derivators, suppose that the
functor $F_{I \times \Delta_1}$ is {\em weakly} triangulated.
Then, for any $X$ in $\S(I \times \Delta_1)$, there is a
{\em unique} isomorphism
\[
\psi_{jj}^X := \varphi^{j,j^?X} \circ
\psi_j^X : j_!j^? F_{I \times \Delta_1}X
\iso F_{I \times \Delta_1} j_!j^? X
\]
such that the relation
$\psi_{jj}^X \circ \eta_j^{F_{I \times \Delta_1}X} =
F_{I \times \Delta_1} (\eta_j^X)$
holds. In particular, the $2$-morphism of functors
\[
\psi_j^X : j^? F_{I \times \Delta_1}X \to F_{I} j^?X
\]
is invertible, for all $X$ in $\S(I \times \Delta_1)$.
More generally, $\psi_{jj} : j_!j^? F_{I \times \Delta_1}
\iso F_{I \times \Delta_1} j_!j^?$ and
$\psi_j : j^? F_{I \times \Delta_1} \iso F_{I} j^?$ are
$2$-isomorphisms of functors.
\end{lemma}

\begin{proof}
For any object $X$ in $\S(I \times \Delta_1)$, let us start by
considering the canonical distinguished triangle that we have
just seen in the proof of Lemma \ref{lem:shift}
\[
\xymatrix{
i_!i^*X \ar[r]^<>(0.5){\varepsilon_i^X} & X \ar[r]^<>(0.5){\eta_j^X} &
j_!j^?X \ar[r]^<>(0.5){\partial^X} & \Sigma i_!i^*X \ko
}
\]
whose connecting morphism $\partial^X : j_!j^*X \to \Sigma i_!i^*X$
is {\em unique} because $\Hom(\Sigma i_!i^*X , j_!j^*X) = 0$ by
adjunction and by recollement.

Now apply the {\em weakly} triangulated functor $F_{I \times \Delta_1}$
and get a distinguished triangle in $\T(I \times \Delta_1)$
\[
\xymatrix{
F_{I \times \Delta_1} i_!i^*X \ar[r]^<>(0.5){F(\varepsilon)} &
F_{I \times \Delta_1} X \ar[r]^<>(0.5){F(\eta)} &
F_{I \times \Delta_1} j_!j^?X \ar[r]^<>(0.5){w} &
\Sigma F_{I \times \Delta_1} i_!i^*X \ko
}
\]
where we shortly write $F(\varepsilon)$ for
$F_{I \times \Delta_1}(\varepsilon_i^X)$, $F(\eta)$ for
$F_{I \times \Delta_1}(\eta_j^X)$. Recall that here the
morphism $w : F_{I \times \Delta_1} j_!j^?X \to
\Sigma F_{I \times \Delta_1} i_!i^*X$ just exists, \ie, it
is not unique, nor canonically constructed.

Let us consider the object $F_{I \times \Delta_1} X$ and write the
related distinguished triangle
\[
\xymatrix{
i_!i^* F_{I \times \Delta_1} X \ar[r]^<>(0.5){\varepsilon^F} &
F_{I \times \Delta_1} X \ar[r]^<>(0.5){\eta^F} &
j_!j^? F_{I \times \Delta_1} X \ar[r]^<>(0.5){\partial^F} &
\Sigma i_!i^* F_{I \times \Delta_1} X .
}
\]
Here we identify $\varepsilon^F = \varepsilon_i^{F_{I \times \Delta_1} X}$,
$\eta^F = \eta_j^{F_{I \times \Delta_1} X}$ and
$\partial^F = \partial^{F_{I \times \Delta_1} X}$ for the {\em unique}
connecting morphism. Since there is a natural transformation (\cf
subsection \ref{ss:mor})
\[
\varphi_{ii} \; := \; (\varphi^i i^*) \circ i_!(\varphi_i^{-1}) \; : \;
i_!i^* F_{I \times \Delta_1} \iso
i_! F_I i^* \to F_{I \times \Delta_1} i_!i^* \ko
\]
for every object $X$ in $\S(I \times \Delta_1)$ there are a morphism
$\varphi_{ii}^X = ((\varphi^i i^*) \circ i_!(\varphi_i^{-1}))^X$
and an induced morphism of distinguished triangles of
$\T(I \times \Delta_1)$
\[
\xymatrix{
i_!i^* F_{I \times \Delta_1} X \ar[d]^{\varphi_{ii}^X}
\ar[r]^<>(0.5){\varepsilon^F} & F_{I \times \Delta_1} X \ar@{=}[d]^\id
\ar[r]^<>(0.5){\eta^F} & j_!j^? F_{I \times \Delta_1} X \ar@{-->}[d]^\psi
\ar[r]^<>(0.5){\partial^F} & \Sigma i_!i^* F_{I \times \Delta_1} X
\ar[d]^{\Sigma\varphi_{ii}^X} \\
F_{I \times \Delta_1} i_!i^* X \ar[r]^<>(0.5){F(\varepsilon)} &
F_{I \times \Delta_1} X \ar[r]^<>(0.5){F(\eta)} &
F_{I \times \Delta_1} j_!j^? X \ar[r]^<>(0.5){w} &
\Sigma F_{I \times \Delta_1} i_!i^* X .
}
\]
Indeed, let us check commutativity of the square on the left.
We have
\begin{eqnarray*}
F_{I \times \Delta_1}(\varepsilon_i^X) \circ \varphi_{ii}^X
& = &
F_{I \times \Delta_1}(\varepsilon_i^X) \circ
\varepsilon_i^{F_{I \times \Delta_1}i_!i^*X} \circ
i_!(\varphi_i^{i_!i^*X}) \circ i_!F_I(\eta_i^{i^*X}) \circ
i_![(\varphi_i^X)^{-1}] \\
& = &
\varepsilon_i^{F_{I \times \Delta_1}X} \circ
i_!i^*F_{I \times \Delta_1}(\varepsilon_i^X) \circ
i_!(\varphi_i^{i_!i^*X}) \circ i_!F_I(\eta_i^{i^*X}) \circ
i_![(\varphi_i^X)^{-1}] \\
& = &
\varepsilon_i^{F_{I \times \Delta_1}X} \circ i_!(\varphi_i^X) \circ
i_!F_Ii^*(\varepsilon_i^X) \circ i_!F_I(\eta_i^{i^*X}) \circ
i_![(\varphi_i^X)^{-1}] \\
& = &
\varepsilon_i^{F_{I \times \Delta_1}X} \circ i_!(\varphi_i^X) \circ
i_![(\varphi_i^X)^{-1}] \\
& = &
\varepsilon_i^{F_{I \times \Delta_1}X} .
\end{eqnarray*}
Here, every canonical isomorphism comes from functoriality,
as $\varphi_i$ is a $2$-isomorphism of functors, and from the
relation $i^*(\varepsilon_i^X) \circ \eta_i^{i^*X} = \id^{i^*X}$.

Let us remark the useful fact that, if we call $p$ the obvious
projection functor $I \times \Delta_1 \to I$, then there are the
adjunction morphisms $j \dashv p \dashv i$. This fact implies
that $j^* \dashv p^* \dashv i^*$ also are adjunction morphisms.
In particular, we get the equality $i_! = p^*$, which entails
that the canonical morphisms
$\varphi := \varphi_{ii}^X = \varphi_{i \circ p}^X :
(i \circ p)^* F_{I \times \Delta_1}X \to
F_{I \times \Delta_1} (i \circ p)^* X$
and $\Sigma\varphi$ are invertible.

Thus, we have shown that $\psi$ is an isomorphism.
Moreover, it is unique because we can see that the group
\begin{eqnarray*}
\Hom(\Sigma F_{I \times \Delta_1} i_!i^* X ,
j_!j^? F_{I \times \Delta_1} X)
& = &
\Hom(\Sigma i_!i^* F_{I \times \Delta_1} X ,
j_!j^? F_{I \times \Delta_1} X) \\
& = &
\Hom(i_!i^* \Sigma F_{I \times \Delta_1} X ,
j_!j^? F_{I \times \Delta_1} X) \\
& = &
\Hom(i^* \Sigma F_{I \times \Delta_1} X ,
i^*j_!j^? F_{I \times \Delta_1} X) \\
& = &
0
\end{eqnarray*}
is canonically trivial via the canonical isomorphism $\varphi$,
adjunction and the canonical isomorphism $i^*j_! = 0$, thanks
to the recollement axioms.

Therefore, if we are able to show that the morphism $\psi_{jj}^X$
makes the central square of the diagram commutative too, we
get $\psi = \psi_{jj}^X$ canonically and the statement follows.
Indeed, there is a commutative diagram, by functoriality of $\eta_j$,
\[
\xymatrix{
j_!j^? F_{I \times \Delta_1} j_!j^? X &&
j_!j^?j_!F_{I} j^? X
\ar[ll]_<>(0.5){j_!j^?(\varphi^{j,j^?X})}^\sim \\
F_{I \times \Delta_1} j_!j^? X
\ar[u]^<>(0.5){\eta_{j}^{F_{I \times \Delta_1} j_!j^? X}} &&
j_! F_{I} j^? X \ar[ll]^<>(0.5){\varphi^{j,j^?X}}
\ar[u]_<>(0.5){\eta_j^{j_! F_{I} j^? X}}^\sim
\ar[ll]^<>(0.5){\varphi^{j,j^?X}}_\sim .
}
\]
We already know that the three arrows labelled with the
symbol $\sim$ are isomorphisms. Therefore, the arrows
$\eta_{j}^{F_{I \times \Delta_1} j_!j^? X}$ and
$j_!(\eps_j^{F_Ij^?X}) : j_!j^?j_!F_{I} j^? X \to j_! F_{I} j^? X$
are invertible, too. In particular, we get that the composition
$\varphi^{j,j^?X} \circ j_!(\eps_j^{F_Ij^?X}) \circ
j_!j^?(\varphi^{j,j^?X})^{-1}$ gives us an inverse to
$\eta_{j}^{F_{I \times \Delta_1} j_!j^? X}$. Now, since the
relation $j_!j^?F_{I \times \Delta_1}(\eta_j^X) \circ
\eta_j^{F_{I \times \Delta_1}X} =
\eta_j^{F_{I \times \Delta_1} j_!j^? X} \circ
F_{I \times \Delta_1}(\eta_j^X)$
holds, it follows that
$\psi_{jj}^X \circ \eta_j^{F_{I \times \Delta_1}X} =
F_{I \times \Delta_1}(\eta_j^X)$
also holds.

Since we have shown that $\psi_{jj}^X$ is invertible, we get
that the morphism $j_!(\psi_j^X)$ must be invertible, too.
Now apply $j^*$ to see that $\psi_j^X$ is an isomorphism.
The claim follows.
\end{proof}

We need a similar lemma for morphisms defined over {\em exact}
derivators.

\begin{lemma} \label{lem:exdiagramtimesarrow}
Let $\ul\ca$ be the prederivator represented by an exact
category $\ca$ and let $\S$ be a triangulated derivator. Let
$F : \ul\ca \to \S$ be an {\em additive} morphism of derivators
such that the functor $F_{I \times \Delta_1}$ is {\em weakly}
exact. Consider an arbitrary object $X$ in
$\ul\ca(I \times \Delta_1)$ such that the `vertical parallel arrows'
of his diagram are deflations.
More precisely, this means that $i_k^*i^*X \to i_k^*j^*X$ is a
deflation, for all $k \in I$ (here, $i$ and $j$ are as in lemma
\ref{lem:diagramtimesarrow} and $i_k : e \to I$ is the obvious map).

Then, there is a {\em unique} isomorphism
\[
\psi_{jj}^X := \varphi^{j,j^?X} \circ
\psi_j^X : j_!j^? F_{I \times \Delta_1}X
\iso F_{I \times \Delta_1} j_!j^? X
\]
such that the relation
$\psi_{jj}^X \circ \eta_j^{F_{I \times \Delta_1}X} =
F_{I \times \Delta_1} (\eta_j^X)$
holds. In particular, the natural transformation of functors
\[
\psi_j^X : j^? F_{I \times \Delta_1}X \to F_{I} j^?X
\]
is invertible. More generally, when restricted on the
subcategory of squares as above,
$\psi_{jj} : j_!j^? F_{I \times \Delta_1} \iso
F_{I \times \Delta_1} j_!j^?$ and $\psi_j :
j^? F_{I \times \Delta_1} \iso F_{I} j^?$ are $2$-isomorphisms
of functors.
\end{lemma}

\begin{proof}
The basic category $\ca = \ul\ca(e)$ is additive and has a
zero object $0$. By \cite[Cor. 3.8]{Groth11} our prederivator
$\ul\ca$ is {\em pointed} according to our definition
\ref{def:pointed}, \ie, the adjoint (exact) functors
$i^!$ and $i^?$ of axiom {\bf Der 6} exist. Notice that
it is enough that a prederivator is pointed in order to define
the suspension and the loop endofunctors $\Sigma_I$ and
$\Omega_I$, for all the diagrams $I$ in $\Dia$ (\cf \cite{Groth11}).

Let us consider an arbitrary object $X$ of $\ul\ca(I \times \Delta_1)$
with the required property. The adjunction morphism
\[
\xymatrix{
i_!i^*X \ar[r]^<>(0.5){\eps_i^X} & X
}
\]
is a deflation. Indeed, it is clear that
$i_k^*i^*(i_!i^*X) = i_k^*i^*X$, for all $k \in I$. Moreover,
we can locally check that
\[
i_k^*j^*(i_!i^*X) = i_k^*j^*p^*i^*X =
(i \; p \; j \; i_k)^*X = i_k^*i^*X \ko
\]
for all $k \in I$. This means that the diagram of the object
$i_!i^*X$ contains two identical horizontal subdiagrams, linked
by vertical parallel identities. It follows that the arrow
$i_k^*i^*(\eps_i^X)$ is an identity, for all $k \in I$, and
that $i_k^*j^*(\eps_i^X)$ is the original deflation
$i_k^*i^*X \to i_k^*j^*X$, for all $k \in I$.

Since the morphism $\eps_i^X$ is a deflation, it must fit
into a conflation
\[
\xymatrix{
\kernel(\eps_i^X) \ar@{ >->}[r] & i_!i^*X \ar@{->>}[r] & X \ko
}
\]
whose image in $\D_\ca(I \times \Delta_1)$ under the exact
functor $can_{I \times \Delta_1}$ we already know to fit into
a {\em canonical} distinguished triangle
(\cf lemma \ref{lem:diagramtimesarrow})
\[
\xymatrix{
\Sigma^{-1}j_!j^?X \ar[r] & i_!i^*X \ar[r] & X \ar[r] & j_!j^?X \; .
}
\]
Uniqueness allows us canonically identifying the object
$\kernel(\eps_i^X)$ with the loop object $\Omega j_!j^*X$
(\cf def. 3.19 in \cite{Groth11}).

Let us apply the {\em weakly} exact functor
$F_{I \times \Delta_1}$ and shift to get a distinguished
triangle in $\S(I \times \Delta_1)$
\[
\xymatrix{
F_{I \times \Delta_1} i_!i^*X \ar[r]^<>(0.5){} &
F_{I \times \Delta_1} X \ar[r]^<>(0.5){} &
F_{I \times \Delta_1} j_!j^?X \ar[r]^<>(0.5){} &
\Sigma F_{I \times \Delta_1} i_!i^*X \; .
}
\]
At this point we have reached the triangulated world and the
proof goes on as in the proof of lemma \ref{lem:diagramtimesarrow}.
\end {proof}

\subsection{The redundancy of the connecting morphism}
\label{ss:redundancy}
Let us start by showing that our definition in item a) (\resp, b)),
in \ref{def:extria} of exact (\resp, triangulated), morphism is the
correct one since the involved (pre)derivators enjoy the property
in the `epivalence axiom', \ie, Axiom {\bf Der 5}. We intend to
show that, if $F : \S \to \T$ is a triangulated morphism of
triangulated derivators, then there are autoequivalences $\Sigma_I^\S$
and $\Sigma_I^\T$ at each diagram $I$ such that, canonically, the
induced functor $F_I : \S(I) \to \T(I)$ is a triangulated functor,
for all diagrams $I$. This means that there is a canonical
$2$-isomorphism of triangulated functors
$\delta_I : F_I \Sigma_I^\S \iso \Sigma_I^\T F_I$, for each
diagram $I$. Altogether, all these morphisms define a canonical
$2$-isomorphism $\delta : F \Sigma \iso \Sigma F$.

\begin{proposition} \label{prop:rightdef}
Let $\ul\ca$ be the represented exact prederivator associated to an
exact category $\ca$. Let $\S$ and $\T$ be triangulated derivators.
Let us suppose that all these derivators be of some type $\Dia$.
\begin{itemize}
\item[a)] Let $F : \ul\ca \to \S$ be a $\partial$-morphism according
to Def.~\ref{def:extria}, item a). Then, for all diagram $I$ in $\Dia$
the local functor $F_I : \ul\ca(I) \to \S(I)$ has a canonical structure
of $\partial$-functor with respect to the exact structure of $\ul\ca(I)$
and to the canonical triangulated structure of $\T(I)$.
\item[b)] Let $F : \S \to \T$ be a triangulated morphism according
to Def.~\ref{def:extria}, item b). Then, for all diagram $I$ in $\Dia$
the local functor $F_I : \S(I) \to \T(I)$ has a canonical structure
of triangulated functor with respect to the canonical triangulated
structures of the categories $\S(I)$ and $\T(I)$.
\end{itemize}
\end{proposition}

\begin{proof}
a) We begin with an exact morphism of prederivators $F : \ul\ca \to \S$
as in the hypothesis. We have to show that the functor
$F_I : \ul\ca(I) \to \S(I)$ is exact, for all diagrams $I$. This
means that, given a conflation $\eps$ in $\ul\ca(I) = \ul\ca^I(e)$
\[
\xymatrix{
X \ar@{ >->}[r]^<>(0.5)i & Y \ar@{->>}[r]^<>(0.5)d & Z \ko
}
\]
there is a natural transformation $\delta_I$ which assigns a morphism
$\delta_I\eps : F_IZ \to \Sigma_IF_IX$ in such a way that
\begin{equation} \label{eqn:F(eps)}
\xymatrix{
F_IX \ar[r]^<>(0.5){F(i)} & F_IY \ar[r]^<>(0.5){F(d)} & F_IZ
\ar[r]^<>(0.5){\delta_I\eps} & \Sigma_IF_IX
}
\end{equation}
is a distinguished triangle of $\S(I)$. Altogether, all these
natural transformations define a canonical $\partial$-morphism
$\delta$.

Associated with the conflation $\eps$ there is a bicartesian square
\[
\xymatrix{
X \ar@{ >->}[r] \ar@{->>}[d] & Y \ar@{->>}[d] \\
0 \ar@{ >->}[r] & Z
}
\]
lying in $\ul\Hom(\boxempty^\circ , \ul\ca^I(e))$. Since we can
write this category as $\ul\Hom(\Delta_1^\circ , \ul\ca^I(\Delta_1))$,
thanks to the axiom {\bf Der 5} there exists a bicartesian
object $S$ in $\A^I(\boxempty)$ whose diagram is the square above.

Clearly, we have a canonical iso $\eps^S : \il_!\il^* S \iso S$.
For the sake of simplicity in the notation, here we write
$i_\il$, $i_\square$, etc., instead of the correct ones
$i^I_\il = i_{I \times \il}$, $i^I_\square = i_{I \times \square}$,
etc. Since the morphism $F$ is supposed to be exact, we get an
object $F^I_\square S$ in $\S^I(\boxempty)$ having a commutative
square as diagram
\begin{eqnarray*}
\begin{matrix}
\\\\\\
\dia^I_\square(F^I_\square S) = F^I_e(\dia^I_\square S)
\end{matrix}
&
\begin{matrix}
\\\\\\
\mbox{=}
\end{matrix}
&
\xymatrix{
F^I_e X \ar[r] \ar[d] & F^I_e Y \ar[d] \\
0 \ar[r] & F^I_e Z
}
\begin{matrix}
\\\\\\
\mbox{,}
\end{matrix}
\end{eqnarray*}
such that
$\eps^{F^I_\square S} : \il_!\il^* F^I_\square S \to F^I_\square S$
is invertible. Our local claim is that there exists a
distinguished triangle in $\S^I(e)$
\[
\xymatrix{
F^I_e X \ar[r] & F^I_e Y \ar[r] & F^I_e Z \ar[r]^<>(0.5){\delta^I_e\eps} &
\Sigma F^I_e X \ko
}
\]
for some morphism $\delta^I_e\eps : F^I_e Z \to \Sigma F^I_e X$,
functorial in $\eps$, \ie, that the functor $F^I_e$ is exact.

Let us fix some notation. We write $\twosquare$ for the diagram
$\Delta_2 \times \Delta_1$ and $\squarearrow$ for the diagram
$\twosquare$ after we erase the object $(1,1)$. There are obvious
fully faithful inclusions $i_\square : \square \to \squarearrow$,
$\il : \lefthalfcap \to \square$ and
$i_\squarearrow : \squarearrow \to \twosquare$, injective
on objects. Moreover, we define other inclusions
$l_\square : \square \to \twosquare$ and
$r_\square : \square \to \twosquare$ defined by the evident
overlap of the image of the square over the small squares
on the left and on the right of the two-square, respectively.
These functors further induce other inclusions
$l_\lefthalfcap : \lefthalfcap \to \twosquare$
and $r_\lefthalfcap : \lefthalfcap \to \twosquare$, which
are the compositions $l_\square i_\lefthalfcap$ and
$r_\square i_\lefthalfcap$, respectively. We also consider the
`global' inclusion $g_\square : \square \to \twosquare$ mapping
the square on the exterior bord of the two-square and the inclusion
$i_{\lefthalfcap \vir \squarearrow} : \lefthalfcap \to \squarearrow$
mapping the diagram $\lefthalfcap$ to the right corner of the
diagram $\squarearrow$.

Let us define the object
$P := {i_\squarearrow}_! {i_\square}_* F^I_\square S$. This is
a {\em polycartesian} object of $\S^I(\twosquare)$ in the sense
of Maltsiniotis \cite{Maltsiniotis07}, \ie, there are isomorphisms
\begin{center}
$\eps^{{l_\square}^*P} : \il_!\il^* {l_\square}^*P \iso {l_\square}^*P$, \;\;\;
$\eps^{{r_\square}^*P} : \il_!\il^* {r_\square}^*P \iso {r_\square}^*P$, \;\;\;
$\eps^{{g_\square}^*P} : \il_!\il^* {g_\square}^*P \iso {g_\square}^*P$.
\end{center}
It suffices to check the first and the second isos. We compute
\begin{eqnarray*}
\il_!\il^* {l_\square}^*P
& = &
\il_!\il^* {i_\square}^* {i_\squarearrow}^* {i_\squarearrow}_!
{i_\square}_* F^I_\square S \\
& = &
\il_!\il^* {i_\square}^* {i_\square}_* F^I_\square S \\
& = &
\il_!\il^* F^I_\square S \\
& \iso &
F^I_\square S \\
& = &
{i_\square}^* {i_\square}_* F^I_\square S \\
& = &
{i_\square}^* {i_\squarearrow}^* {i_\squarearrow}_!
{i_\square}_* F^I_\square S \\
& = &
{l_\square}^*P .
\end{eqnarray*}
Here we use the relation $l_\square = i_\squarearrow i_\square$
in the first and the last equalities. The crucial iso
$\eps^{F^I_\square S}$ is the fourth, which is our
hypothesis. As for the square on the right, let us compute
\begin{eqnarray*}
\il_!\il^* {r_\square}^*P
& = &
\il_!\il^* {r_\square}^* {i_\squarearrow}_! {i_\square}_* F^I_\square S \\
& = &
\il_! {r_\lefthalfcap}^* {i_\squarearrow}_! {i_\square}_* F^I_\square S \\
& = &
\il_! {i_{\lefthalfcap \vir \squarearrow}}^*
{i_\squarearrow}^* {i_\squarearrow}_! {i_\square}_* F^I_\square S \\
& = &
\il_! {i_{\lefthalfcap \vir \squarearrow}}^* {i_\square}_* F^I_\square S \\
& \iso &
{r_\square}^* {i_\squarearrow}_! {i_\square}_* F^I_\square S \\
& = &
{r_\square}^*P .
\end{eqnarray*}
Here, the only non trivial iso comes from the invertible
natural transformation
$\il_! {i_{\lefthalfcap \vir \squarearrow}}^* \iso
{r_\square}^* {i_\squarearrow}_!$. To see why this transformation
is invertible, let us consider the co-cartesian square in the
category $\Dia$
\[
\xymatrix{
\lefthalfcap \ar[r]^<>(0.5){\il}
\ar[d]_{i_{\lefthalfcap \vir \squarearrow}} &
\square \ar[d]^{r_\square} \\
\squarearrow \ar[r]_{i_\squarearrow} & \twosquare .
}
\]
By the ``commentaires" after the axiom {\bf Der 7} in
\cite{Maltsiniotis07} or, in greater detail, by the dual of
Prop. 6.9 in \cite[Prop. 6.9]{CisinskiNeeman08} the required
isomorphism follows. The local claim follows by the description
of the triangulated structure over $\S^I(e)$ as, \eg, in
\cite{Maltsiniotis07}. Indeed, the diagram of the object $P$
is
\begin{eqnarray*}
\begin{matrix}
\\\\\\
\dia^I_\twosquare(P)
\end{matrix}
&
\begin{matrix}
\\\\\\
\mbox{=}
\end{matrix}
&
\xymatrix{
F^I_e X \ar[r] \ar[d] & F^I_e Y \ar[d] \ar[r] & 0 \ar[d] \\
0 \ar[r] & F^I_e Z \ar[r]_{f_{2}} & P^I_{12}
}
\begin{matrix}
\\\\\\
\mbox{,}
\end{matrix}
\end{eqnarray*}
and there is an isomorphism $\theta^I : P^I_{12} \iso \Sigma^I_e F^I_e X$
which induces a {\em standard} distinguished triangle in $\S^I(e)$
\[
\xymatrix{
F^I_e X \ar[r] & F^I_e Y \ar[r] & F^I_e Z \ar[r]^<>(0.5){\delta^I_e\eps} &
\Sigma^I_e F^I_e X \ko
}
\]
where the morphism $\delta^I_e\eps : F^I_e Z \to \Sigma F^I_e X$
is given by the composition $\theta^I \circ f_2$. Thus, this
triangle provides a canonical construction of the distinguished
triangle \myref{eqn:F(eps)}. \\

It remains to show the functoriality of the construction.
So, let us consider an exact $2$-morphism $\mu : F \to F'$, where
$F' : \ul\ca \to \S$ is another exact morphism of prederivators.
We desire to show that the functor $\mu_I : F_I \to F'_I$
actually is a natural transformation of $\partial$-functors.
Indeed, given a conflation $\eps$ in $\ul\ca(I) = \ul\ca^I(e)$
\[
\xymatrix{
X \ar@{ >->}[r] & Y \ar@{->>}[r] & Z \ko
}
\]
thanks to the axiom {\bf Der 5} there is a bicartesian object
$S$ in $\A^I(\boxempty)$ whose diagram is as above. Again,
we have a canonical iso $\eps^S : \il_!\il^* S \iso S$.
We already know that there are isomorphisms
$\eps^{F^I_\square S} : \il_!\il^* F^I_\square S \iso
F^I_\square S$ and $\eps^{F'^I_\square S} :
\il_!\il^* F'^I_\square S \iso F'^I_\square S$ because
$F$ and $F'$ are exact.

Since the $2$-morphism of prederivators $\mu : F \to F'$ is
supposed to be exact, it induces a morphism
${\mu^I_\square}^S : F^I_\square S \to F'^I_\square S$
in $\S^I(\boxempty)$, whose diagram in $\S^I(e)$ is
\begin{eqnarray*}
\begin{matrix}
\\\\\\
\dia^I_\square({\mu^I_\square}^S) = {\mu^I_e}^{\dia^I_\square S}
\end{matrix}
&
\begin{matrix}
\\\\\\
\mbox{=}
\end{matrix}
&
\xymatrix@!0{
F^I_eX \ar[rr] \ar[dd] \ar[dr] && F^I_eY \ar'[d][dd] \ar[dr] \\
& 0 \ar[rr] \ar[dd] && F^I_eZ \ar[dd] \\
F'^I_eX \ar'[r][rr] \ar[dr] && F'^I_eY \ar[dr] \\
& 0 \ar[rr] && F'^I_eZ
}
\begin{matrix}
\\\\\\
\mbox{\ko}
\end{matrix}
\end{eqnarray*}
such that the morphism
$\il_!\il^*{\mu^I_\square}^S \iso {\mu^I_\square}^S$ is
invertible (\cf item c) of Def. \ref{def:extria}).

Our local claim is that there exists a morphism of distinguished
triangles in $\S^I(e)$
\[
\xymatrix{
F^I_e X \ar[r] \ar[d]_{{\mu^I_e}^X} &
F^I_e Y \ar[r] \ar[d]_{{\mu^I_e}^Y} &
F^I_e Z \ar[r]^<>(0.5){\delta^I_e\eps} \ar[d]^{{\mu^I_e}^Z} & \Sigma F^I_e X
\ar[d]^<>(0.5){\Sigma{\mu^I_e}^X} \\
F'^I_e X \ar[r] & F'^I_e Y \ar[r] &
F'^I_e Z \ar[r]^<>(0.5){\delta'^I_e\eps} & \Sigma F'^I_e X .
}
\]

As above, let us define the polycartesian objects
$P := {i_\squarearrow}_! {i_\square}_* F^I_\square S$ and
$P' := {i_\squarearrow}_! {i_\square}_* F'^I_\square S$.
We also define a morphism $\mu^P : P \to P'$ to be
${i_\squarearrow}_! {i_\square}_* {\mu^I_\square}^S$. This
morphism is {\em polycartesian} in the following sense:
there are isomorphisms
\begin{center}
$\eps^{{l_\square}^*\mu^P} :
\il_!\il^* {l_\square}^*\mu^P \iso
{l_\square}^*\mu^P$,
$\eps^{{r_\square}^*\mu^P} :
\il_!\il^* {r_\square}^*\mu^P \iso
{r_\square}^*\mu^P$,
$\eps^{{g_\square}^*\mu^P} :
\il_!\il^* {g_\square}^*\mu^P \iso
{g_\square}^*\mu^P$.
\end{center}
It suffices to check the first and the second isos. We compute
\begin{eqnarray*}
\il_!\il^* {l_\square}^*{\mu^P}
& = &
\il_!\il^* {i_\square}^*
{i_\squarearrow}^* {i_\squarearrow}_!
{i_\square}_* {\mu^I_\square}^S \\
& = &
\il_!\il^* {i_\square}^* {i_\square}_* {\mu^I_\square}^S \\
& = &
\il_!\il^* {\mu^I_\square}^S \\
& \iso &
{\mu^I_\square}^S \\
& = &
{i_\square}^* {i_\square}_* {\mu^I_\square}^S \\
& = &
{i_\square}^* {i_\squarearrow}^* {i_\squarearrow}_!
{i_\square}_* {\mu^I_\square}^S \\
& = &
{l_\square}^*{\mu^P} .
\end{eqnarray*}
Here we use the relation $l_\square = i_\squarearrow i_\square$
in the first and the last equalities. The crucial iso
$\il_!\il^* {\mu^I_\square}^S \iso {\mu^I_\square}^S$ is the
fourth, which is our hypothesis.
As for the square on the right, let us compute
\begin{eqnarray*}
\il_!\il^* {r_\square}^*\mu^P
& = &
\il_!\il^* {r_\square}^* {i_\squarearrow}_!
{i_\square}_* {\mu^I_\square}^S \\
& = &
\il_! {r_\lefthalfcap}^* {i_\squarearrow}_!
{i_\square}_* {\mu^I_\square}^S \\
& = &
\il_! {i_{\lefthalfcap \vir \squarearrow}}^*
{i_\squarearrow}^* {i_\squarearrow}_!
{i_\square}_* {\mu^I_\square}^S \\
& = &
\il_! {i_{\lefthalfcap \vir \squarearrow}}^*
{i_\square}_* {\mu^I_\square}^S \\
& \iso &
{r_\square}^* {i_\squarearrow}_! {i_\square}_*
{\mu^I_\square}^S \\
& = &
{r_\square}^*\mu^P .
\end{eqnarray*}
Here, the isos we are using are in complete analogy with those
at the corresponding point in the proof of the object case.
The local claim follows by the description of the triangulated
structure over $\S^I(e)$ as, \eg, in \cite{Maltsiniotis07} and
an analogous construction to the one we have just done for the
object case that we leave to the interested reader. Altogether,
these natural transformations of functors $\mu_I : F^I \to F'^I$
clearly define an exact $2$-morphism of $\partial$-morphisms
$\mu : F \to F'$. \\

b) Since the statement and the proof are analogous to those in
item a) we just sketch out the proof of this case.

The main difference is that we start with a triangulated morphism
$F : \S \to \T$ of triangulated derivators. Suppose we are given
a distinguished triangle with respect to the {\em canonical}
triangulated structure of the category $\S(I)$,
\[
\xymatrix{
X \ar[r]^f & Y \ar[r]^g & Z \ar[r]^<>(0.5){\eps} &
\Sigma_I X \ko
}
\]
This means that, up to iso, there is a polycartesian object $P$
in $\S(\twosquare \times I)$ whose diagram is
\begin{eqnarray*}
\begin{matrix}
\\\\\\
\dia^I_\twosquare(P)
\end{matrix}
&
\begin{matrix}
\\\\\\
\mbox{=}
\end{matrix}
&
\xymatrix{
X \ar[r] \ar[d] & Y \ar[d] \ar[r] & 0 \ar[d] \\
0 \ar[r] & Z \ar[r]_{f_{2}} & P^I_{12}
}
\begin{matrix}
\\\\\\
\mbox{,}
\end{matrix}
\end{eqnarray*}
with an isomorphism $\theta^I : P^I_{12} \iso \Sigma^I X$.
From now on the proof is the same as in item a) by considering the
polycartesian object $F_I(P)$ instead of the previous one $P$.
The construction provides us with an invertible natural
transformation $\delta_I : F_I\Sigma_I \iso \Sigma_IF_I$.
Altogether, all these transformations define a canonical
triangulated $2$-isomorphism $\delta : F \Sigma \iso \Sigma F$.
\end{proof}

Clearly, we may forget the property of a morphism of derivators of
being triangulated. Consequently, we also forget the {\em canonical}
$2$-morphism $\delta$ and think of the underlying additive structure
of categories and morphisms only. We indicate the related forgetful
functor by means of the vertical bar `$|$'.

\begin{proposition} \label{prop:redundancy}
Let $\ul\ca$ be the exact prederivator represented by an exact
category $\ca$. Let $\S$ and $\T$ be triangulated derivators.
Let us suppose that all these (pre)derivators be of some type $\Dia$.
\begin{itemize}
\item[a)] The forgetful functor
\[
\ul\HOM_{tr}(\S,\T) \to \ul\HOM_{add}(\S|,\T|)
\]
is an isomorphism onto the full subcategory consisting of the
additive morphisms of derivators $F$ such that
$F_{\Delta_1} : \S(\Delta_1) \to \T(\Delta_1)$ and
$F_\square : \S(\boxempty) \to \T(\boxempty)$ are
{\em weakly} triangulated functors with respect to the canonical
triangulated structures of these categories.
\item[b)] The forgetful functor
\[
\ul\HOM_{\partial}(\ul\ca,\S) \to \ul\HOM_{add}(\ul\ca|,\S|)
\]
is an isomorphism onto the full subcategory consisting of the
additive morphisms of derivators $F$ such that
$F_{\Delta_1} : \ul\ca(\Delta_1) \to \S(\Delta_1)$
and $F_\square : \ul\ca(\boxempty) \to \S(\boxempty)$ are
{\em weak} $\partial$-functors with respect to the canonical
additive and triangulated structures of these categories.
\end{itemize}
\end{proposition}

\begin{proof}
Let us note that we already know from proposition \ref{prop:rightdef}
that the forgetful functors in the statements a) and b) actually take
a triangulated (\resp, exact) morphism of (pre)derivators $F$ to an
additive morphism $F|$ that is locally triangulated (\resp, locally exact),
and, a fortiori, locally weakly triangulated (\resp, locally weakly
exact). Moreover, since we know from remark \ref{rmk:all} that our
forgetful functors are fully faithful, it remains to check their
essential surjectivity only.

a) Suppose we are given an arbitrary additive morphism
$F : \S \to \T$ such that $F_{\Delta_1}$ and $F_\square$ are
weakly triangulated functors. We have to check that it preserves
(homotopy) bicartesian objects according to the condition in the
item b) of Definition \ref{def:extria}. Since the derivators $\S$
and $\T$ have the {\em property} of being triangulated, thanks to
the axiom {\bf Der 7} it is sufficient to check that $F$ preserves
(homotopy) cocartesian objects.

Let us denote by $ab$ the objects in the diagrams $\lefthalfcap$,
$\righthalfcup$ and $\boxempty$ whose coordinates are $(a,b)$.
Accordingly, we denote $i_{ab,\lefthalfcap}$, $i_{ab,\righthalfcup}$
and $i_{ab,\square}$ the functors defined by the object
$ab$ having coordinates $(a,b)$, which send the diagram $e$ into
the diagrams $\lefthalfcap$, $\righthalfcup$ and $\boxempty$,
respectively. We simply write $i_{ab}$ when there is no
possibility of confusion.

Remark that the pair of open / closed immersions
$\xymatrix{e \ar@{ (->}[r]^{i_{11}} & \boxempty &
\lefthalfcap \ar@{ (->}[l]_{\il}}$
gives rise to a recollement of triangulated categories
\[
\xymatrix{
\S(e) \ar[rr]^{{i_{11}}_!} && \S(\boxempty) \ar[rr]^{\il^*}
\ar@/^1.5pc/[ll]^{{i_{11}}^*} \ar@/_1.5pc/[ll]_{{i_{11}}^?} &&
\S(\lefthalfcap) \ar@/^1.5pc/[ll]^{\il_*} \ar@/_1.5pc/[ll]_{\il_!} .
}
\]
Thus, for any object $X$ in $\S(\boxempty)$, there is a
distinguished triangle
\[
\xymatrix{
\il_!\il^*X \ar[r]^<>(0.5){\varepsilon} & X \ar[r]^<>(0.5){\eta} &
{i_{11}}_!{i_{11}}^?X \ar[r]^<>(0.5){\partial} &
\Sigma_\square \il_!\il^*X .
}
\]
Here, we respectively write $\varepsilon$, $\eta$, $\partial$
as for $\varepsilon_\il^X$, $\eta_{i_{11}}^X$, $\partial^X$.
Remark that the connecting morphism $\partial$ is unique since
$\Hom(\Sigma_\square \il_!\il^*X , {i_{11}}_!{i_{11}}^?X) = 0$
by adjunction and by the canonical $2$-isomorphism
$\il^* {i_{11}}_! = 0$, because of recollement.

Let us apply the {\em weakly} triangulated functor $F_\square$
and get a distinguished triangle in $\T(\boxempty)$
\[
\xymatrix{
F_\square \il_!\il^*X \ar[r]^<>(0.5){F_\square\varepsilon} &
F_\square X \ar[r]^<>(0.5){F_\square\eta} &
F_\square {i_{11}}_!{i_{11}}^?X \ar[r]^<>(0.5){w} &
\Sigma_\square F_\square \il_!\il^*X .
}
\]
Now, consider the canonical distinguished triangle associated
to the object $F_\square X$,
\[
\xymatrix{
\il_!\il^* F_\square X \ar[r]^<>(0.5){\varepsilon^{F_\square}} &
F_\square X \ar[r]^<>(0.5){\eta^{F_\square}} &
{i_{11}}_!{i_{11}}^? F_\square X \ar[r]^<>(0.5){\partial^{F_\square}} &
\Sigma_\square \il_!\il^* F_\square X \ko
}
\]
whose connecting morphism is also unique. Again, we use
$\varepsilon^{F_\square}$, $\eta^{F_\square}$,
$\partial^{F_\square}$ as for $\varepsilon_\il^{F_\square X}$,
$\eta_{i_{11}}^{F_\square X}$, $\partial^{F_\square X}$.

Let us call $i_1 : e \to \Delta_1$ the open immersion which identifies
$e$ with $1$ and $i_{\Delta_1} : \Delta_1 \to \boxempty$ the open
immersion which identifies $0$ and $1$ with $10$ and $11$, respectively.
Having fixed this notation, we can write $i_{11} = i_{\Delta_1} i_1$.
Let us fix an arbitrary object $X$ lying in $\S(\Delta_1)$ or,
respectively, in $\S(\boxempty)$. By applying Lemma
\ref{lem:diagramtimesarrow} twice to the diagram $I = e$ with
$j = i_1$ and to the diagram $I = \Delta_1$ with $j = i_{\Delta_1}$,
respectively, we get two canonical isomorphisms
\[
\psi_{i_1i_1}^X : {i_1}_!{i_1}^? F_{\Delta_1} X \iso
F_{\Delta_1} {i_1}_!{i_1}^? X
\qquad \mbox{and} \qquad
\psi_{i_{\Delta_1}i_{\Delta_1}}^X :
{i_{\Delta_1}}_!{i_{\Delta_1}}^? F_\square X \iso
F_\square {i_{\Delta_1}}_!{i_{\Delta_1}}^? X
\]
such that the two commutation relations
\[
\psi_{i_1i_1}^X \circ \eta_{i_1}^{F_{\Delta_1}X} =
F_{\Delta_1} (\eta_{i_1}^X) \qquad \mbox{and} \qquad
\psi_{i_{\Delta_1}i_{\Delta_1}}^X \circ
\eta_{i_{\Delta_1}}^{F_\square X} =
F_\square (\eta_{i_{\Delta_1}}^X)
\]
hold. Moreover, the functorial images of the two isomorphisms
under ${i_1}^*$ and ${i_{\Delta_1}}^*$, respectively, furnish
isomorphisms
\[
\psi_{i_1}^X : {i_1}^? F_{\Delta_1} X \iso F_e {i_1}^? X
\qquad \mbox{and} \qquad
\psi_{i_{\Delta_1}} : {i_{\Delta_1}}^? F_\square X \iso
F_{\Delta_1} {i_{\Delta_1}}^? X .
\]

We can combine all these isomorphisms via the compositions of
natural transformations
\[
\xymatrix{
\id \ar[rrr]^<>(0.5){\eta_{\Delta_1}}
\ar@/_2.5pc/[rrrrrr]_<>(0.5){\eta_{i_{11}}} &&&
{i_{\Delta_1}}_!{i_{\Delta_1}}^?
\ar[rrr]^<>(0.5){{i_{\Delta_1}}_! \eta_{i_1} {i_{\Delta_1}}^?} &&&
{i_{\Delta_1}}_!{i_1}_!{i_1}^?{i_{\Delta_1}}^? = {i_{11}}_!{i_{11}}^?
}
\]
and
\[
\xymatrix{
{i_{11}}^?{i_{11}}_! = {i_1}^?{i_{\Delta_1}}^?{i_{\Delta_1}}_!{i_1}_!
\ar[rrr]^<>(0.5){{i_1}^? \eps_{i_{\Delta_1}} {i_1}_!}
\ar@/_2.5pc/[rrrrrr]_<>(0.5){\eps_{i_{11}}} &&&
{i_1}^?{i_1}_!
\ar[rrr]^<>(0.5){\eps_{i_{\Delta_1}}} &&&
\id
}
\]
We get, canonically, an isomorphism
$\psi_{i_1}^{i_{\Delta_1}^? X} \circ
{i_1}^?(\psi_{i_{\Delta_1}}^X)$ as follows
\begin{eqnarray*}
{i_{11}}^? F_\square X = (i_{\Delta_1} i_1)^? F_\square X
                         & = & {i_1}^? {i_{\Delta_1}}^? F_\square X \\
                         & \iso & {i_1}^? F_{\Delta_1} {i_{\Delta_1}}^? X \\
                         & \iso & F_e {i_1}^? {i_{\Delta_1}}^? X
                         = F_e ({i_{\Delta_1} i_1})^? X
                         = F_e {i_{11}}^? X \ko
\end{eqnarray*}
which is canonically isomorphic to $\psi_{i_{11}}^X$.

Clearly, there is another functorial isomorphism
$\psi_{i_{11}i_{11}}^X : {i_{11}}_!{i_{11}}^? F_\square X \iso
F_\square {i_{11}}_! {i_{11}}^? X$, explicitly given by the
composition
\[
\xymatrix{
{i_{11}}_!{i_{11}}^? F_\square X
\ar[rrr]^<>(0.5){{i_{11}}_!(\psi_{i_{11}}^X)}_\sim &&&
{i_{11}}_! F_e {i_{11}}^? X
\ar[rrr]^<>(0.5){\varphi^{i_{11},{i_{11}}^?X}}_\sim &&&
F_\square {i_{11}}_! {i_{11}}^? X .
}
\]
Let us remark the important fact that the relation
$\psi_{i_{11}i_{11}}^X \circ \eta_{i_{11}}^{F_\square X} =
F_\square(\eta_{i_{11}}^X)$ holds. Indeed, we can form the
following diagram
\[
\xymatrix{
F_\square X \ar[r]^<>(0.5){\eta_{i_{\Delta_1}}^{F_\square X}}
\ar@{=}[ddd] &
{i_{\Delta_1}}_! {i_{\Delta_1}}^? F_\square X
\ar[rrr]^<>(0.5){{i_{\Delta_1}}_!
(\eta_{i_1}^{{i_{\Delta_1}}^? F_\square X})}
\ar[d]_<>(0.5){{i_{\Delta_1}}_!(\psi_{i_{\Delta_1}}^X)}^\sim &&&
{i_{\Delta_1}}_!{i_1}_!{i_1}^?{i_{\Delta_1}}^? F_\square X
\ar@{=}[r]
\ar[d]^<>(0.5){{i_{\Delta_1}}_! {i_1}_!{i_1}^?
(\psi_{i_{\Delta_1}}^X)}_\sim &
{i_{11}}_!{i_{11}}^? F_\square X
\ar[ddd]^<>(0.5){\psi_{i_{11}i_{11}}^X}_\sim \\
& {i_{\Delta_1}}_! F_{\Delta_1} {i_{\Delta_1}}^? X
\ar[rrr]^<>(0.5){{i_{\Delta_1}}_!
(\eta_{i_1}^{F_{\Delta_1} {i_{\Delta_1}}^? X})}
\ar@{=}[d] &&&
{i_{\Delta_1}}_!{i_1}_!{i_1}^? F_{\Delta_1} {i_{\Delta_1}}^? X
\ar[d]^<>(0.5){{i_{\Delta_1}}_!
(\psi^{i_{\Delta_1}^?X}_{i_1i_1})}_\sim \\
& {i_{\Delta_1}}_! F_{\Delta_1} {i_{\Delta_1}}^? X
\ar[rrr]^<>(0.5){{i_{\Delta_1}}_! F_{\Delta_1}
(\eta_{i_1}^{{i_{\Delta_1}}^? X})}
\ar[d]_<>(0.5){\varphi^{i_{\Delta_1} ,\: {i_{\Delta_1}}^?X}}^\sim &&&
{i_{\Delta_1}}_! F_{\Delta_1} {i_1}_!{i_1}^? {i_{\Delta_1}}^? X
\ar[d]^<>(0.5){\varphi^{i_{\Delta_1} ,\: {i_1}_!{i_1}^?
{i_{\Delta_1}}^?X}}_\sim \\
F_\square X \ar[r]_<>(0.5){F_\square \eta_{i_{\Delta_1}}^{X}} &
F_\square {i_{\Delta_1}}_!{i_{\Delta_1}}^? X
\ar[rrr]_<>(0.5){F_\square {i_{\Delta_1}}_!
(\eta_{i_1}^{{i_{\Delta_1}}^? X})} &&&
F_\square {i_{\Delta_1}}_!{i_1}_!{i_1}^?{i_{\Delta_1}}^? X
\ar@{=}[r] & F_\square {i_{11}}_!{i_{11}}^? X .
}
\]
Let us consider the first square on the left. By definition,
we have $\varphi^{i_{\Delta_1} ,\: {i_{\Delta_1}}^?X} \circ
{i_{\Delta_1}}_!(\psi_{i_{\Delta_1}}^X) =
\psi_{i_{\Delta_1}i_{\Delta_1}}^X$
and we know (\cf lemma \ref{lem:diagramtimesarrow}) that this
morphism is an iso and makes the square commute.

The three squares in the center of the diagram also commute.
Indeed, the square at the top commutes by functoriality of
$\eta_{i_1}$, the square in the middle commutes by lemma
\ref{lem:diagramtimesarrow} and the square at the bottom
commutes because $\varphi^{i_{\Delta_1}}$ is a natural
transformation.

The last square on the right also is commutative. This fact is
easy to check by using the general relation proved in subsection
\ref{ss:mor} in our special case $v = i_1$ and $u = i_{\Delta_1}$,
which gives the formula
\[
\varphi^{i_{\Delta_1}, \; {i_1}_!X} \circ
{i_{\Delta_1}}_!(\varphi^{i_1,X}) = \varphi^{i_{11}, \; X} \ko
\]
for all $X$ in $\S(e)$. By using this formula, we get a similar
relation, valid for any $X$ in $\S(e)$,
\begin{eqnarray*}
\psi_{i_{11}}^X & = &
\eps_{i_{11}}^{F_\square(i_{11})^?X} \circ
(i_{11})^?[(\varphi^{i_{11},\;(i_{11})^?X})^{-1}] \circ
(i_{11})^?F_e(\eta_{i_{11}}^X) \\
& = & \eps_{i_1}^{F_\square(i_{11})^?X} \circ
{i_1}^?(\eps_{i_{\Delta_1}}^{{i_1}_!F_\square(i_{11})^?X}) \circ
(i_{11})^?{i_{\Delta_1}}_![(\varphi^{{i_1},\;(i_{11})^?X})^{-1}] \circ \\
&& (i_{11})^?[(\varphi^{{i_{\Delta_1}},\;{i_1}_!(i_{11})^?X})^{-1}] \circ
(i_{11})^?F_e{i_{\Delta_1}}_!(\eta_{i_1}^{{i_{\Delta_1}}^?X}) \circ
(i_{11})^?F_e(\eta_{i_{\Delta_1}}^{X}) \\
& = & \eps_{i_1}^{F_\square(i_{11})^?X} \circ
{i_1}^?[(\varphi^{{i_1},\;(i_{11})^?X})^{-1}] \circ
{i_1}^?(\eps_{i_{\Delta_1}}^{F_{\Delta_1}
{i_1}_!{i_1}^?{i_{\Delta_1}}^?X}) \circ \\
&& (i_{11})^?{i_{\Delta_1}}_!F_{\Delta_1}
(\eta_{i_1}^{{i_{\Delta_1}}^?X}) \circ
(i_{11})^?[(\varphi^{{i_{\Delta_1}},\;{i_{\Delta_1}}^?X})^{-1}] \circ
(i_{11})^?F_e(\eta_{i_{\Delta_1}}^X) \\
& = & \eps_{i_1}^{F_\square(i_{11})^?X} \circ
{i_1}^?[(\varphi^{{i_1},\;(i_{11})^?X})^{-1}] \circ
{i_1}^?F_{\Delta_1}(\eta_{i_1}^{{i_{\Delta_1}}^?X}) \circ \\
&& {i_1}^?(\eps_{i_{\Delta_1}}^{F_{\Delta_1}{i_{\Delta_1}}^?X}) \circ
(i_{11})^?[(\varphi^{{i_{\Delta_1}},\;{i_{\Delta_1}}^?X})^{-1}] \circ
(i_{11})^?F_e(\eta_{i_{\Delta_1}}^X) \\
& = & \psi_{i_1}^{{i_{\Delta_1}}^?X} \circ
{i_1}^?(\psi_{i_{\Delta_1}}^X) .
\end{eqnarray*}
Altogether, all these formulae show commutativity of the last
square on the right.

Coming back to our fixed object $X$ of $\S(\boxempty)$, let us
consider the morphism
$\varphi_{\il\il}^X : \il_!\il^* F_\square X \to F_\square \il_!\il^* X$
that means, in our notations, the composition
$\varphi^{\il,\;\il^*X} \circ \il_![(\varphi_\il^X)^{-1}]$.
We can see that the morphism $\eps^{F_\square}$ factors
through $\varphi_{\il\il}^X$. Indeed, we can compute
\begin{eqnarray*}
F_\square(\eps_{\il}^X) \circ \varphi_{\il\il}^X
& = &
F_\square(\eps_{\il}^X) \circ \varphi^{\il_!,\;\il^*X} \circ
\il_![(\varphi_\il^X)^{-1}] \\
& = &
F_\square(\eps_{\il}^X) \circ
\eps_\il^{F_\square\il_!\il^*X} \circ
\il_!(\varphi_\il^{\il_!\il^*X}) \circ
\il_!F_\lefthalfcap(\eta_\il^{\il^*X}) \circ
\il_![(\varphi_\il^X)^{-1}] \\
& = &
\eps_{\il}^{F_\square X} \circ
\il_!\il^* F_\square(\eps_\il^{X}) \circ
\il_!(\varphi_\il^{\il_!\il^*X}) \circ
\il_!F_\lefthalfcap(\eta_\il^{\il^*X}) \circ
\il_![(\varphi_\il^X)^{-1}] \\
& = &
\eps_{\il}^{F_\square X} \circ
\il_!(\varphi_\il^{X}) \circ
\il_!F_\lefthalfcap \il^*(\eps_\il^{X}) \circ
\il_!F_\lefthalfcap (\eta_\il^{\il^*X}) \circ
\il_![(\varphi_\il^X)^{-1}] \\
& = &
\eps_\il^{F_\square X} .
\end{eqnarray*}
Let us consider an extension of this factorization to a morphism
of distinguished triangles of $\T(\boxempty)$
\[
\xymatrix{
\il_!\il^* F_\square X \ar[d]_{\varphi_{\il\il}^X}
\ar[r]^<>(0.5){\varepsilon^{F_\square}} & F_\square X
\ar@{=}[d]^\id \ar[r]^<>(0.5){\eta^{F_\square}} &
{i_{11}}_!{i_{11}}^? F_\square X
\ar@{-->}[d]^{\psi_\square^X}
\ar[r]^<>(0.5){\partial^{F_\square}} &
\Sigma_\square \il_!\il^* F_\square X
\ar[d]^{\Sigma \varphi_{\il\il}^X} \\
F_\square \il_!\il^* X \ar[r]^<>(0.5){F_\square \varepsilon} &
F_\square X \ar[r]^<>(0.5){F_\square \eta} &
F_\square {i_{11}}_!{i_{11}}^? X \ar[r]^<>(0.5){w} &
\Sigma_\square F_\square \il_!\il^* X .
}
\]
The extension morphism $\psi_\square^X$ is {\em unique} making
commute the central square because the group
\begin{eqnarray*}
\Hom(\Sigma_\square \il_!\il^* F_\square X ,
F_\square {i_{11}}_!{i_{11}}^? X)
& = &
\Hom(\Sigma_\square \il_!\il^* F_\square X ,
{i_{11}}_!{i_{11}}^? F_\square X) \\
& = &
\Hom(\il_! \Sigma_\lefthalfcap \il^* F_\square X ,
{i_{11}}_!{i_{11}}^? F_\square X) \\
& = &
\Hom(\Sigma_\lefthalfcap \il^* F_\square X ,
\il^*{i_{11}}_!{i_{11}}^? F_\square X) \\
& = &
0
\end{eqnarray*}
is canonically trivial. Here, we use the canonical isomorphism
$\psi_{i_{11}i_{11}}^X$ in the first equality, the canonical
invertible natural transformation $\delta_\lefthalfcap : \il_! \Sigma_\lefthalfcap \iso \Sigma_\square \il_!$ in the second
equality, adjunction in the third one and the canonical $2$-isomorphism
$\il^* {i_{11}}_! = 0$ of the recollement axioms in the last. Hence,
we must canonically identify $\psi_\square^X$ with the isomorphism
$\psi_{i_{11}i_{11}}^X$.
It follows that $\varphi_{\il\il}^X$ is invertible, too. Therefore,
$\eps_\il^X$ is invertible if and only if $\eps_\il^{F_\square X}$ is.

b) Since Axiom {\bf Der 7} holds for the derivator $\S$, it suffices
to show that, given any cocartesian (hence cartesian) square
$X \in \ul\ca(\boxempty)$ as in Definition \ref{def:extria}, part a),
whose diagram is
\begin{eqnarray*}
\begin{matrix}
\\\\\\
\dia_\square X
\end{matrix}
&
\begin{matrix}
\\\\\\
\mbox{=}
\end{matrix}
&
\xymatrix{
i_{00}^*X \ar@{ >->}[r] \ar@{->>}[d] &
i_{01}^*X \ar@{->>}[d] \\
i_{10}^*X \ar@{ >->}[r] & i_{11}^*X
}
\begin{matrix}
\\\\\\
\mbox{,}
\end{matrix}
\end{eqnarray*}
the adjunction morphism
\[
\eps^{F_\square X} : \il_!\il^* F_\square X \to F_\square X
\]
is invertible.

Let us start with a slightly more general object $X$ in
$\ul\ca(\boxempty)$, whose diagram is a commutative square with
inflations as horizontal arrows and deflations as vertical
arrows, which is not required to be cartesian, nor cocartesian.
Let us suppose that the adjoint morphism
$\eps^X : \il_!\il^*X \to X$ is a deflation. This is not
restrictive. Indeed, it occurs under the mild additional
hypothesis that the exact categories we are considering have
split idempotents, or, which is the case we will be dealing
with, when the square object $X$ is bicartesian. In fact,
the object ${i_{11}}^*\il_!\il^*X$ is the homotopy colimit
$\hocolim_\lefthalfcap \il^*X$.

Let us explain this, in general. Let us consider the category
$11 \backslash \lefthalfcap$, defined by the inclusion
$i_\lefthalfcap : \lefthalfcap \to \boxempty$. There is a
(noncommutative) square
\[
\xymatrix{
11 \backslash \lefthalfcap \ar[r]_<>(0.4)j^<>(0.4)\simeq
\ar[d]_{p_{11 \backslash \lefthalfcap}} &
\lefthalfcap \ar[d]^{i_\lefthalfcap} \ar[dl]^{p_\lefthalfcap} \\
e \ar[r]_{i_{11,\square}} & \; \boxempty & \ko
}
\]
where the forgetful functor $j$ is an isomorphism of categories.
By Axiom {\bf Der 4}, we find a functorial invertible natural
transformation of functors
\[
c'_{11 \backslash \lefthalfcap} : (p_\lefthalfcap)_! =
(p_{11 \backslash \lefthalfcap})_! j^* \Iso
i_{11,\square}^* i_{\lefthalfcap !} \;\; .
\]
Moreover, it is not difficult to see that there is a
{\em canonical} iso
\[
i_{01,\,\lefthalfcap}^*Z \coprod_{i_{00,\,\lefthalfcap}^*Z}
i_{10,\,\lefthalfcap}^*Z \iso {p_\lefthalfcap}_!Z \ko
\]
for any object $Z$ in $\ul\ca(\lefthalfcap)$.
Thus, by the uniqueness of push-outs in exact categories it
follows that the (unique) universal morphism ${i_{11}}^*(\eps^X)$
is a deflation, if the exact category $\ul\ca(e) = \ca$ has
splitting idempotents, or even an iso, if the the square is
supposed to be cocartesian.

Let us consider the conflation induced by the deflation
$\eps^X$,
\[
\xymatrix{
\Omega {i_{11}}_!{i_{11}}^? \ar@{ >->}[r] &
\il_!\il^*X \ar@{->>}[r] & X \ko
}
\]
and apply the {\em weak} $\partial$-functor $F_\square$.
After shifting, we get a distinguished triangle
\[
\xymatrix{
F_\square \il_!\il^*X \ar[r]^<>(0.5){} &
F_\square X \ar[r]^<>(0.5){} &
F_\square {i_{11}}_!{i_{11}}^?X \ar[r]^<>(0.5){} &
\Sigma_\square F_\square \il_!\il^*X \ko
}
\]
The proof goes on along the same lines of the proof of
item a), with the difference that this time we will apply
lemma \ref{lem:exdiagramtimesarrow} anytime we were
applying lemma \ref{lem:diagramtimesarrow}.
\end{proof}

\section{The Universal Property} \label{s:main}
In this section we give a proof of the main theorem \ref{thm:main}.
Let us put together some trivial observations at first.

\begin{lemma} \label{lemma:lemma1}
Let $\ul\ca$ be the represented exact prederivator associated to
an exact category $\ca$. Let $\E$ be a triangulated) derivator.
Fix an arbitrary diagram $I$ in $\Dia_\f$. Then, the sequence of
exact (\resp, triangulated) categories defined by
\[
\ul\ca^I_n := \ul\ca(C_n \times I) \ko \quad
\sfe^I_n := \E(C_n \times I) \ko n \in \N \ko
\]
with the induced functors defined by
\[
u_I^* := (u \times \id_I)^* : \ul\ca^I_n \to \ul\ca^I_m \ko \quad
u_I^* := (u \times \id_I)^* : \sfe^I_n \to \sfe^I_m \ko
\]
for any morphism $u : C_m \to C_n$ in $\Cubes$, and the natural
transformations defined by
\[
\alpha_I^* := (\alpha \times \id_I)^* : v_I^* \Rightarrow u_I^* \ko
\]
for any $2$-morphism $\alpha : u \Rightarrow v$ in $\Cubes$,
give rise to an exact (\resp, triangulated) tower $\sfe^I$
(\resp, $\ul\ca^I$) such that $\sfe^I(C_n) = \sfe^I_n$ (\resp,
$\ul\ca^I_n = \ul\ca^I(C_n)$).
\end{lemma}

\begin{proof}
It reduces to a straightforward checking that the axioms of an exact
(\resp, triangulated) tower hold for $\sfe^I$.
\end{proof}

\begin{lemma} \label{lemma:lemma2}
Let $F : \D \to \E$ be a triangulated morphism of triangulated
derivators. Fix an arbitrary diagram $I$ in $\Dia_\f$.
Then, the sequence of associated triangulated functors defined by
\[
F^I_n := F_{C_n \times I} : \sfd^I_n \to \sfe^I_n \ko n \in \N \ko
\]
with the induced invertible natural transformations of functors defined by
\[
\varphi^I_u := \varphi_{u \times \id_I} : F^I_m u_I^* \iso u_I^* F^I_n \ko
\]
for any morphism $u : C_m \to C_n$ in $\Cubes$, gives rise to a
morphism of triangulated towers $F^I : \sfd^I \to \sfe^I$ such that
$F^I_n = F_{C_n \times I}$.
\end{lemma}

\begin{proof}
Again, it is straightforward to check that the axioms of a
triangulated morphism of towers hold for $F^I$.
\end{proof}

Clearly, an analogous version of the last lemma holds for exact
morphisms of represented exact prederivators associated to
additive categories.
Let $I$ be any diagram of $\Dia_\f$. Let us associate an additive
(\resp, exact) tower $\ul{\sfa^I}$ with any additive (\resp, exact)
category $\ca$, according to the following definition
\[
\ul{\sfa^I}(C_n) := \ul\Hom(C_n^\circ, \ul\Hom(I^\circ, \ca)) \ko
n \in \N \; .
\]

\begin{lemma} \label{lemma:lemma3}
Let $\ca$ be an exact category and $\E$ a triangulated derivator.
Suppose that $F : \ul\ca \to \E$ is an exact morphism of prederivators.
Fix an arbitrary diagram $I$ in $\Dia_\f$. Then, there are an exact
tower $\ul{\sfa^I}$, a triangulated tower $\sfe^I$ and an induced
exact morphism of towers
\[
F^I : \ul{\sfa^I} \to \sfe^I
\]
where $\ul{\sfa^I}$ equals $\ul\ca^I$ (according to the notations
of Lemma \ref{lemma:lemma1}).
\end{lemma}

\begin{proof}
For each $n \in \N$, we have
\begin{eqnarray*}
\ul\ca^I(C_n) & = & \ul\ca(C_n \times I) \\
      & = & \ul\Hom((C_n \times I)^\circ, \ca) \\
      & = & \ul\Hom(C_n^\circ \times I^\circ, \ca) \\
      & = & \ul\Hom(C_n^\circ, \ul\Hom(I^\circ, \ca)) \\
      & = & \ul{\sfa^I}(C_n) \; .
\end{eqnarray*}
Now use Lemma \ref{lemma:lemma1} and Lemma \ref{lemma:lemma2}.
\end{proof}

We can now prove the main Theorem \ref{thm:main}.

\begin{proof}
Let us consider an exact morphism of prederivators $F : \ul\ca \to \E$.
We functorially construct an object $\tilde F$ lying in
$\ul\HOM_{tr}(\D_\ca,\E)$ such that
$can^*(\tilde F) = \tilde F \circ can$ is $F$ (up to a unique
iso). By Lemma \ref{lemma:lemma3}, if we fix a diagram $I$ we have
the induced exact morphism of towers $F^I : \ul{\sfa^I} \to \sfe^I$.
This is an object in the category $\ul\HOM_{ex}(\ul{\sfa^I} \ko \sfe^I)$.
By Keller's theorem \ref{thm:kellerthm}, this morphism extends uniquely
to a morphism of triangulated towers
$\tilde{F^I} : \sfd_{\ca^I} \to \sfe^I$. Here $\sfd_{\ca^I}$ is a
tower that associates the bounded derived category
$\cd^b(\ul\Hom(C_n^\circ, \ul\Hom(I^\circ, \ca))$ with any cube $C_n$
and equals the tower $\sfd_\ca^{\;\,I}$ which is defined by the relation
$\sfd_\ca^{\;\,I} (C_n) = \cd^b(\ul\Hom((C_n \times I)^\circ, \ca))$.
Thus, the morphism $\tilde{F^I} : \sfd_{\ca^I} \to \sfe^I$ identifies
with a morphism $\tilde{F}^I : \sfd_\ca^{\;\,I} \to \sfe^I$. The {\em base}
of the morphism $\tilde{F}^I$ (\ie, the evaluation at $C_0$) gives us
a triangulated functor $(\tilde{F}^I)_0 : \D_\ca(I) \to \E(I)$.

It remains to check that, by letting $I$ running in $\Dia_\f$, we get
a triangulated morphism of derivators $\tilde{F} : \D_\ca \to \E$,
defined by the equality $\tilde{F}_I := (\tilde{F}^I)_0$, whose image
in $\ul\HOM_{ex}(\ul\ca,\E)$ is $F$. Notice that, by the construction,
the composition $\tilde{F}_I \circ can_I$ is $F_I$, for all diagrams
$I$ (up to a canonical iso).
So, we can start with the second item in Definition \ref{def:dermor}.
Let $u : I \to J$ be a morphism of finite diagrams. Since $F$
is a morphism of prederivators, there is a square
\[
\xymatrix{
\ul\ca(J) \ar[r]^{u^*} \ar[d]_{F_J} & \ul\ca(I) \ar[d]^{F_I} \\
\E(J) \ar[r]_{u^*} & \E(I) & \ko
}
\]
which is commutative (up to a unique iso) by the definition, \ie,
there is an invertible natural transformation of functors
$\varphi_u : F_I u^* \iso u^* F_J$. Now, the previous square is
(canonically isomorphic to) the base of the following commutative
square of towers
\[
\xymatrix{
\ul{\sfa^J} \ar[r]^{{u^*}^\wedge} \ar[d]_{F^J} & \ul{\sfa^I} \ar[d]^{F^I} \\
\sfe^J \ar[r]_{{u^*}^\wedge} & \sfe^I & \ko
}
\]
where ${u^*}^\wedge$ is the induced exact morphism of towers
$\ul\Hom((-)^\circ,u^*)$. By Theorem \ref{thm:kellerthm} the
morphisms $F^J$ and $F^I$ uniquely extend to morphisms
$\tilde{F^J}$ and $\tilde{F^I}$.
Since the morphism ${u^*}^\wedge$ is exact it extend to the
derived morphism $\R {u^*}^\wedge$. So, we get a square of
triangulated towers
\[
\xymatrix{
\sfd_{\ca^J} \ar[r]^{\R {u^*}^\wedge} \ar[d]_{\tilde{F^J}} &
\sfd_{\ca^I} \ar[d]^{\tilde{F^I}} \\
\sfe^J \ar[r]_{{u^*}^\wedge} & \sfe^I & \ko
}
\]
which is equal to the following commutative square
\[
\xymatrix{
\sfd_{\ca}^{\;\,J} \ar[r]^{\R {u^*}^\wedge} \ar[d]_{\tilde{F}^J} &
\sfd_{\ca}^{\;\,I} \ar[d]^{\tilde{F}^I} \\
\sfe^J \ar[r]_{{u^*}^\wedge} & \sfe^I & \ko
}
\]
whose base (evaluation at $C_0$) is a square of triangulated categories
\[
\xymatrix{
\D_\ca(J) \ar[rr]^{u^* := \R u^*} \ar[d]_{\tilde{F}_J} &&
\D_\ca(I) \ar[d]^{\tilde{F}_I} \\
\E(J) \ar[rr]_{u^*} && \E(I) & .
}
\]
The functor $\R u^*$ exists since $u^* : \ul\ca(J) \to \ul\ca(I)$
is an exact functor of exact categories.

It remains to check that this square is commutative up to a unique
iso. Let $\tilde\varphi_u : \tilde{F}_I \R u^* \to u^* \tilde{F}_J$
be the natural transformation of functors induced by $\varphi_u$ via
the fully faithful functor $can_J$. We apply item b) in \cite[Lemme
2]{Keller07}, where we identify the functor $G$ with $\tilde{F}_I \R u^*$,
the functor $G'$ with $u^* \tilde{F}_J$ and the class of objects
$\mathcal{X}$ with the set of objects of the category $\ul\ca(J)$.
Indeed, the image of the subcategory $\ul\ca(J)$ generates $\D_\ca(J)$
as a triangulated category. Since the restriction of the natural
transformation $\tilde\varphi_u$ to the evident restriction
subfunctors over the subcategory $\ul\ca(J)$ equals the invertible
natural transformation $\varphi_u$, it follows by the cited Lemma
that $\tilde\varphi_u : \tilde{F}_I \R u^* \to u^* \tilde{F}_J$ is
an invertible natural transformation of functors, too. It is then
clear how to check the remaining axioms of a morphism of derivators.

The additive morphism $\tilde{F}$ extends to a triangulated morphism.
Indeed, the functor $\tilde{F}_I$ is triangulated with respect to
the canonical triangulated structures of the categories $\D_\ca(I)$
and $\E(I)$, for all diagrams $I$. Therefore, we can apply the item
a) of Prop. \ref{prop:redundancy}, which (largely) suffices to tell
us that the morphism $\tilde{F}$ is triangulated.

It is clear that the functor $F \mapsto \tilde{F}$ is a quasi-inverse
to the functor $F \mapsto F \circ can$ since this is locally true
for every diagram $I$.
\end{proof}


\section{Derived extension of an exact category} \label{s:cor}
In this section we work with derivators of type $\Dia_\f$, \ie,
derivators which are defined over the $2$-subcategory of $\Dia$ that
consists of finite directed diagrams.

Let us denote by $\Mor_{I^\circ}(\T(e))$ the category of
{\em $I^\circ$-morphisms} in $\T(e)$, \ie, an object in this category
is a family of morphisms of $\T(e)$
\[
\{ F_{\sigma(u),\tau(u)} : F_{\sigma(u)} \to
F_{\tau(u)} \}_{u \in \Mor(I^\circ)} \quad \in \quad
\prod_{j \in I} \,\, \prod_{i \in I} \prod_{I(j,i)}
\Hom_{\T(e)}(F_i, F_j)
\]
indexed over all arrows in $I^\circ$, with the convention that
$F_{\id_i} = \id_{F_i}$ for all $i \in I$. (It is possible that
some morphism in the diagram is zero if in the category $I$
there are objects with no arrows in between.) A morphism in
$\Mor_{I^\circ}(\T(e))$ is given by a family of morphisms
\[
\{ g_i : F_i \to G_i \}_{i \in I} \quad \in \quad
\prod_{i \in I} \Hom_{\T(e)}(F_i, G_i)
\]
such that the square
\[
\xymatrix{
F_{\sigma(u)} \ar[rr]^{F_{\sigma(u),\tau(u)}}
\ar[d]_{g_{\sigma(u)}} && F_{\tau(u)} \ar[d]^{g_{\tau(u)}} \\
G_{\sigma(u)} \ar[rr]_{G_{\sigma(u),\tau(u)}} && \; G_{\tau(u)}
}
\]
is commutative, for each arrow $u \in {\Mor(I^\circ)}$.
It is clear that there is a natural isomorphism of categories
\[
\ul\Hom(I^\circ,\T(e)) \xymatrix{\ar[r]^\simeq &}
\Mor_{I^\circ}(\T(e)) \ko
\]
\[
\quad F \mapsto
\{F(u) : F({\sigma(u)}) \to F({\tau(u)}) \}_{u \in {\Mor(I^\circ)}}
\]
whose inverse is given by the functor which maps
$\{ F_{\sigma(u),\tau(u)} \}_{u \in \Mor(I^\circ)}$ to the presheaf
$F$ defined by $F(u) = F_{\sigma(u),\tau(u)}$, for all $u \in I^\circ$.
Let us consider the functor
\[
\T(I) \xymatrix{\ar[rr]^{\mor_{I}} &&} \Mor_{I^\circ}(\T(e)) \ko
\]
\[
X \mapsto
\{ \alpha^*_{\sigma(u),\tau(u)}X :
{\tau(u)}^*X \to {\sigma(u)}^*X \}_{u \in \Mor(I)} \ko
\]
where, for each arrow $u \in \Mor(I)$, the natural transformation
$\alpha^*_{\sigma(u),\tau(u)} : \tau(u)^* \to \sigma(u)^*$ is induced
by the $2$-arrow
\[
\alpha_{\sigma(u),\tau(u)} : \sigma(u) \Rightarrow \tau(u) \; .
\]
There is a commutative diagram
\[
\xymatrix{
\ul\Hom(I^\circ,\T(e)) \ar[r]^<>(0.5){\simeq} & \Mor_{I^\circ}(\T(e)) \\
\;\, \T(I) \ar[u]^{\dia_{I}} \ar[ur]_<>(0.5){\mor_{I}} .
}
\]
and a canonical bijection
\[
\Hom_{\ul\Hom(I^\circ,\T(e))}(\dia_I X, \dia_I Z)
\iso \Hom_{\Mor_{I^\circ}(\T(e))}(\mor_I X, \mor_I Z) .
\]

Let us prepare the crucial ingredients for the proof of the main
theorem of this section.

\subsection{Full faithfulness of the diagram functor} \label{ss:ff}
The following proposition gives a sufficient condition which allows
us to lift morphisms of presheaves {\em uniquely}, when this is
possible, from $\ul\Hom(I^\circ,\T(e))$ to morphisms of $\T(I)$.
This will be helpful in the main theorem of this section in order
to tell whether the functor $\dia_I$ is fully faithful when we
consider its restriction to some subcategory of $\T(I)$ which
enjoys this property (Toda condition).

\begin{proposition} \label{prop:fullfaith}
Let $\T$ be a triangulated derivator of type $\Dia_\f$.
Fix an arbitrary finite diagram $I$.
Let $X$ and $Z$ be objects lying in $\T(I)$. Suppose that the
(Toda) condition
\[
\Hom_{\T(e)}(\Sigma^n i^*X, j^*Z) = 0 \ko \quad n>0 ,
\]
holds for each pair $(i,j) \in I \times I$.
Then, the functor $\dia_I : \T(I) \to \ul\Hom(I^\circ,\T(e))$ induces
a bijection
\[
\xymatrix{
\Hom_{\T(I)}(X, Z) \ar[r]^<>(0.5)\sim &
\Hom_{\ul\Hom(I^\circ,\T(e))}(\dia_I X, \dia_I Z) .
}
\]
\end{proposition}

\begin{proof}
Let us endow the additive category $\T(e)$ with an exact structure
by defining conflations as the split short exact sequences. In this
way, once a diagram $I$ of $\Dia_\f$ has been fixed, we can
consider $\ul\Hom(I^\circ,\T(e))$ as an exact category with the
pointwise (split) exact structure.

Let $X$ be an arbitrary object in the triangulated category
$\T(I)$. For each arrow $a : j \to i$ in the
diagram $I$, let $a_i \in \Hom_{\T(I)}(j_!i^*X, i_!i^*X)$
be the image of the identity $\id^{i_!i^*X}$ under the
composition of homomorphisms
\begin{eqnarray*}
\Hom_{\T(I)}(i_!i^*X, i_!i^*X) \iso &
                 \Hom_{\T(e)}(i^*X, i^*i_!i^*X) \\
                 \to & \Hom_{\T(e)}(i^*X, j^*i_!i^*X) \\
                 \iso & \Hom_{\T(I)}(j_!i^*X, i_!i^*X)
\end{eqnarray*}
induced by the adjoint pairs $i_! \dashv i^*$ and $j_! \dashv j^*$
and the $2$-arrow $\alpha^*_{ji}$
\[
\xymatrix{\T(I) \rtwocell^{j^*}_{i^*}{^{\alpha^*_{ji}}} & \T(e)}
\]
associated with the natural transformation $\alpha_{ji}$
\[
\xymatrix{e \rtwocell^i_j{^{\alpha_{ji}}} & I}
\]
defined by the morphism $a$.

By using the formal properties of adjoint functors, we can find the
formula for $a_i$
\[
a_i = \varepsilon_j^{i_!i^*X} \circ j_![\alpha^*_{ji}(i_!i^*X)]
\circ j_!(\eta_i^{i^*X}) \ko
\]
where, for each $l \in I$ and $Y \in \T(I)$, the morphisms
$\varepsilon_l^Y$ (\resp, $\eta_l^Y$) are defined by using the counit
(\resp, the unit) of the adjunction $l_! \dashv l^*$.

Similarly, let $a_j \in \Hom_{\T(I)}(j_!i^*X, j_!j^*X)$ be the image
of the identity $\id^{j_!j^*X}$ under the composition of homomorphisms
\begin{eqnarray*}
\Hom_{\T(I)}(j_!j^*X, j_!j^*X) \iso &
                 \Hom_{\T(e)}(j^*X, j^*j_!j^*X) \\
                 \to & \Hom_{\T(e)}(i^*X, j^*j_!j^*X) \\
                 \iso & \Hom_{\T(I)}(j_!i^*X, j_!j^*X)
\end{eqnarray*}
induced by the adjoint pair $j_! \dashv j^*$ and the $2$-arrow
defined by the morphism $a$ as above. Here, it is easy to see that the
formula
\[
a_j = j_![\alpha_{ji}^*(X)]
\]
for $a_j$ holds. Thus, for each morphism
$a : j \to i$ in the diagram $I$, we get an arrow
\[
\xymatrix{
j_!i^*X \ar[rr]^<>(0.5){[a_i, -a_j]^t} &&
i_!i^*X \oplus j_!j^*X .
}
\]
We claim that the composition of morphisms
\[
\xymatrix{
j_!i^*X \ar[rr]^<>(0.5){[a_i, -a_j]^t} && i_!i^*X \oplus j_!j^*X
\ar[rr]^<>(0.5){[\varepsilon_i,\varepsilon_j]} && X
}
\]
is zero. Indeed, by $2$-functoriality of $\alpha_{ji}^*$ there is the
equality
\[
j^*(\varepsilon_i) \circ \alpha_{ji}^*(i_!i^*X) =
\alpha_{ji}^*(X) \circ i^*(\varepsilon_i) .
\]
Compose with the morphism $\eta_i^{i^*X}$ and use
$i^*(\varepsilon_i) \circ \eta_i^{i^*X} = \id^{i^*X}$ to find the
relation for $\alpha_{ji}^*$
\[
j^*(\varepsilon_i) \circ \alpha_{ji}^*(i_!i^*X) \circ \eta_i^{i^*X} =
\alpha_{ji}^*(X) .
\]
Hence, we can compute
\begin{eqnarray*}
[\varepsilon_i, \varepsilon_j] \circ [a_i, -a_j]^t & = &
\varepsilon_i \circ \varepsilon_j^{i_!i^*X} \circ
j_![\alpha^*_{ji}(i_!i^*X)] \circ j_!(\eta_i^{i^*X}) -
\varepsilon_j \circ j_![\alpha_{ji}^*(X)] \\
& = & \varepsilon_j \circ j_![j^*(\varepsilon_i)] \circ
j_![\alpha^*_{ji}(i_!i^*X)] \circ j_!(\eta_i^{i^*X}) -
\varepsilon_j \circ j_![\alpha_{ji}^*(X)] \\
& = & 0 .
\end{eqnarray*}
Here the second equality comes from the functoriality of $\varepsilon_j$
and the third by the relation for $\alpha_{ji}^*$.

Recall that the set of objects and the $\Hom$-sets in $I$ are
finite and that the triangulated category $\T(I)$ has finite
coproducts. So, we obtain a morphism in $\T(I)$
\[
\xymatrix{
\coprod_{I(j,i)} j_!i^*X \ar[rr] &&
i_!i^*X \oplus j_!j^*X \ar[rr]^<>(0.5){can} &&
\coprod_{i \in I} i_!i^*X \ko
}
\]
whose components are $[a_i, -a_j]^t$.
All these morphisms are the components of an arrow
\[
\xymatrix{
\coprod_{j \in I} \coprod_{i \in I - \{j\}} \coprod_{I(j,i)}
j_!i^*X \ar[rr]^<>(0.5){u} && \coprod_{i \in I} i_!i^*X .
}
\]
Now it is clear that the composition
\[
\xymatrix{
\coprod_{j \in I} \coprod_{i \in I - \{j\}} \coprod_{I(j,i)}
j_!i^*X \ar[rr]^<>(0.5){u} && \coprod_{i \in I} i_!i^*X
\ar[rr]^<>(0.5){\varepsilon} && X \ko
}
\]
where the morphism $\varepsilon$ is induced by the $\varepsilon_i$,
$i \in I$, still vanishes.

Let us call $QX$ the object $\coprod_{i \in I} i_!i^*X$ and consider a
distinguished triangle in $\T(I)$
\[
\xymatrix{
LX \ar[rr]^<>(0.5){v} && QX \ar[rr]^<>(0.5){\varepsilon} && X
\ar[rr]^<>(0.5){} && \Sigma LX .
}
\]
Notice that the morphism $k^*(v)$ is the kernel of $k^*(\varepsilon)$,
for any $k \in I$. Indeed, since the functor $k^*$ is triangulated, we
have a distinguished triangle in $\T(e)$
\[
\xymatrix{
k^*(LX) \ar[rr]^<>(0.5){k^*(v)} && \coprod_{i \in I} \coprod_{I(k,i)}
i^*X \ar[rr]^<>(0.5){k^*(\varepsilon)} && k^*X
\ar[rr]^<>(0.5){} && \Sigma k^*(LX) .
}
\]
The components of the morphism $k^*(\varepsilon)$ are
\[
\alpha^*_{ki}(X) : i^*X \to k^*X \ko
\]
where $\alpha_{ki}$ is the $2$-arrow
\[
\xymatrix{e \rtwocell^i_k{^{\alpha_{ki}}} & I}
\]
defined by a morphism $k \to i$ in the diagram category $I$.
Since in the diagram $I$ there are no nontrivial loops, we have that
$k^*(\varepsilon)$ is a section, with retraction given by a morphism
\[
\xymatrix{
k^*X \ar[rr]^<>(0.5){} && \coprod_{i \in I} \coprod_{I(k,i)} i^*X \ko
}
\]
all of whose components are zero but one, the identity on $k^*X$.
It follows that $k^*(v)$ is actually the kernel of $k^*(\varepsilon)$
and that the connecting morphism $\partial^{k^*X}$ must be zero.
Hence, we have that locally the kernel is given by the equality
\[
j^*(LX) = \coprod_{i \in I - \{j\}} \coprod_{I(j,i)} i^*X .
\]
Now it is clear that
\[
QLX = \coprod_{j \in I} j_!j^*LX = \coprod_{j \in I}
\coprod_{i \in I - \{j\}} \coprod_{I(j,i)} j_!i^*X
\]
is just the domain of the morphism $u$ which we have constructed
above.

Since we have seen that the composition  $\varepsilon \circ u$
vanishes, there exists an arrow $\varepsilon^1 : QLX \to LX$
such that, by composition with $v$, it gives $u$, as in the
following commutative diagram
{\footnotesize
\[
\xymatrix{
&& QLX \ar[d]^<>(0.5){u} \ar[dll]_{\varepsilon^1} \ar[drr]^0 \\
LX \ar[rr]^<>(0.5){v} && QX
\ar[rr]^<>(0.5){\varepsilon} && X \ar[rr]^<>(0.5){} && \Sigma LX .
}
\]
}
By iterating this construction we get a `resolution' as in the following
commutative diagram in $\T(I)$
\[
\xymatrix{
QL^nX \ar@{=}[dr]_<>(0.5){\varepsilon^n} \ar[rr]^<>(0.5){u^n} &&
QL^{n-1}X \ar[dr]_<>(0.5){\varepsilon^{n-1}} \ar[r]^<>(0.5){u^{n-1}} & \dots
\ar[r]^<>(0.5){u^2} & QLX \ar[dr]_<>(0.5){\varepsilon^1} \ar[rr]^<>(0.5){u^1}
&& QX \ar[r]^<>(0.5){\varepsilon^0} & X \\
& L^{n}X \ar[ur]_<>(0.5){v^{n}} && \ldots
\ar[ur]_<>(0.5){v^2} && LX \ar[ur]_<>(0.5){v^1} & .
}
\]
Here, the maps $\varepsilon^{0}$, $u^{1}$, $v^{1}$ are
$\varepsilon$, $u$ and $v$, respectively.
Moreover, compositions $u^{k} \circ u^{k+1}$ vanish, and all consecutive
arrows $v^{l+1}$, $\varepsilon^{l}$ fit in a distinguished triangle,
for all $l$,
\[
\xymatrix{
L^{l+1}X \ar[r]^<>(0.5){v^{l+1}} & QL^{l}X
\ar[r]^<>(0.5){\varepsilon^{l}} & L^{l}X
\ar[r]^<>(0.5){} & \Sigma L^{l+1}X
} \ko l \in \N \ko
\]
such that $u^{l} = v^{l}\varepsilon^{l}$, for all $l \in \N-\{0\}$.
In this sequence of maps we can easily check that
\[
m^*(L^nX) \; = \; \coprod_{i_1 \neq m} \coprod_{I(m,i_1)} \ldots
\coprod_{i_n \neq i_{n-1}} \coprod_{I(i_{n-1},i_n)} i_n^*X \ko
\]
for all $m \in I$. This shows that our `resolution' of $X$
must be finite since the diagram $I$ is supposed to be finite.

If $l = n$ or $l = n-1$ we have something more, since in these cases
distinguished triangles can be written as
\[
\xymatrix{
QLW \ar[r]^<>(0.5){v} & QW
\ar[r]^<>(0.5){\varepsilon} & W \ar[r] & \Sigma QLW \ko
}
\]
where $W$ can be $L^nX$ or $L^{n-1}X$.
We recall that this means a distinguished triangle
\[
\xymatrix{
\coprod_{j \in I} \coprod_{i \in I - \{j\}}
\coprod_{I(j,i)}
j_!i^*W \ar[r]^<>(0.5){u} & \coprod_{i \in I} i_!i^*W
\ar[r]^<>(0.5){\varepsilon} & W \ar[r] & \Sigma
\coprod_{j \in I} \coprod_{i \in I - \{j\}}
\coprod_{I(j,i)} j_!i^*W .
}
\]
Let us apply the functor $\Hom_{\T(I)}(?,Z)$ to this distinguished
triangle and get a long exact sequence of abelian groups
\[
\xymatrix{
\Hom_{\T(I)}(\Sigma \coprod_{j \in I}
\coprod_{i \in I - \{j\}} \coprod_{I(j,i)} j_!i^*W, Z) \ar[d] \\
\Hom_{\T(I)}(W, Z) \ar[d]^{\ul{\varepsilon}} \\
\Hom_{\T(I)}(\coprod_{i \in I} i_!i^*W, Z) \ar[d]^{\ul{u}} \\
\Hom_{\T(I)}(\coprod_{j \in I} \coprod_{i \in I - \{j\}}
\coprod_{I(j,i)} j_!i^*W, Z) .
}
\]
Here, the morphisms $\ul\varepsilon$ and $\ul u$ are induced
by composition from $\varepsilon$ and $u$. By using the bijections
induced by the adjunction pairs $j_! \dashv j^*$ and $i_! \dashv i^*$,
we obtain an exact sequence of abelian groups
\[
\xymatrix{
\prod_{j \in I} \prod_{i \in I - \{j\}} \prod_{I(j,i)}
\Hom_{\T(e)}(\Sigma i^*W, j^*Z)
\ar[d] \\
\Hom_{\T(I)}(W, Z)
\ar[d]^{\varepsilon'} \\
\prod_{i \in I}\Hom_{\T(e)}(i^*W, i^*Z) \ar[d]^{u'} \\
\prod_{j \in I} \prod_{i \in I - \{j\}} \prod_{I(j,i)}
\Hom_{\T(e)}(i^*W, j^*Z) .
}
\]
Since this sequence is exact, it follows that the image of the
homomorphism $\varepsilon '$ precisely consists of the families
of morphisms $\{ f_i : i^*W \to i^*Z \}_{i \in I}$ which specify
a morphism in the category $\Mor_{I^\circ}(\T(e))$.
Indeed, if $u'(\{f_i\}_{i \in I})$ is zero, we have a commutative square
\[
\xymatrix{
i^*W \ar[rrr]^<>(0.5){[\id^{i^*W}, -\alpha^*_{ij}W]^t}
\ar[dd]_{f_i} \ar[ddrrr]^0 &&& i^*W \oplus j^*W \ar[dd]^{f_i+f_j} \\\\
i^*Z \ar[rrr]_<>(0.5){[\id^{i^*Z}, -\alpha^*_{ij}Z]^t} &&&
\; i^*Z \oplus j^*Z \ko
}
\]
for each arrow $a \in {I(j,i)}$.
This clearly gives the equality
$f_j \circ \alpha^*_{ij}(W) = \alpha^*_{ij}(Z) \circ f_i$, which
is exactly the condition for the family $\{ f_i \}_{i \in I}$
to be a morphism lying in the category $\Mor_{I^\circ}(\T(e))$.

Thus, we get an exact sequence
\[
\xymatrix{
\prod_{j \in I} \prod_{i \in I - \{j\}} \prod_{{I(j,i)}}
\Hom_{\T(e)}(\Sigma i^* W, j^*Z) \ar[d] \\
\Hom_{\T(I)}(W, Z)
\ar[d]^{} \\
\Hom_{\Mor_{I^\circ}(\T(e))}(\mor_{I} W, \mor_{I} Z) \ar[d] \\
\; \; 0 .
}
\]
Notice that, passing in long exact sequence and using the Toda
condition of the hypothesis, it is not difficult to see that
there are vanishing groups
\[
\Hom_{\T(e)}(\Sigma^n i^*W, j^*Z) = 0 \ko \quad n>0 ,
\]
for each pair $(i,j) \in I \times I$.
This condition forces the first group in the sequence to be zero
and we get a bijection
\[
\xymatrix{
\Hom_{\T(I)}(W, Z) \ar[r]^<>(0.5){\simeq} &
\Hom_{\Mor_{I^\circ}(\T(e))}(\mor_{I} W, \mor_{I} Z) .
}
\]

Going from $l = n-2$ on, the situation is new. In this case a
distinguished triangle is the horizontal one in the commutative
diagram
\[
\xymatrix{
QL^{n-1}X \ar[d]^<>(0.5){\varepsilon^{n-1}} \ar[dr]^<>(0.5){u^{n-1}} &&&
\Sigma QL^{n-1}X \ar[d]^<>(0.5){\Sigma \varepsilon^{n-1}} \\
L^{n-1}X \ar[r]^<>(0.5){v^{n-1}} & QL^{n-2}X
\ar[r]^<>(0.5){\varepsilon^{n-2}} & L^{n-2}X \ar[r] & \Sigma L^{n-1}X .
}
\]
After applying the functor $\Hom_{\T(I)}(?,Z)$ as in the initial
step, we get a long exact sequence of abelian groups, that we report
as the vertical sequence in the commutative diagram
\[
\xymatrix{
\Hom_{\T(I)}(\Sigma L^{n-1}X, Z) \ar[d]
\ar@{ (->}[rr]^<>(0.5){\Sigma\ul\varepsilon^{n-1}} &&
\Hom_{\T(I)}(\Sigma QL^{n-1}X, Z) \\
\Hom_{\T(I)}(L^{n-2}X, Z) \ar[d]^{\ul\varepsilon^{n-2}} \\
\Hom_{\T(I)}(QL^{n-2}X, Z) \ar[d]^{\ul v^{n-1}}
\ar[drr]^{\ul u^{n-1}} \\
\Hom_{\T(I)}(L^{n-1}X, Z) \ar@{ (->}[rr]^<>(0.5){\ul\varepsilon^{n-1}} &&
\Hom_{\T(I)}(QL^{n-1}X, Z) .
}
\]

By the preceding step we know that the maps $\ul\varepsilon^{n-1}$
and $\Sigma\ul\varepsilon^{n-1}$ are mono. This entails that the
group $\Hom_{\T(I)}(\Sigma L^{n-1}X, Z)$ must be zero since the
group $\Hom_{\T(I)}(\Sigma QL^{n-1}X, Z)$ is trivial. Indeed,
using again the adjunction morphisms, we have
\begin{eqnarray*}
\Hom_{\T(I)}(\Sigma QL^{n-1}X, Z)
& = & \Hom_{\T(I)}(\coprod_{i \in I} i_!i^*L^{n-1}X, Z) \\
& \iso & \prod_{i \in I} \Hom_{\T(I)}(i_!i^*L^{n-1}X, Z) \\
& \iso & \prod_{i \in I} \Hom_{\T(I)}(i^*L^{n-1}X, i^*Z) \\
& = & 0 .
\end{eqnarray*}
The last equality comes from the condition found in the initial step.
It follows that the homomorphism $\ul\varepsilon^{n-2}$ is a mono.

Moreover, the morphisms in $\Hom_{\T(I)}(QL^{n-2}X, Z)$ that are
killed by the map $\ul v^{n-1}$ are exactly the same which are killed
by $\ul u^{n-1}$. Hence, the homomorphism  $\ul\varepsilon^{n-2}
\circ \ul v^{n-1}$, which is zero, must factor, up to iso, through the
group $\Hom_{\Mor_{I^\circ}(\T(e))}(\mor_{I} L^{n-2}X, \mor_{I} Z)$.

If we put everything together we find an exact sequence
\[
\xymatrix{
0 \ar[r] & \Hom_{\T(I)}(L^{n-2}X, Z) \ar[r]^<>(0.5)\eqiso &
\Hom_{\Mor_{I^\circ}(\T(e))}(\mor_{I} L^{n-2}X, \mor_{I} Z) \ar[r] & 0 .
}
\]
Now we can proceed by induction until we get an isomorphism
\[
\xymatrix{
\Hom_{\T(I)}(X, Z) \ar[r]^<>(0.5)\eqiso &
\Hom_{\Mor_{I^\circ}(\T(e))}(\mor_{I} X, \mor_{I} Z) .
}
\]

Clearly, this isomorphism is induced by the functor $\mor_I$.
Thus, we can use the canonical isomorphism between the categories
$\ul\Hom(I^\circ,\T(e))$ and $\Mor_{I^\circ}(\T(e))$ and the
claim follows.
\end{proof}

\subsection{Epivalence of the diagram functor} \label{ss:epidia}
Suppose we are given a triangulated derivator $\T$ of type
$\Dia$. Let us recall, from Appendix 1 in \cite{KellerNicolas11}
to which we inspire, that, for each $i$ in a small category $I$
lying in $\Dia$, the evaluation functor of presheaves
$F \mapsto F_i$ has a left adjoint denoted as $? \ten i$,
\[
\xymatrix{
\ul\Hom(I^\circ, \T(e)) \ar@<+0.5ex>[d]^{(?)_i} \\
\T(e) \ar@<+0.5ex>[u]^{? \ten i} & .
}
\]

For each $j \in I$, there is a canonical isomorphism
\[
(F \ten i)_j = \coprod_{I(j,i)}F .
\]

\begin{lemma}[Keller and Nicol{\`a}s, \cite{KellerNicolas11}]
\label{lemmaKeller}
Let $I$ be any small category in $\Dia$. For each $i \in I$
the triangle
\[
\xymatrix{
\T(I) \ar[r]^<>(0.5){\dia_I} & \ul\Hom(I^\circ,\T(e)) \\
\T(e) \ar[u]^{i_!} \ar[ur]_{? \ten i}
}
\]
commutes up to a canonical isomorphism.
\end{lemma}

\begin{remark} \label{rmk:lemmaexact}
Everything we have said in this section is true if we replace the
triangulated derivator $\T$ with the represented exact prederivator
$\ul\ca$. In particular, Lemma \ref{lemmaKeller} holds in this case.
Moreover, suppose that $F_e : \ul\ca(e) \to \T(e)$ is an exact
functor and $\ul{F_e}_I$ is the induced functor defined on presheaves.
It is not difficult to check that the diagram
\[
\xymatrix{
\ul\Hom(I^\circ, \ul\ca(e)) \ar@<+0.5ex>[d]^{(?)_i} \ar[r]^{\ul{F_e}_I} &
\ul\Hom(I^\circ, \T(e)) \ar@<+0.5ex>[d]^{(?)_i} \\
\ul\ca(e) \ar@<+0.5ex>[u]^{? \ten i} \ar[r]_{F_e} &
\T(e) \ar@<+0.5ex>[u]^{? \ten i}
}
\]
commutes in both directions.
\end{remark}

The following proposition gives a sufficient condition which allows
presheaves and their morphisms lift from $\ul\Hom(I^\circ,\T(e))$
to objects and morphisms of $\T(I)$. This tells us that the functor
$\dia_I$ is an epivalence when we consider its restriction to the
preimage of some subcategory of presheaves which enjoys this property
(that we call `Toda condition'). This result is one of the key
ingredients in the proof of the main theorem of this section.

\begin{proposition} \label{prop:ess-surj}
Let $\T$ be a triangulated derivator of type $\Dia_\f$.
Fix an arbitrary finite diagram $I$. Then,
\begin{itemize}
\item[a)] given an object $F$ in $\ul\Hom(I^\circ,\T(e))$, such
that the (Toda) condition
\[
\Hom_{\T(e)}(\Sigma^n F_i, F_j) = 0 \ko \quad n>0 ,
\]
holds for each pair $(i,j) \in I \times I$, there exists an object
$\tilde{F}$ in $\T(I)$ such that $\dia_{I}(\tilde{F})$ is canonically
isomorphic to $F$;
\item[b)] given a morphism $f : F \to F'$ in $\ul\Hom(I^\circ,\T(e))$,
such
that the (Toda) conditions
\[
\Hom_{\T(e)}(\Sigma^n F_i, F_j) = 0 \ko \,
\Hom_{\T(e)}(\Sigma^n F'_i, F'_j) = 0 \ko \,
\Hom_{\T(e)}(\Sigma^n F_i, F'_j) = 0 \ko
\]
$n > 0$, hold for each pair $(i,j) \in I \times I$, there exists a
morphism $\tilde{f} : \tilde{F} \to \tilde{F'}$ in $\T(I)$ such
that $\dia_{I}(\tilde{f})$ is canonically isomorphic to $f$.
\end{itemize}
\end{proposition}

\begin{proof}
a) {\em Step 1. An exact category with finite global dimension.}
Recall that every additive category can be endowed with an exact
structure by considering as conflations the split exact pairs
(\cf Example \ref{ex:exact cats}). In this way, we can consider
$\T(e)$ as an exact category and endow $\ul\Hom(I^\circ,\T(e))$
with an exact structure by taking as conflations the pointwise
split exact pairs.

Let us construct a projective-like resolution of the object $F$
lying in $\ul\Hom(I^\circ,\T(e))$. As a first step we can consider
the morphism
\[
PF := \coprod_{i \in I} F_i \ten i
\xymatrix{\ar[r]^{p^0} &} F
\]
defined by using the counit of the adjunctions $? \ten i \dashv (?)_i$.
Clearly, the morphism $p^0$ is a deflation. Indeed, evaluating $PX$
over $j$, we get a retraction in $\T(e)$
\[
\xymatrix{
(PF)_j = \coprod_{i \in I} \coprod_{I(j,i)} F_i \ar[r] & F_j .
}
\]
We can form the (possibly non split) conflation in
$\ul\Hom(I^\circ,\T(e))$
\[
\xymatrix{
KF \ar@{ >->}[r]^{i^0} & PF \ar@{->>}[r]^{p^0} & F .
}
\]

It may happen that $KF$ is already a projective-like object of
type $PKF$ and consequently the global dimension of the category
$\ul\Hom(I^\circ,\T(e))$ is one even if the maximal length of the
diagram $I$ is infinite (\cf the case where $I$ is $\N^\circ$ in
\cite[Appendix 1]{KellerNicolas11}).
But in general this is not the case and we have to iterate the
construction. Suppose that $n$ is the maximal length of a chain of
nonidentical arrows
\[
i_n \lra \ldots \lra i_1 \lra i_0
\]
in $I$. Then, it is not difficult to check by induction that the
object $K^{n+1}F$ is zero. This fact shows that $K^nF$ is projective
and that the length of a projective resolution is bounded by $n$, as
in the following chain complex
\[
\xymatrix{
0 \ar[r] & PK^nF = K^nF \ar@{ >->}[r]^<>(0.5){i^{n-1}} & PK^{n-1}F
\ar[r]^<>(0.5){u^{n-1}} &
\dots \ar[r] & PKF \ar[r]^{u^1} & PF \ar@{->>}[r]^{p^0} & F .
}
\]

Let us give a description of the morphisms in the above resolution.
We can start with $u^1$. For each arrow $a^1 : j \to i$ in $I$, let
$a^1_j \in \Hom(F_i \ten j, F_j \ten j)$ be the image of the identity
$\id^{F_j \ten j}$ under the composition of the morphisms
\begin{eqnarray*}
\Hom_{\ul\Hom(I^\circ,\T(e))}(F_j \ten j, F_j \ten j) \iso &
               \Hom_{\T(e)}(F_j, (F_j \ten j)_j) \\
               \to & \Hom_{\T(e)}(F_i, (F_j \ten j)_j) \\
               \iso & \Hom_{\ul\Hom(I^\circ,\T(e))}(F_i \ten j, F_j \ten j)
\end{eqnarray*}
induced by the adjoint pair $? \ten j \dashv (?)_j$ and the $2$-arrow
\[
\xymatrix{\ul\Hom(I^\circ,\T(e)) \;\; \rtwocell^{(?)_i}_{(?)_j} & \T(e)}
\]
associated with the natural transformation
\[
\xymatrix{e \rtwocell^j_i & I}
\]
defined by the morphism $a^1$. Similarly, the morphism $a^1$ and the
adjoint pair $? \ten i \dashv (?)_i$ induce a morphism
$a^1_i : F_i \ten j \to F_i \ten i$, image of $\id^{F_i \ten i}$. Thus,
for each morphism $a^1 : j \to i$ in $I$, we get an arrow
\[
\xymatrix{
F_i \ten j \ar[rr]^<>(0.5){[a^1_i, -a^1_j]^t} &&
F_i \ten i \oplus F_j \ten j \ar[rr]^<>(0.5){can} &&
\coprod_{i \in I}F_i \ten i.
}
\]
These are the components of the arrow
\[
\xymatrix{
(KF)_j \ten j = \coprod_{i \in I - \{j\}} \coprod_{I(j,i)}
F_i \ten j \ar[rr]^<>(0.5){u^1_j} && \coprod_{i \in I}F_i \ten i
}
\]
describing the morphism
\[
\xymatrix{
PKF = \coprod_{j \in I}(KF)_j \ten j \ar[rr]^<>(0.5){u^1} && PF .
}
\]
We recall that here $I(j,i)$ denotes the discrete category of arrows
$j \to i$ in $I$.

An explicit diagram, associated to a maximal length chain of nonidentical
arrows in $I$,
\[
i_n \lra \ldots  \lra i_1 \lra i_0 \ko
\]
which describes a part of the whole diagram associated to $I$, might help
\[
\xymatrix{
PKF \ar@{->>}[d] & 0 \ar[r] \ar[d] & (KF)_{i_1} \ar[r]
\ar@{=}[d] & (KF)_{i_2} \oplus (K^2F)_{i_2} \ldots
\ar[r] \ar@{->>}[d] & \ldots \\
KF \ar@{ >->}[d] & 0 \ar[r] \ar@{ >->}[d] & (KF)_{i_1} \ar[r]
\ar@{ >->}[d] & (KF)_{i_2} \ar[r] \ar@{ >->}[d] & \ldots \\
PF \ar@{->>}[d] & F_{i_0} \ar[r] \ar@{=}[d] &
F_{i_1} \oplus (KF)_{i_1} \ar[r] \ar@{->>}[d] & F_{i_2} \oplus (KF)_{i_2}
\ar[r] \ar@{->>}[d] & \ldots \\
F & F_{i_0} \ar[r] & F_{i_1} \ar[r] & F_{i_2} \ar[r] & \ldots & .
}
\]
By induction, one can get the description of all the maps in the
projective resolution of $F$. \\

{\em Step 2. Lifting a morphism of projective objects along
the diagram functor.}
Define two objects in $\T(I)$
\[
\tilde{PKF} := \coprod_{j \in I} j_!(KF)_j
= \coprod_{j \in I} \coprod_{i \in I - \{j\}}
\coprod_{I(j,i)} j_!(F_i)
\]
and
\[
\tilde{PF} := \coprod_{i \in I} i_!(F_i) .
\]
Lemma \ref{lemmaKeller} tells us that these objects lift the diagrams
$PKF$ and $PF$ along the functor $\dia_{I}$. Let us construct the
morphism $\tilde{u}^1$ that lifts $u^1$. For each arrow $a^1 : j \to i$
in the diagram $I$, let $\tilde{a^1}_i$ be the image of the identity
$\id^{i_!F_i}$ under the composition of morphisms
\begin{eqnarray*}
\Hom_{\T(I)}(i_!F_i, i_!F_i) \iso &
                 \Hom_{\T(e)}(F_i, i^*i_!F_i) \\
                 \to & \Hom_{\T(e)}(F_i, j^*i_!F_i) \\
                 \iso & \Hom_{\T(I)}(j_!F_i, i_!F_i)
\end{eqnarray*}
induced by the adjoint pairs $i_! \dashv i^*$ and $j_! \dashv j^*$
and the $2$-arrow
\[
\xymatrix{\T(I) \rtwocell^{i^*}_{j^*} & \T(e)}
\]
associated with the natural transformation
\[
\xymatrix{e \rtwocell^j_i & I}
\]
defined by the morphism $a^1$. Let us denote by $\tilde{a^1}_j$ the
morphism $j_!(F_{a^1})$. Thus, for each morphism $a^1 : j \to i$ in $I$,
we get an arrow
\[
\xymatrix{
j_!F_i \ar[rr]^<>(0.5){[\tilde{a^1}_i, -\tilde{a^1}_j]^t} &&
i_!F_i \oplus j_!F_j \ar[rr]^<>(0.5){can} &&
\coprod_{i \in I} i_!F_i .
}
\]
These are the components of an arrow
\[
\xymatrix{
j_!(KF)_j = \coprod_{i \in I - \{j\}} \coprod_{I(j,i)}
j_!F_i \ar[rr]^<>(0.5){\tilde{u^1}_j} && \coprod_{i \in I} i_!F_i
}
\]
describing the morphism
\[
\xymatrix{
\tilde{PKF} = \coprod_{j \in I} j_!(KF)_j
\ar[rr]^<>(0.5){\tilde{u^1}} && \tilde{PF} .
}
\]
It is clear by the construction and by Lemma \ref{lemmaKeller} that
this morphism lifts $u^1$, \ie, the morphism $\dia_{I}(\tilde{u^1})$
is isomorphic to $u^1$.

Analogously, we can lift all the morphisms
$u^l : PK^lF \to PK^{l-1}F$, $l \in \{ 1, \ldots n \}$, of the
projective resolution of $F$ to morphisms $\tilde{u^l} : \tilde{PK^lF}
\to \tilde{PK^{l-1}F}$, $l \in \{ 1, \ldots n \}$, in $\T(I)$. \\

{\em Step 3. Lifting a projective resolution along the diagram
functor.} We construct a sequence of morphisms in $\T(I)$ which
lifts the projective resolution that we have constructed in
$\ul\Hom(I^\circ,\T(e))$,
\[
\xymatrix{
0 \ar[r] & PK^nF = K^nF \ar@{ >->}[r]^<>(0.5){u^n} & PK^{n-1}F
\ar[r]^<>(0.5){u^{n-1}} & \dots \ar[r]^<>(0.5){u^2} &
PKF \ar[r]^<>(0.5){u^1} & PF \ar@{->>}[r]^<>(0.5){u^0} & F .
}
\]
Recall that this is a strictly acyclic resolution, \ie, there are
conflations
\[
\xymatrix{
K^{l+1}F \ar@{ >->}[r]^<>(0.5){i^{l}} & PK^{l}F
\ar@{->>}[r]^<>(0.5){p^{l}} & K^{l}F
} \ko l \in \N \ko
\]
such that $u^l = i^{l-1}p^l$, for all $l \in \N-\{0\}$.

We can start by lifting the morphism $u^n = i^{n-1}$ in the conflation
\[
\xymatrix{
K^nF \ar@{ >->}[r]^<>(0.5){u^n} & PK^{n-1}F
\ar@{->>}[r]^<>(0.5){p^{n-1}} & K^{n-1}F
}
\]
along the diagram functor to a morphism
$\tilde{u^n} : \tilde{K^nF} \to \tilde{PK^{n-1}F}$, following the
step 2. Let us consider a distinguished triangle
\[
\xymatrix{
\tilde{K^nF} \ar[r]^<>(0.5){\tilde{u^n}} & \tilde{PK^{n-1}F}
\ar[r]^<>(0.5){\tilde{p^{n-1}}} &
\tilde{K^{n-1}F} \ar[r] & \Sigma \tilde{K^nF} \ko
}
\]
which extends $\tilde{u^n}$ in $\T(I)$.
The image of the composition of morphisms
$\tilde{p^{n-1}} \circ \tilde{u^n}$ under the functor $\dia_I$
is the zero morphism $\dia_I(\tilde{p^{n-1}}) \circ u^n$.
Therefore, there exists a {\em unique} morphism
$\varphi^{n-1} : K^{n-1}F \to \dia_I(\tilde{K^{n-1}F})$
such that $\dia_I(\tilde{p^{n-1}})$ equals the composition
$\varphi^{n-1} \circ p^{n-1}$.

For each $m \in I$, after applying the (triangulated) functor $m^*$
to the distinguished triangle above, we get a distinguished triangle
in $\T(e)$
\[
\xymatrix{
(K^nF)_m \ar[rr]^<>(0.5){(u^n)_m} && (PK^{n-1}F)_m
\ar[rr]^<>(0.5){(\dia_{I}(\tilde{p^{n-1}}))_m} &&
(\dia_{I}(\tilde{K^{n-1}F}))_m \ar[r] & \Sigma (K^nF)_m
}
\]
which is split since $u^n$ is an inflation. This shows that
$(\dia_{I}(\tilde{K^{n-1}F}))_m$ is the cokernel of $(u^n)_m$, for
all $m \in I$.
Thus, $(\varphi^{n-1})_m$ must be a (canonical) isomorphism, for all
$m \in I$. It follows that $\dia_{I}(\tilde{p^{n-1}})$ is the cokernel
of $u$. Hence the canonical morphism $\varphi^{n-1}$ is invertible.

Let us lift the morphism $u^{n-1}$ along $\dia_I$ to a morphism
$\tilde{u^{n-1}}$ as in step 2. It is not difficult to verify,
using the hypotheses by induction over $\N$, that
\[
\Hom_{\T(e)}(\Sigma^n i^*(\tilde{PK^{l+2}F}), j^*(\tilde{PK^{l}F})) = 0
\ko \quad n>0 ,
\]
for all $l \in \N$. Since $\dia_{I}(\tilde{u^{n-1}} \circ \tilde{u^n})$
is zero, Proposition \ref{prop:fullfaith} with $l = n-2$ applies and
tells us that the composition $\tilde{u^{n-1}} \circ \tilde{u^n}$
already vanishes in $\T(I)$. As a consequence, we get a morphism
$\tilde{i^{n-2}} : \tilde{K^{n-1}F} \to \tilde{PK^{n-2}F}$ such that
the morphism $\tilde{u^{n-1}} : \tilde{PK^{n-1}F} \to \tilde{PK^{n-2}F}$
equals the composition $\tilde{i^{n-2}} \circ \tilde{p^{n-1}}$, whose
image under $\dia_{I}$ is the inflation $i^{n-2} : K^{n-1}F \to PK^{n-2}F$.

It is clear that we can iterate this construction until we get
a distinguished triangle
\[
\xymatrix{
\tilde{KF} \ar[r]^<>(0.5){\tilde{i^0}} & \tilde{PF}
\ar[r]^<>(0.5){\tilde{p^0}} &
\tilde{F} \ar[r] & \Sigma \tilde{KF} \ko
}
\]
whose image under $\dia_{I}$ gives (up to a canonical iso) a
conflation
\[
\xymatrix{
KF \ar@{ >->}[r]^<>(0.5){i^0} & PF
\ar@{->>}[r]^<>(0.5){p^{0} = u^{0}} & F .
}
\]
Thus, by lifting a projective resolution of $F$, we have constructed
an object $\tilde{F}$ in the category $\T(I)$ which lifts $F$ along
the diagram functor $\dia_I$. \\

b) {\em Step 1. Lifting a square of projective objects along the
diagram functor.}
Suppose that we are given a commutative square lying in
$\ul\Hom(I^\circ,\T(e))$
\[
\xymatrix{
PKF \ar[r]^<>(0.5){u} \ar[d]^h & PF \ar[d]^g \\
PKF' \ar[r]^<>(0.5){u'} & PF' .
}
\]
By looking at the structure of the objects $PKF$ and $PF$
(\cf with Step 1 in the proof of Claim a), we can see that the
morphisms $g$ and $h$ are completely determined by their
components, which are of the type $g_i \ten j$ and $h_i \ten j$,
respectively. So, by Lemma \ref{lemmaKeller}, it is sufficient to
lift these components to $j_!(g_i)$ and $j_!(h_i)$ and then
reconstruct to obtain morphisms $\tilde h$ and $\tilde g$ which
lift $h$ and $g$, respectively.

Thanks to the hypotheses, we can use the second step in the proof
of Claim a) of this proposition and lift the morphisms $u$ and $u'$
to the category $\T(I)$. We get a commutative square
\[
\xymatrix{
\tilde{PKF} \ar[r]^<>(0.5){\tilde{u}} \ar[d]^{\tilde{h}} &
\tilde{PF} \ar[d]^{\tilde{g}} \\
\tilde{PKF'} \ar[r]^<>(0.5){\tilde{u'}} & \tilde{PF'}
}
\]
in $\T(I)$ whose image under the diagram functor is canonically
isomorphic to the given square.

To verify commutativity, let us use the hypotheses about $F$ and $F'$
in the same way as above. It is easy to check that the vanishing
\[
\Hom_{\T(e)}(\Sigma^n i^*(\tilde{PK^{l+2}F}), j^*(\tilde{PK^{l}F'})) = 0
\ko \quad n>0 ,
\]
holds for all $l \in \N$. Since
$\dia_{I}(\tilde{g} \circ \tilde{u} - \tilde{u'} \circ \tilde{h})$
is zero, Proposition \ref{prop:fullfaith} with $l = n-2$ applies and
tells us that the square commutes. \\

{\em Step 2. Lifting a morphism of projective resolutions along the
diagram functor.}
Let $f : F \to F'$ be an arbitrary morphism in the category
$\ul\Hom(I^\circ,\T(e))$. We know from the proof of Claim a) that we
can construct projective resolutions of $F$ and $F'$ whose lengths are
bounded by the maximal length of a chain of nonidentical arrows in $I$,
\[
\xymatrix{
0 \ar[r] & PK^nF = K^nF \ar@{ >->}[r]^<>(0.5){u^n} & PK^{n-1}F
\ar[r]^<>(0.5){u^{n-1}} & \dots \ar[r]^<>(0.5){u^2} &
PKF \ar[r]^<>(0.5){u^1} & PF \ar@{->>}[r]^<>(0.5){u^0} & F
}
\]
and
\[
\xymatrix{
0 \ar[r] & PK^nF' = K^nF' \ar@{ >->}[r]^<>(0.5){u'^n} & PK^{n-1}F'
\ar[r]^<>(0.5){u'^{n-1}} & \dots \ar[r]^<>(0.5){u'^2} &
PKF' \ar[r]^<>(0.5){u'^1} & PF' \ar@{->>}[r]^<>(0.5){u'^0} & F' .
}
\]
Since these resolutions are made of projective objects, the
morphism $f : F \to F'$ extends to a morphism of resolutions lying
in $\ul\Hom(I^\circ,\T(e))$
\[
\xymatrix{
0 \ar[r] \ar[d] & PK^nF = K^nF \ar@{ >->}[r]^<>(0.5){u^n}
\ar[d]^{f^n} & PK^{n-1}F \ar[r]^<>(0.5){u^{n-1}} \ar[d]^{f^{n-1}} &
\dots \ar[r]^<>(0.5){u^2} & PKF \ar[r]^<>(0.5){u^1} \ar[d]^{f^1} &
PF \ar@{->>}[r]^<>(0.5){u^0} \ar[d]^{f^0} & F \ar[d]^f \\
0 \ar[r] & PK^nF' = K^nF' \ar@{ >->}[r]^<>(0.5){u'^n} & PK^{n-1}F'
\ar[r]^<>(0.5){u'^{n-1}} & \dots \ar[r]^<>(0.5){u'^2} &
PKF' \ar[r]^<>(0.5){u'^1} & PF' \ar@{->>}[r]^<>(0.5){u'^0} & F' .
}
\]

Let us start by lifting the square of projective objects
on the left side of the commutative diagram
\[
\xymatrix{
K^nF \ar@{ >->}[r]^<>(0.5){u^n} \ar[d]^{f^n} & PK^{n-1}F
\ar@{->>}[r]^<>(0.5){p^{n-1}} \ar[d]^{f^{n-1}} &
K^{n-1}F \ar[d]^{g^{n-1}} \\
K^nF' \ar@{ >->}[r]^<>(0.5){u'^n} & PK^{n-1}F'
\ar@{->>}[r]^<>(0.5){p'^{n-1}} & K^{n-1}F' \ko
}
\]
where the morphism $g^{n-1}$ is induced by the universal property
of the cokernel. It is routine to verify that the sufficient
conditions which allow us using the first step of this claim hold.
We get a commutative square
\[
\xymatrix{
\tilde{K^nF} \ar[r]^<>(0.5){\tilde{u^n}} \ar[d]^{\tilde{f^n}} &
\tilde{PK^{n-1}F} \ar[d]^{\tilde{f^{n-1}}} \\
\tilde{K^nF'} \ar[r]^<>(0.5){\tilde{u'^n}} & \tilde{PK^{n-1}F'} \ko
}
\]
whose image under the diagram functor is canonically isomorphic to
the given one.

Let us consider an extension in $\T(I)$ of the latter square
to a morphism of distinguished triangles
\[
\xymatrix{
\tilde{K^nF} \ar[r]^<>(0.5){\tilde{u^n}} \ar[d]^{\tilde{f^n}} &
\tilde{PK^{n-1}F} \ar[r]^<>(0.5){\tilde{p^{n-1}}}
\ar[d]^{\tilde{f^{n-1}}} &
\tilde{K^{n-1}F} \ar[d]^{\tilde{g^{n-1}}} \ar[r] &
\Sigma \tilde{K^nF} \ar[d]^<>(0.5){\Sigma \tilde{f^n}} \\
\tilde{K^nF'} \ar[r]^<>(0.5){\tilde{u'^n}} & \tilde{PK^{n-1}F'}
\ar[r]^<>(0.5){\tilde{p'^{n-1}}} & \tilde{K^{n-1}F'} \ar[r] &
\Sigma \tilde{K^nF'} .
}
\]
The images of the compositions of morphisms
$\tilde{p^{n-1}} \circ \tilde{u^n}$ and
$\tilde{p'^{n-1}} \circ \tilde{u'^n}$ under the functor $\dia_I$
are the zero morphisms $\dia_I(\tilde{p^{n-1}}) \circ {u^n}$ and
$\dia_I(\tilde{p'^{n-1}}) \circ {u'}^n$, respectively. It follows
that there exist {\em unique} morphisms $\varphi^{n-1} : K^{n-1}F
\to \dia_I(\tilde{K^{n-1}F})$ and $\varphi '^{n-1} : K^{n-1}F' \to
\dia_I(\tilde{K^{n-1}F'})$ such that $\dia_I(\tilde{p^{n-1}})$ and
$\dia_I(\tilde{p'^{n-1}})$ respectively equal compositions
$\varphi^{n-1} \circ p^{n-1}$ and $\varphi '^{n-1} \circ p'^{n-1}$.
Moreover, there is an isomorphism $\dia_I(\tilde{g^{n-1}}) \circ
\varphi^{n-1} = \varphi '^{n-1} \circ {g}^{n-1}$.

For each $m \in I$, we apply the (triangulated) functor $m^*$ to
the morphism of distinguished triangles above and get a morphism
of distinguished triangles in $\T(e)$
\[
\xymatrix{
(K^nF)_m \ar[rr]^<>(0.5){(u^n)_m} \ar[d]^{(f^n)_m} && (PK^{n-1}F)_m
\ar[rr]^<>(0.5){(\dia_{I}(\tilde{p^{n-1}}))_m}
\ar[d]^{(f^{n-1})_m} && (\dia_{I}(\tilde{K^{n-1}F}))_m
\ar[d]^{(\dia_{I}(\tilde{g^{n-1}}))_m}
\ar[r] & \Sigma (K^nF)_m \ar[d]^<>(0.5){\Sigma (f^n)_m} \\
(K^nF')_m \ar[rr]^<>(0.5){(u'^n)_m} && (PK^{n-1}F')_m
\ar[rr]^<>(0.5){(\dia_{I}(\tilde{p'^{n-1}}))_m} &&
(\dia_{I}(\tilde{K^{n-1}F'}))_m \ar[r] & \Sigma (K^nF')_m .
}
\]
Here, distinguished triangles are split since $u^n$ and $u'^n$ are
inflations. This shows that the morphisms
$(\dia_{I}(\tilde{K^{n-1}F}))_m$ and $(\dia_{I}(\tilde{K^{n-1}F'}))_m$
are the cokernels of $(u^n)_m$ and $(u'^n)_m$, for all $i \in I$.
Thus, $(\varphi^{n-1})_m$ and $(\varphi '^{n-1})_m$ must be (canonical)
isomorphisms, for all $m \in I$. It follows that the morphisms
$\dia_{I}(\tilde{p^{n-1}})$ and $\dia_{I}(\tilde{p'^{n-1}})$ are the
cokernels of $u^n$ and $u'^n$. Hence, the morphisms $\varphi^{n-1}$
and $\varphi '^{n-1}$ are invertible.
Finally, the universal property of cokernels induces a (canonical)
isomorphism from $g^{n-1}$ to $\dia_{I}(\tilde{g^{n-1}})$.

At this point we can proceed by lifting the square of projective
objects
\[
\xymatrix{
PK^{n-1}F \ar[r]^<>(0.5){u^{n-1}} \ar[d]^{f^{n-1}} & PK^{n-2}F
\ar[d]^{f^{n-2}} \\
PK^{n-1}F' \ar[r]^<>(0.5){u'^{n-1}} & PK^{n-2}F' \ko
}
\]
according to the Step 1 of this Claim b), in order to get a
commutative square
\[
\xymatrix{
\tilde{PK^{n-1}F} \ar[r]^<>(0.5){\tilde{u^{n-1}}}
\ar[d]^{\tilde{f^{n-1}}} &
\tilde{PK^{n-2}F} \ar[d]^{\tilde{f^{n-2}}} \\
\tilde{PK^{n-1}F'} \ar[r]^<>(0.5){\tilde{u'^{n-1}}} &
\tilde{PK^{n-2}F'} \ko
}
\]
whose image under the diagram functor is isomorphic to the given
one.

Sufficient conditions in order to apply Proposition
\ref{prop:fullfaith} hold. Since the images
$\dia_{I}(\tilde{u^{n-1}} \circ \tilde{u^n})$ and
$\dia_{I}(\tilde{u'^{n-1}} \circ \tilde{u'^n})$ are both
isomorphic to zero, it follows that the compositions
$\tilde{u^{n-1}} \circ \tilde{u^n}$ and
$\tilde{u'^{n-1}} \circ \tilde{u'^n}$ already vanish in $\T(I)$.
Therefore, there exist morphisms $\tilde{i^{n-2}}$ and
$\tilde{i'^{n-2}}$, whose images are the inflations $i^{n-2}$ and
$i'^{n-2}$, such that $\tilde{u^{n-1}}$ and $\tilde{u'^{n-1}}$ are
respectively isomorphic to the compositions $\tilde{i^{n-2}} \circ
\tilde{p^{n-1}}$ and $\tilde{i'^{n-2}} \circ \tilde{p'^{n-1}}$.

It is clear that we can continue by iterating this construction until
the cohomological degree is $0$. In the end, we will get a morphism
$\tilde{f} : \tilde{F} \to \tilde{F'}$, which lifts the given
morphism $f : F \to F'$ to the triangulated category $\T(I)$ along
the diagram functor $\dia_{I}$.
\end{proof}

\subsection{Invertibility of the diagram functor} \label{ss:inv}
Let us also recall that we say that an additive functor
$F : \ce \to \ct$ from an exact category $\ce$ to a
triangulated category $\ct$ is {\em exact} or, equivalently,
a {\em $\partial$-functor} if to any conflation in $\ce$,
\[
\xymatrix{
X \ar@{ >->}[r]^<>(0.5)u & Y \ar@{->>}[r]^<>(0.5)v & Z \ko
}
\]
it {\em functorially} associates a morphism $\partial(\eps)$
such that the diagram
\[
\xymatrix{
F(X) \ar[r]^<>(0.5){F(u)} & F(Y) \ar[r]^<>(0.5){F(v)} &
F(Z) \ar[r]^<>(0.5){\partial(\eps)} & \Sigma F(X)
}
\]
is a distinguished triangle of $\ct$.
It is straightforward to see that an additive $2$-morphism
$\mu : F \to F'$ of exact functors is an {\em exact}
$2$-morphism if to the conflation above it {\em functorially}
associates a morphism of distinguished triangles of $\ct$,
\[
\xymatrix{
F(X) \ar[r]^<>(0.5){F(u)} \ar[d]_{\mu^X} &
F(Y) \ar[r]^<>(0.5){F(v)} \ar[d]_{\mu^Y} &
F(Z) \ar[r]^<>(0.5){\partial(\eps)} \ar[d]^{\mu^Z} &
\Sigma F(X) \ar[d]^{\Sigma\mu^X} \\
F'(X) \ar[r]^<>(0.5){F'(u)} & F'(Y) \ar[r]^<>(0.5){F'(v)} &
F'(Z) \ar[r]^<>(0.5){\partial'(\eps)} & \Sigma F'(X) \; .
}
\]

In any case, let us remark that, in the presence of the Toda
condition as in the items a) and b) of the following theorem,
it would be equivalent if the basic exact $2$-morphism $\mu$ of
basic exact functors $F$, $F'$ that we want to extend were
taken {\em weakly} exact only. Indeed, it is not difficult to
see that {\em all} the weakly exact functors are exact in the
presence of such a condition.

\begin{theorem} \label{thm:ext1}
Let $\ca$ be an exact category and $\ct$ a triangulated
category such that there exists an isomorphism $\ct \iso \T(e)$
for some triangulated derivator $\T$ of type $\Dia_\f$.
\begin{itemize}
\item[a)] Suppose that $F : \ca \to \ct$ is an exact functor such
that the (Toda) condition
\[
\Hom_{\ct}(\Sigma^n F(X), F(Y)) = 0 \ko \quad n>0 ,
\]
holds for all $X, Y$ in $\ca$. Then, there exists an exact morphism
of prederivators $\tilde{F} : \ul\ca \to \T$ (of type $\Dia_\f$)
having base $F$ (up to iso).
The morphism $\tilde{F}$ is unique up to a unique natural isomorphism.
\item[b)] Suppose that $F$ and $F'$ are exact morphisms $\ul\ca \to \T$
and that the (Toda) conditions
\[
\Hom_{\T(e)}(\Sigma^n F_e(X), F_e(Y)) = 0 \ko \quad n>0 \ko
\]
\[
\Hom_{\T(e)}(\Sigma^n F'_e(X), F'_e(Y)) = 0 \ko \quad n>0 \ko
\]
\[
\Hom_{\T(e)}(\Sigma^n F_e(X), F'_e(Y)) = 0 \ko \quad n>0 \ko
\]
hold for all $X, Y$ in $\ul\ca(e) = \ca$. Then, the map
\[
\Hom_{\ul\HOM_{ex}(\ul\ca,\T)}(F,F') \xymatrix{\ar[r] &}
\Hom_{\ul\Hom_{ex}(\ca,\T(e))}(F_e,F'_e) \ko \quad \mu \mapsto \mu_e
\]
is bijective.
\end{itemize}
\end{theorem}

\begin{proof}
a) For any diagram $I$ in $\Dia_\f$, let us denote by
\[
\ul{F}_I : \ul\Hom(I^\circ,\ul\ca(e)) \xymatrix{\ar[r] &}
\ul\Hom(I^\circ,\T(e))
\]
the induced functor defined on presheaves.
Let $\V(I)$ be the essential image of the functor $\ul{F}_I$, which
is an additive subcategory of $\ul\Hom(I^\circ,\ul\ca(e))$. Because
of the hypotheses of claim a), every object lying in $\V(I)$ clearly
satisfies the hypotheses of claim a) of Proposition \ref{prop:ess-surj}.
Therefore, according to that claim they lift to $\T(I)$. Thus, we
can form an additive category $\U(I)$ by taking all the preimages
under $\dia_{I}$ of all the objects lying in $\V(I)$ and then
considering the additive full subcategory they span in $\T(I)$.

By construction, the restriction of the diagram functor
$\dia_{I}$ to $\U(I)$ is essentially surjective. Actually, it is
fully faithful, too. Indeed, for every pair of objects $X$ and $Z$
in $\U(I)$, the hypotheses of Proposition \ref{prop:fullfaith}
hold. Hence, the conclusion gives us a bijection
\[
\xymatrix{
\Hom_{\T(I)}(X, Z) \ar[r]^<>(0.5)\sim &
\Hom_{\ul\Hom(I^\circ,\T(e))}(\dia_I X, \dia_I Z) \ko
}
\]
induced by $\dia_I$, which clearly reduces to a bijection
\[
\xymatrix{
\Hom_{\U(I)}(X, Z) \ar[r]^<>(0.5)\sim &
\Hom_{\V(I)}(\dia_I X, \dia_I Z) \ko
}
\]
induced by its restriction to the subcategory $\U(I)$, for all
pairs of objects $X$ and $Z$ in $\U(I)$.

Said otherwise, the restriction of the diagram functor to
$\U(I)$ is an (additive) equivalence of (additive) categories.
Nevertheless, we will show that its (inverse) composition with
$\ul{F}_I \circ \dia^\A_I$ gives rise to a $\partial$-functor
from $\A(I)$ to $\T(I)$.

Consider the following commutative diagram of additive categories
\[
\xymatrix{
\ul\Hom(I^\circ,\ul\ca(e)) \ar[r] & \V(I) \ar[r] &
\ul\Hom(I^\circ,\T(e)) \\
\ul\ca(I) \ar[u]^\eqiso \ar@{-->}[r]_{} & \U(I)
\ar[u]_\simeq \ar[r] & \; \T(I) \ar[u]_{\dia_{I}}
}
\]
that we have constructed for any diagram $I$ in $\Dia_\f$. The dashed
arrow is obtained by composition with the functor which is inverse to
the vertical equivalence in the centre of the diagram.
It clearly induces an additive functor
$\tilde{F}_I : \ul\ca(I)| \to \T(I)|$
for every diagram $I$ lying in $\Dia_\f$. Here and in the sequel,
the symbol $|$ means the image under the forgetful functor.

This construction gives rise to a diagram of {\em additive} prederivators,
\[
\xymatrix{
\ul\ca| \ar[r] & \V & \U \ar[l]_\sim \ar[r] & \T| .
}
\]
After inverting the equivalence and composing, we get an additive
morphism of additive prederivators $\tilde{F} : \ul\ca| \to \T|$.
It is indeed easy to check that, for all the morphisms $u : I \to J$
in $\Dia_\f$, we have an invertible natural transformation of functors
$u^*\tilde{F}_J \iso \tilde{F}_Iu^*$.
Moreover, it is clear that the base of $\tilde{F}$ is naturally
isomorphic to the functor $F$ as an {\em additive} functor.

We want to show that, with respect to the exact and triangulated
structures of the categories $\A(I)$ and $\T(I)$, the functor
$\tilde{F}_I$ is weakly exact according to definition
\ref{def:weakly}. Suppose we are given a pair of morphisms in $\U(I)$
\[
\xymatrix{
X \ar[r]^f & Y \ar[r]^g & Z
}
\]
which is the image under $\tilde{F}_I$ of some conflation $\eps$
lying in $\ul\ca(I)$.
We can extend the morphism $f$ to a distinguished triangle
\[
\xymatrix{
X \ar[r]^f & Y \ar[r]^h & W \ar[r]^l & \Sigma X
}
\]
in the triangulated category $\T(I)$.
Since the composition of $f$ with $g$ vanishes there is an arrow
$\varphi : W \to Z$ which lifts $g$. After applying to this
distinguished triangle the (triangulated) functor $i^*$, we get a
distinguished triangle in $\T(e)$
\[
\xymatrix{
i^*X \ar[r]^{i^*f} & i^*Y \ar[r]^{i^*h} & i^*W
\ar[r]^<>(0.5){\delta_i l} & \Sigma i^*X \ko
}
\]
for any $i \in I$.
As the morphisms $f$ and $g$ are in the image of a conflation of
$\A(I)$ and $F : \ca \to \ct$ is supposed to be a $\partial$-functor,
we get that this triangle fits in a morphism of distinguished triangles
\[
\xymatrix{
i^*X \ar[r]^{i^*f} \ar@{=}[d] & i^*Y \ar[r]^{i^*h} \ar@{=}[d] &
i^*W \ar[r]^<>(0.5){\delta_i l} \ar[d]^<>(0.5){\psi_i}_\wr &
\Sigma i^*X \ar@{=}[d] \\
i^*X \ar[r]^{i^*f} & i^*Y \ar[r]^{i^*g} & i^*Z
\ar[r]^<>(0.5){m_i} & \Sigma i^*X .
}
\]
Here, the invertible arrow $\psi_i : i^*W \to i^*Z$ which makes the
diagram commute exists thanks to the axioms of triangulated categories.
Since the objects $X$ and $Z$ belong to the subcategory $\U(I)$, their
images under the (restriction of the) functor $\dia_I$ are in $\V(I)$.

Hence, we know from the Toda condition that the abelian group
$\Hom_{\T(e)}(\Sigma i^*X,i^*Z)$ must be zero for each $i \in I$.
This condition makes clear that the isomorphism $\psi_i$ is
uniquely determined. It follows that $\psi_i$ must be canonically
isomorphic to $i^*\varphi$, for all $i \in I$. Now we can use
Axiom {\bf Der 2} of derivators which ensures that $\varphi$
actually is an isomorphism. This shows that the pair of morphisms
$(f \vir g)$ actually extends to a distinguished triangle of $\T(I)$
\[
\xymatrix{
X \ar[r]^f & Y \ar[r]^g & Z \ar[r]^<>(0.5)m & \Sigma X \ko
}
\]
for {\em some} morphism $m$.

Now, since we have checked that the hypotheses of Proposition
\ref{prop:redundancy} b) are fulfilled, we know that the morphism
$\tilde{F} : \ul\ca \to \T$ has the {\em property} of being an
exact morphism of prederivators. We have already checked at the
beginning of the proof of \ref{prop:redundancy} b) that the base
of our morphism $\tilde F$ is canonically an exact functor.
Therefore, it must coincide (up to a canonical iso) with $F$ as
an {\em exact} functor.
The uniqueness of our construction also follows by the Toda
condition.

b) We want to show that the functor
\[
\ul\HOM_{ex}(\ul\ca,\T) \xymatrix{\ar[rr]^{\ev_e} &&}
\ul\Hom_{ex}(\ul\ca(e),\T(e)), \qquad F \mapsto F_e
\]
induces a bijection
\[
\xymatrix{
\Hom_{\ul\HOM_{ex}(\ul\ca,\T)}(F, F') \ar[r]^<>(0.5)\sim &
\Hom_{\ul\Hom_{ex}(\ul\ca(e),\T(e))}(F_e, F'_e) .
}
\]

It is known (cf. \cite[8.1]{Keller90}) the easy fact that,
under the hypotheses of b) (Toda conditions), the set of
natural transformations of $\partial$-functors from $F_e$
to $F'_e$ is in bijection with the set of natural transformations
of their underlying additive functors.
Moreover, after item b) of Prop. \ref{prop:redundancy} we
know that the set of $2$-morphisms of exact morphisms from
$F$ to $F'$ is in bijection with the subset of $2$-morphisms
of their underlying additive $2$-morphisms from $F|$ to $F'|$.

Thus, to prove the claim in item b) of this theorem, we only have
to show that the functor
\[
\ul\HOM_{add}(\ul\ca|,\T|) \xymatrix{\ar[rr]^{{\ev|}_e} &&}
\ul\Hom_{add}(\ul\ca(e)|,\T(e)|), \qquad F| \mapsto F|_e
\]
induces a bijection
\[
\xymatrix{
\Hom_{\ul\HOM_{add}(\ul\ca|,\T|)}(F|, F'|) \ar[r]^<>(0.5)\sim &
\Hom_{\ul\Hom_{add}(\ul\ca(e)|,\T(e)|)}(F|_e, F'|_e) .
}
\]

{\em From now on we omit the symbol $|$ .} Let us factor the functor
$\ev_e$ as follows
\[
\xymatrix{
\ul\HOM_{add}(\ul\ca,\T) \ar[rr]^<>(0.5){\ev_e}
\ar[dr]_{\ul{\dia}_{\ul\ca}} && \ul\Hom_{add}(\ul\ca(e),\T(e)) \\
& \ul\HOM_{add}(\ul{\A},\ul{\T(e)}) \ar[ur]_\sim & .
}
\]
Here, $\ul{\T(e)}$ is the additive prederivator which associates the
additive category $\ul\Hom(I^\circ,\T(e))$ with a diagram $I \in \Dia_\f$.
The additive category $\ul\HOM_{add}(\ul\ca,\ul{\T(e)})$ contains as
objects the morphisms of the underlying additive prederivators, \ie,
morphisms $F : \ul\ca \to \ul{\T(e)}$ such that $F_I$ are additive functors
for all diagrams $I \in \Dia_\f$ with compatibility conditions. The
additive functor $\ul{\dia}_{\ul\ca}$ is induced by the morphism of the
underlying additive prederivators $\dia : \T \to \ul{\T(e)}$ by composition.
Moreover, the additive functor of additive categories on the right of the
last diagram above is given by evaluation on the terminal diagram $e$.
It is not hard to directly check that it is an equivalence of categories.

Thus, in order to prove the claim, it is enough to show that the
morphism $\dia$ induces a bijection
\[
\xymatrix{
\Hom_{\ul\HOM_{add}(\ul\ca,\T)}(F, F') \ar[r]^<>(0.5)\sim &
\Hom_{\ul\HOM_{add}(\ul\ca,\ul{\T(e)})}(\dia \circ F, \dia \circ F') .
}
\]
By the contravariant version of Lemma A.5 in \cite{Keller91}, we can
check this isomorphism locally, \ie, we have to show that the functor
$\dia_I : \T(I) \to \ul{\T(e)}(I)$ induces a bijection
\[
\xymatrix{
\Hom_{\ul\Hom_{add}(\ul\ca(J),\T(I))}(F_I \circ u^*, F'_I \circ u^*)
\ar[d]^<>(0.5)\sim \\
\Hom_{\ul\Hom_{add}(\ul\ca(J),\ul{\T(e)}(I))}((\dia \circ F)_I \circ u^*,
(\dia \circ F')_I \circ u^*) \ko
}
\]
for each morphism $u : I \to J$ in $\Dia_\f$. This is
{\em a posteriori} true if the map
\[
\xymatrix{
\Hom_{\T(I)}(\Sigma^n F_I X, F'_I Y) \ar[r] &
\Hom_{\ul{\T(e)}(I)}(\Sigma^n \dia_I(F_I X), \dia_I(F'_I Y))
}
\]
is a bijection for all $X$, $Y$ in $\ul\ca(I)$. But this is true,
since, by the argument of the proof of Proposition
\ref{prop:fullfaith}, we know that, for all $n \in \N$, there is
an exact sequence
\[
\xymatrix{
\prod_{j \in I} \prod_{i \in I - \{j\}} \prod_{I^\circ(i,j)}
\Hom_{\T(e)}(\Sigma^{n+1} (\dia_{I}(F_I X))_i,
(\dia_{I}(F'_I Y))_j) \ar[d] \\
\Hom_{\T(I)}(\Sigma^n F_I X, F'_I Y)
\ar[d]^{} \\
\Hom_{\ul{\T(e)}(I)}(\Sigma^n \dia_{I}(F_I X), \dia_{I}(F'_I Y)) \ar[d] \\
\;\;\;\, 0 \ko
}
\]
where the first group vanishes by the third Toda condition in the
hypotheses. The assertion follows.
\end{proof}

This theorem, combined with Theorem \ref{thm:main} immediately yields
Theorem \ref{thm:ext} as a corollary.


\begin{thebibliography}{100}

\bibitem{Anderson79}
Don W. Anderson, \emph{Axiomatic homotopy theory}, in
\emph{Algebraic Topology, Waterloo, 1978}, Lecture Notes in
Mathematics, no.~741, Springer-Verlag, Berlin, 1979.

\bibitem{BokstedtNeeman93}
Marcel B{\"o}kstedt, Amnon Neeman, \emph{Homotopy limits in triangulated
categories}, Compositio Math. \textbf{86} (1993), 209--234.

\bibitem{BousfieldFriedlander78}
Aldridge Bousfield, Eric Friedlander, \emph{Homotopy theory of
$\Gamma$-spaces, spectra, and bisimplicial sets}, Springer Lecture
Notes in Math., vol.~658, Springer, Berlin, 1978, 80-130.

\bibitem{Carlson16}
Kevin Carlson, \emph{An embedding of quasicategories in prederivators},
preprint, arXiv:1612.06980v1 [math.CT] 21 Dec 2016.

\bibitem{Cisinski03}
Denis-Charles Cisinski, \emph{Images directs cohomologiques dans
les cat{\'e}gories de mod{\`e}les}, Ann. math. Blaise Pascal,
\textbf{10} (2003), 195--244 (in French).

\bibitem{Cisinski08}
\bysame, \emph{Propri{\'e}t{\'e}s universelles et
extensions de Kan d{\'e}riv{\'e}es}, Theory Appl. Categ.,
\textbf{20} (2008), no.~17, 605--649 (in French).

\bibitem{CisinskiNeeman08}
Denis-Charles Cisinski, Amnon Neeman, \emph{Additivity for derivator
$K$-theory}, Adv. Math., \textbf{217} (2008), 1381--1475.

\bibitem{DoldPuppe61}
Albrecht Dold, Dieter Puppe, \emph{Homologie nicht--additiver Funktoren.
Anwendungen}, Ann. Inst. Fourier, tome~11 (1961), 201--312 (in German).

\bibitem{Dwyer04}
William Dwyer, \emph{Localizations}, in \emph{Axiomatic, enriched and
motivic homotopy theory}, vol.~131 of NATO Sci. Ser. II Math. Phys. Chem.,
Kluwer Acad. Publ., Dordrecht (2004), 3--28.

\bibitem{Franke96}
Jens Franke, \emph{Uniqueness theorems for certain triangulated
categories possessing an Adams spectral sequence}, preprint,
available at http://www.math.uiuc.edu/K-theory/0139/, 1996.

\bibitem{Groth11}
Moritz Groth, \emph{Derivators, pointed derivators, and stable derivators},
Algebr. Geom. Topol. \textbf{13(1)} (2013), 313--374
[arXiv:1112.3840v2 [math.AT] 13 Feb 2012].

\bibitem{Groth12}
\bysame, \emph{Monoidal derivators and additive derivators},
preprint, arXiv:1203.5071v1 [math.AT] 22 Mar 2012.

\bibitem{Groth16}
\bysame, \emph{Revisiting the canonicity of canonical triangulations},
preprint, arXiv:1602.04846v1 [math.AT] 15 Feb 2016.

\bibitem{GrothPontoShulman14}
Moritz Groth, Kate Ponto, Michael Shulman,
\emph{Mayer-–Vietoris sequences in stable derivators},
Homology, Homotopy Appl., vol.~16, no.~1 (2014) 265-–294.

\bibitem{Grothendieck90}
Alexander Grothendieck, \emph{Les d{\'e}rivateurs}, manuscript ($\sim$1990),
ed. by M. K{\"u}nzer, J. Malgoire, G. Maltsiniotis, available at
http://www.math.jussieu.fr/\textasciitilde maltsin/groth/Derivateurs.html
(in French).

\bibitem{Heller88}
Alex Heller, \emph{Homotopy theories}, Mem. Amer. Math. Soc.,
vol.~71, no.~383 (1988).

\bibitem{Heller97}
\bysame, \emph{Stable homotopy theories and stabilization},
J. Pure Appl. Alg., \textbf{115} (1997), 113--130.

\bibitem{Kapranov88}
Mikhail M. Kapranov, \emph{On the derived categories of coherent sheaves
on some homogenous spaces}, Invent. Math. \textbf{92} (1988), 479--508.

\bibitem{Keller90}
Bernhard Keller, \emph{Chain complexes and stable categories},
Manuscripta Math. \textbf{67} (1990), no.~4, 379--417.

\bibitem{Keller91}
\bysame, \emph{Derived categories and universal problems},
Comm. in Alg. \textbf{19} (1991), 379--417.

\bibitem{Keller07}
\bysame, \emph{Le d{\'e}rivateur triangul{\'e} associ{\'e} {\`a} une
cat{\'e}gorie exacte}, Appendix to \cite{Maltsiniotis07} (in French).

\bibitem{KellerNicolas11}
Bernhard Keller, Pedro Nicol{\`a}s, \emph{Weight structures and simple
DG modules for positive DG algebras}, Int. Math. Res. Not. (2013),
no.~5, 1028--1078 [arXiv:1009.5904v3 [math.RT] 14 Sep 2011].

\bibitem{KellerVossieck87}
Bernhard Keller, Dieter Vossieck,
\emph{Sous les cat{\'e}gories d{\'e}riv{\'e}es}, C.R. Acad. Sci. Paris
\textbf{305}, S{\'e}rie I, (1987), 225--228 (in French).

\bibitem{Lydakys98}
Manos Lydakis, \emph{Simplicial functors and stable homotopy theory}
preprint, available via Hopf archive at
http://hopf.math.purdue.edu/Lydakis/s{\_}functors.pdf (1998)

\bibitem{LuntsOrlov10}
Valery A. Lunts, Dmitri O. Orlov, \emph{Uniqueness of enhancement
for triangulated categories}, J. Amer. Math. Soc., vol.~23, no.~3
(2010), 853--908.

\bibitem{Maltsiniotis07}
Georges Maltsiniotis, \emph{La $K$-th{\'e}orie d'un d{\'e}rivateur
triangul{\'e}}, Categories in Algebra, geometry and mathematical physics,
vol. 431, Contemp. Math., Amer. Math. Soc., Providence, RI, 2007,
341--368 (in French).

\bibitem{MandellMaySchwedeShipley01}
Michael Mandell, Peter May, Stefan Schwede, Brooke Shipley,
\emph{Model categories of diagram spectra}, Proceedings of the London
Mathematical Society, \textbf{82} (2001), 441--512.

\bibitem{Neeman99}
Amnon Neeman, \emph{Triangulated {Categories}}, Annals of Math. Studies,
vol.~148, Princeton University Press, Princeton, NJ, 2001.

\bibitem{PontoShulman16}
Kate Ponto, Michael Shulman, \emph{The linearity of traces in monoidal
categories and bicategories}, Theory Appl. Categ., \textbf{31} (2016),
no.~23, 594--689.

\bibitem{Quillen67}
Daniel~G. Quillen, \emph{Homotopical algebra}, Lecture Notes in
Mathematics, no.~43, Springer-Verlag, Berlin, 1967.

\bibitem{Quillen73}
\bysame, \emph{Higher algebraic ${K}$-theory. {I}}, Algebraic
$K$-theory, I: Higher $K$-theories (Proc. Conf., Battelle Memorial Inst.,
Seattle, Wash., 1972), Lecture Notes in Math., vol. 341, Springer verlag,
1973, 85--147.

\bibitem{Radulescu-Banu09}
Andrei Radulescu-Banu, \emph{Cofibrations in homotopy theory},
preprint, arXiv:math/0610009v4 [math.AT] 8 Feb 2009.

\bibitem{Verdier96}
Jean-Louis Verdier, \emph{Des cat{\'e}gories d{\'e}riv{\'e}es des
cat{\'e}gories ab{\'e}liennes}, Ast{\'e}risque, vol.~239, Soci{\'e}t{\'e}
Math{\'e}matique de France, 1996 (in French).

\end{thebibliography}

\def\cprime{$'$}
\providecommand{\bysame}{\leavevmode\hbox to3em{\hrulefill}\thinspace}
\providecommand{\MR}{\relax\ifhmode\unskip\space\fi MR }
\providecommand{\MRhref}[2]{%
  \href{http://www.ams.org/mathscinet-getitem?mr=#1}{#2}
}
\providecommand{\href}[2]{#2}

\end{document}